\documentclass[a4paper, 11pt, twoside]{amsart}
\usepackage[a4paper, top=3cm, bottom=3cm, right=2.5cm, left=2.5cm]{geometry}

\usepackage[english]{babel}
\usepackage[utf8]{inputenc}
\usepackage[T1]{fontenc}
\usepackage{microtype}
\usepackage[style=british]{csquotes}

\usepackage{amsmath, amsthm, amssymb, bm,bbm, mathtools}
\usepackage{enumerate}

\usepackage{calc}
\usepackage{graphicx}
\usepackage{caption}
\usepackage[labelfont=normalsize]{subcaption}
\graphicspath{{img/}}
\usepackage[dvipsnames]{xcolor}
\definecolor{mygreen}{RGB}{0, 152,15}

\usepackage[%
	unicode,
	colorlinks=true,
	linkcolor=blue,
	urlcolor=blue,
    linktocpage=true,
	citecolor=blue]{hyperref}
\hypersetup{%
	pdfpagemode = UseNone,
	pdfstartview = FitH,
	pdfdisplaydoctitle = true,
	pdfborder = {0 0 0}
}

\emergencystretch=\maxdimen
\makeatletter
\renewcommand\subsection{\@startsection{subsection}{2}%
  \z@{-.5\linespacing\@plus-.7\linespacing}{.5\linespacing}%
  {\normalfont\scshape}}
\makeatother

\newcommand\blfootnote[1]{%
  \begingroup
  \renewcommand\thefootnote{}\footnote{#1}%
  \addtocounter{footnote}{-1}%
  \endgroup
}

\newtheorem{thm}{Theorem}[section]
\newtheorem{lemma}[thm]{Lemma}
\newtheorem{prop}[thm]{Proposition}
\newtheorem{cor}[thm]{Corollary}
\newtheorem{conj}[thm]{Conjecture}
\newtheorem{defn}[thm]{Definition}

\theoremstyle{remark}
\newtheorem{remark}[thm]{Remark}

\numberwithin{equation}{section}

\def\P{\mathbb{P}}
\def\H{\mathbb{H}}
\def\D{\mathbb{D}}
\def\Q{\mathbb{Q}}

\def\R{\mathbb{R}}
\def\C{\mathbb{C}}

\def\Z{\mathbb{Z}}

\DeclareMathOperator{\SLE}{SLE}
\DeclareMathOperator{\CLE}{CLE}
\DeclareMathOperator{\BCLE}{BCLE}

\DeclareMathOperator{\diam}{diam\,}
\DeclareMathOperator{\dist}{dist\,}

\DeclarePairedDelimiter\abs{\lvert}{\rvert}
\DeclarePairedDelimiter\norm{\lVert}{\rVert}

\begin{document}

\title{The fuzzy Potts model in the plane: Scaling limits and arm exponents}
\author{Laurin Köhler-Schindler and Matthis Lehmkuehler}
\date{January 16, 2025}
\begin{abstract}
    We consider a critical  Fortuin-Kasteleyn (FK) percolation with cluster weight $q \in [1,4)$  in the plane, and color its clusters in red (respectively blue) with probability $r \in (0,1)$ (respectively $1-r$), independently of each other. We study the resulting fuzzy Potts model, which corresponds to the critical Ising model in the special case $q=2$ and $r=1/2$. We show that under the assumption that the critical FK percolation converges to a conformally invariant scaling limit (which is known to hold for the FK-Ising model, i.e.\ $q=2$), the obtained coloring converges to variants of Conformal Loop Ensembles constructed, described and studied by Miller, Sheffield and Werner. Based on discrete considerations, we also show that the arm exponents for this coloring in the discrete model are identical to the ones of the continuum model. Using the values of these arm exponents in the continuum, we determine the arm exponents for the fuzzy Potts model.
\end{abstract}
\blfootnote{L.\,KS., M.L.\ -- ETH Zürich, Rämistrasse 101, 8092 Zürich, Switzerland}

\maketitle
\tableofcontents
\thispagestyle{empty}
\newpage

\section{Introduction} \label{sec:intro-fuzzy}

\subsection{Coloring FK percolation clusters}

The subject of investigation of the present paper is the random coloring (with
two colors) of a planar graph (more specifically, here $\Z^2$ or $\Z \times
\Z_+$) that is obtained when one colors independently the clusters of a FK
percolation model as introduced in \cite{fk-model}. The resulting models are
often referred to as fuzzy Potts models \cite{haggstrom-fuzzy-positive,
fuzzy-potts-def,percolation-density-chayes,haggstrom-color-percolation} (also
known as fractional fuzzy Potts models) and belong to a wider class of random
colorings known as `divide and color' models. In this paper, we will study the
case where we start with a critical FK percolation. Note that the construction
involves two layers of randomness, first the sampling of a critical FK
percolation with some cluster weight $q$ and then the coloring of each of the
clusters in red or blue with respective probabilities $r$ and $1-r$.

The name `fuzzy Potts model' is due to the fact that if $r=i/q$ for $i\in
\{1,\dots,q-1\}$ the fuzzy Potts model can also be obtained from a Potts model
with spins $\{1,\dots,q\}$ by coloring each vertex with spin in $\{1,\dots,i\}$
in red and in blue otherwise. This is a consequence of the Edwards-Sokal
coupling and it is exactly the construction of an Ising model from an FK-Ising
percolation.

Before describing our results on these models, it is worth recalling how the
scaling limit and arm exponents of critical models such as percolation or the
Ising model have been derived mathematically: (i) A first crucial step is to
show (via discrete considerations exhibiting conformal invariance features) that
the interfaces in the discrete models converge to SLE type curves in the scaling
limit. (ii) A second independent ingredient is the derivation of exponents
related to this continuum scaling limit (this is performed via SLE computations
and considerations) that in turn enable the computation of `arm exponents' for
this continuum scaling limit. (iii) A final step is to show that these continuum
arm exponents are actually identical to the corresponding arm exponents for the
discrete model (this step involves typically discrete `arm separation' and
`quasi-multiplicativity' arguments for the discrete arm events). See for
instance \cite{werner-perc-notes} for an overview in the case of critical
Bernoulli percolation (see also \cite{smirnov-werner-percolation,smirnov-cardy,
lsw-bm-exponents1, lsw-bm-exponents2, lsw-bm-exponents3, lsw-mono-arm}).

\medspace

For the fuzzy Potts model mentioned above, the situation is more involved, due
to this `two layer randomness', so that this strategy needs to be adapted
appropriately.

(i) We \emph{assume} that the critical FK percolation model converges to its
conjectured Conformal Loop Ensemble (CLE) scaling limit (we will state this
conjecture more precisely as Conjecture \ref{conj:fk-to-cle} in this paper, and
just refer to it as the FK conformal invariance conjecture). This conjecture is
known to hold in the $q=2$ case. Under this assumption, the first question is
whether the scaling limit of the coloring is the coloring of the scaling limit.
This is not obvious, because one has for instance to rule out the possibility
that the coloring of the `small microscopic FK cluster' that disappear in the
scaling limit contribute to creating larger monochromatic islands. This will be
the first type of results that we will derive in the present paper.

(ii) The conclusion of this first step is then that the scaling limit of these
random colorings are described by the colorings of these continuous conformal
loop ensembles. These continuum colorings have been introduced and studied in a
series of works by Miller, Sheffield and Werner on the shoulder of which the
present work will stand.  It is worth recalling some aspects of this series of
papers: In the seminal paper \cite{cle-percolations}, it is shown that
when one starts with a non-simple CLE (i.e., a $\CLE_{\kappa'}$ for $\kappa' \in
(4,8)$ of the type that conjecturally appears as scaling limits of critical FK
percolation models with cluster weight $q\in (0,4)$) and one colors its clusters
independently using some coloring parameter $r$, then the obtained coloring is
described by a variant of $\CLE_{\kappa}$ for $\kappa = 16 / \kappa'$ that they
describe in terms of Boundary Conformal Loop Ensembles (BCLE). The interfaces of
these BCLEs are variants of $\SLE_\kappa$ curves involving another parameter
$\rho$. 
In \cite{cle-percolations}, the explicit relation between $\rho$ and $r$
is not determined, but Miller, Sheffield and Werner have derived this relation
in a subsequent recent paper \cite{msw-non-simple} using arguments based on
Liouville Quantum Gravity. Altogether, it provides a full explicit description
of these continuous fuzzy Potts models in terms of $\kappa'$ and $r$. In
\cite{cle-percolations,msw-simple}, a related setting is also considered when one starts
with a simple CLE instead of a non-simple CLE. This led them to obtain the values of some continuum arm
exponents. The formulas for these exponents are of a rather novel type,
reflecting the fact that the relation between $r$ and $\rho$ seems to be
connected directly to `Liouville Quantum Gravity' considerations. In the present
paper, we will use these facts to derive the (almost) complete set of continuum
arm exponents for these `colored CLE' models.

(iii) Finally, one has to show that the arm exponents for these continuum
`colored CLE' models do match those of the discrete fuzzy Potts models. This
will require again some new inputs and results about these discrete models but is
more subtle than the corresponding problem in the percolation setting since the
fuzzy Potts model does not satisfy a Markov property which allows the existing
proofs for arm separation and quasi-multiplicativity to be adapted. Instead, we
will work via the FK percolation model and use its Markov property to reason
about the fuzzy Potts model. In particular, the work \cite{duminil2021planar}
will be the key tool allowing us to perform the relevant constructions in the
discrete setting.

\medspace

The above ideas will then be combined into the main results of this paper that
we now state. We consider $q \in [1,4)$ and consider the critical FK percolation
model with cluster weight $q$. Throughout
this paper we will then define
\begin{align*}
    \kappa' =4\pi/\arccos(-\sqrt{q}/2) \in (4, 6 ]  \quad{and}\quad \kappa =
    16 / \kappa' \in [8/3, 4)\;.
\end{align*}
The results are then {\em conditional on the conformal invariance conjecture}
for this model, which says that the scaling limit of the cluster boundaries of
this critical FK percolation with cluster weight $q$ is given by a
$\CLE_{\kappa'}$, which is a conformally invariant law on collections of loops.
So far, this conjecture has only been proved in the case $q=2$ (i.e.,
$\kappa'=16/3$) in a series of works \cite{smirnov-ising, kemp-smirnov-fk-bdy,
kemp-smirnov-fk-full} and so in this particular case, the theorems that we will
now state are in fact unconditional.

To state the theorems, we first need to define discrete \emph{arm events}
$A^{s}_\tau(m,n)\subset \{R,B\}^{\Z^2}$ and $A^{+s}_{\tau}(m,n)\subset \{R,B\}^{\Z\times
\Z_+}$ for $1\le m\le n$ where $\tau=\tau_1\cdots\tau_k$ is a finite length word
with letters $\tau_1,\dots,\tau_k\in \{R,B\}$, called \emph{color sequence} in
the following. The superscript $s$ will be explained in Section
\ref{sec:fuzzy-potts}; it should be ignored for the moment. We write $\Lambda_n
= [-n,n]^2\cap \Z^2$ for the box of size $n\ge 1$ with boundary
$\partial\Lambda_n = \Lambda_n \setminus \Lambda_{n-1}$ and $\Lambda_{m,n} =
\Lambda_{n}\setminus \Lambda_{m-1}$ for the annulus from $m\ge 1$ to $n\ge m$.
We also set $\Lambda_{m,n}^+ = \Lambda_{m,n}\cap (\Z\times \Z_+)$.

Let $A^s_{\tau}(m,n)$ be the event which contains a configuration $\sigma \in
\{R,B\}^{\Z^2}$ if and only if there are disjoint nearest-neighbor paths
$\gamma^1,\dots,\gamma^k$ in $\Lambda_{m,n}$ from $\partial \Lambda_m$ to
$\partial \Lambda_n$ which are ordered counterclockwise and are such that
$\sigma$ has color $\tau_i$ along $\gamma^i$ for all $1 \le i\le k$. Similarly,
$A^{+s}_{\tau}(m,n)$ is the event in the upper halfplane which contains a color
configuration $\sigma \in \{R,B\}^{\Z\times \Z_+}$ if and only if there are
disjoint nearest-neighbor paths $\gamma^1,\dots,\gamma^k$ in $\Lambda^+_{m,n} $
from $\partial \Lambda_m$ to $\partial \Lambda_n$ which are ordered
counterclockwise, start with the rightmost arm, and are such that $\sigma$ has color $\tau_i$ along $\gamma^i$
for all $1 \le i\le k$.

We write $p_\tau(q,r,m,n)$ for the probability of the event $A^s_{\tau}(m,n)$
under the fuzzy Potts model on $\Z^2$ with coloring parameter $r$ on a critical
FK percolation with weight $q$ and $p^+_\tau(q,r,m,n)$ for the probability of
the event $A^{+s}_{\tau}(m,n)$ under the fuzzy Potts model on $\Z\times \Z_+$
with coloring parameter $r$ on a critical FK percolation with weight $q$ and
free boundary conditions.

Finally, we need to define the number of `interfaces' that are associated to the arm
events $A^s_{\tau}(m,n)$ and $A^{+s}_{\tau}(m,n)$; it will turn out that
only this interface count affects the arm exponent. We let $\tau_{k+1}:=\tau_1$
and define
\begin{align*}
	I(\tau) = \#\{1\le i\le k\colon \tau_i\neq \tau_{i+1}\}\quad\text{and}\quad
	I^+(\tau) = 1+\#\{1\le i< k\colon \tau_i\neq \tau_{i+1}\}\;.
\end{align*}
Note that $I(\tau)$ will always be an even number. One way of phrasing the fact
that only the interface count is relevant in the theorems below is the statement
that having one red arm yields any finite number of additional consecutive red
arms with only a slightly smaller probability that is insignificant for the arm
exponent. The following theorem provides the values of all arm exponents (except
the monochromatic ones) for the full plane.

\begin{thm}
	\label{thm:main-bulk}
    Let $q \in [1, 4)$ and suppose that the conformal invariance conjecture
    holds for the critical FK percolation model with cluster weight $q$ and
    write $\kappa=4\arccos(-\sqrt{q}/2))/\pi \in [8/3,4)$. Let $r\in (0,1)$ and
    define for $j\ge 1$,
	\begin{align*}
		\alpha_{2j}(r) = \frac{16j^2 - (\kappa-4)^2}{8\kappa}\;.
	\end{align*}
    Then for any color sequence $\tau=\tau_1\cdots\tau_k$ which is not all red
    or all blue, i.e.\ $I(\tau) >0$, we have
	\begin{align*}
		p_\tau(q,r,m,n) = \left(\frac{m}{n}
        \right)^{\alpha_{I(\tau)}(r)+o(1)}
	\end{align*}
	as $n/m\to\infty$ for $k\le m\le n$, where the $o(1)$ term may depend on
	$\tau$.
\end{thm}

For the upper halfplane, we obtain the values of all arm exponents.

\begin{thm}
	\label{thm:main-boundary}
    Let $q \in [1, 4)$ and suppose that the conformal invariance conjecture
    holds for the critical FK percolation model with cluster weight $q$ and write
    $\kappa = 4\arccos(-\sqrt{q}/2)/\pi\in [8/3,4)$. Let $r\in (0,1)$ and define
    for $j\ge 1$,
	\begin{align*}
		\alpha_{2j}^+(r) &= \frac{2j(2j+\kappa/2-2)}{\kappa}\;,\\
		\alpha_{2j-1}^+(r) &= \frac{1}{\kappa}\left(
		2j+\kappa-4-\frac{2}{\pi}\arctan\left(
		\frac{\sin(\pi\kappa/2)}{1+\cos(\pi\kappa/2)-1/r} \right) \right)\\
		&\qquad\cdot \left( 2j+ \frac{\kappa-4}{2}-\frac{2}{\pi}\arctan\left(
		\frac{\sin(\pi\kappa/2)}{1+\cos(\pi\kappa/2)-1/r} \right) \right) \;.
	\end{align*}
    Then for any color sequence $\tau=\tau_1\cdots\tau_k$ starting with $R$ we
    have
	\begin{align*}
		p^+_\tau(q,r,m,n )
		= \left(\frac{m}{n}\right)^{\alpha^+_{I^+(\tau)}(r) +o(1)}
	\end{align*}
	as $n/m\to \infty$ for $k\le m\le n$, where the $o(1)$ term may depend on
	$\tau$. The case where $\tau$ starts with $B$ is obtained by replacing $r$
	with $1-r$.
\end{thm}

In both cases, we stress the following feature: The `even' exponents
$\alpha_{2j}(r)$ (resp.\ $\alpha^+_{2j}(r)$) do not depend on the actual value of
$r \in (0,1)$. This is not at all clear (or even intuitive) from the definition
of the discrete model (except in the special case $\alpha_2^+(r)=1$ which turns
out to be a universal arm exponent) and it appears to  be difficult to prove
this result using only discrete tools (i.e., without relying on scaling limit
conjectures).

It is worth emphasizing that in the special case of the FK-Ising model (i.e.,
$q=2$), the conformal invariance conjecture is known to hold by
\cite{smirnov-ising, kemp-smirnov-fk-bdy, kemp-smirnov-fk-full} (see also
\cite{garban-wu-fk-ising}), so that the previous result is unconditional. In
that case, we have $\kappa =3$ and one can then summarize the previous theorems
in the following (unconditional) formulation.

\begin{thm}
    If $q=2$ then the critical exponents for the fuzzy Potts model with coloring
    parameter $r\in (0,1)$ are
    \begin{gather*}
        \alpha_{2j}(r) = \frac{(2j)^2-1/4}{6}\;,\quad
        \alpha^+_{2j}(r) = \frac{2j(2j-1/2)}{3}\;,\\
        \alpha^+_{2j-1}(r) = \frac{1}{3}\left(2j-1-\frac{2}{\pi}\arctan\left(
                \frac{r}{1-r}\right) \right) \left(2j-\frac{1}{2}-
                \frac{2}{\pi}\arctan\left( \frac{r}{1-r}\right) \right)
    \end{gather*}
    for $j\ge 1$ provided that the color sequence starts with $R$. The case where it
    starts with $B$ is obtained by replacing $r$ by $1-r$.
\end{thm}

\begin{remark}
    The case $r=1/q$ is special. In particular, when $q=2$ (i.e.\ $\kappa=3$) or
    $q=3$ (i.e.\ $\kappa=10/3$), it is related to
    the Potts models. We obtain the formulas
    \begin{align*}
        \alpha^+_{2j-1}(1/q) &= \frac{1}{\kappa}\left(
		2j+\frac{3\kappa}{2}-6\right)(2j+\kappa-4)\;,\\
        \alpha^+_{2j-1}(1-1/q) &= \frac{1}{\kappa}\left(
		2j-\frac{\kappa}{2}\right)(2j-\kappa+2)\;.
    \end{align*}
    The critical Ising model corresponds to cluster weight $q = 2$ (so
    $\kappa=3$) and coloring parameter $r = 1/2$. The exponents then agree with
    the results obtained for the Ising model in \cite{wu-ising-arm}.
\end{remark}

\subsection{Phase diagram of the fuzzy Potts model}

As explained in the previous section, this paper is concerned with the fuzzy
Potts model on a critical FK percolation with cluster weight $q\in [1,4)$. The
following paragraphs elaborate on the whole phase diagram of this model. We
sample $\omega\sim \phi^0_{\Z^2,p,q}$ from the infinite volume measure of FK
percolation with parameter $p\in [0,1]$ and cluster weight $q\ge 1$ (the
condition $q\ge 1$ ensures that the model is positively associated). Thus
$\omega$ is a percolation configuration in $\{0,1\}^{E(\Z^2)}$ where $E(\Z^2)$
denotes the nearest-neighbor edges on $\Z^2$. We write $\mathcal{C}$ for the set
of clusters and $\{0\leftrightarrow x\}$ for the event that $0$ and $x$ are in
the same cluster.

The fuzzy Potts measure $\mu^0_{\Z^2,p,q,r}$ on $\{R,B\}^{\Z^2}$ is obtained by
coloring the vertices of each cluster in $\mathcal{C}$ independently all with
the color red ($R$) with probability $r\in (0,1)$ and all in blue ($B$)
otherwise.

We would like to understand the geometry of the red and blue fuzzy Potts clusters as a
function of the parameters $(p,q,r)$ of the model. To this end, we
first review some major results about the geometry of the underlying
FK percolation model. It turns out that the event $\{0\leftrightarrow x\}$ is
well-suited to describe the phase diagram of FK percolation.
\begin{itemize}
    \item For every cluster weight $q\ge 1$, there is a sharp phase transition
        at the critical point $p_c(q)=\sqrt{q}/(1+\sqrt{q})$: For subcritical $p
        <p_c(q)$, the probability $\phi^0_{\Z^2,p,q}(0\leftrightarrow x)$ decays
        exponentially fast to $0$ as $\norm{x}\to \infty$, whereas for
        supercritical $p>p_c(q)$, it is bounded away from 0 uniformly in the
        point $x$.
        This has been established by \cite{befarra2012self}; see also
        \cite{drt-sharp-fk-tree, dm-q-sharpness, drt-sharp-old}.
    \item At the critical point $p_c(q)$, the geometry depends on the cluster
        weight: For $q>4$, the phase transition is discontinuous (see
        \cite{duminil2016discontinuity, ray2020,duminil2020}) and like in the
        subcritical phase, $\phi^0_{\Z^2,p_c(q),q}(0\leftrightarrow x)$ decays
        \emph{exponentially} fast to $0$ as $\norm{x}\to \infty$ . For $q \in
        [1,4]$, the phase transition is continuous as shown in
        \cite{duminil2017continuity} and  the probability
        $\phi^0_{\Z^2,p,q}(0\leftrightarrow x)$ decays \emph{polynomially} fast to
        $0$ as $\norm{x}\to \infty$. The key tool to study models at the point
        of a continuous phase transition are Russo-Seymour-Welsh (RSW) estimates
        on crossing probabilities. Recently, very powerful RSW estimates have
        been established in \cite{duminil2021planar} for $q\in [1,4)$, and this
        allows for a good understanding of the critical and near-critical
        geometry. Much less is known in the $q=4$ case, which
        is more subtle since the RSW estimates of \cite{duminil2021planar} are
        expected to be wrong in this case.
\end{itemize}
We can now make some simple observations about the fuzzy Potts model with the
parameters $(p,q,r)$: Let $R_\infty$ denote the event on $\{R,B\}^{\Z^2}$ that
there is an infinite red fuzzy Potts cluster.
\begin{itemize}
    \item When the clusters in $\mathcal C$ are exponentially small (i.e.\ for
        $p<p_c(q)$ and for $p=p_c(q)$ when $q>4$), the fuzzy Potts model with
        coloring parameter $r$ should behave similarly to Bernoulli site
        percolation. In particular, one expects there to be a critical point
        $r_c(p,q)\in (0,1)$ such that
        \begin{align*}
            \mu^0_{\Z^2,p,q,r}(R_\infty)>0\quad\text{if and only if}\quad
            r>r_c(p,q)\;.
        \end{align*}
        In \cite{balint-dac-sharp}, this has been established
        for $q=1$ and in fact, there do not appear to be major obstacles to
        generalizing this result to $q\ge 1$ by making use of (more recent)
        sharpness results for FK percolation. Further properties have been
        obtained for $q=1$ in \cite{balint-dac-critical,tassion-dac-confidence,
        tassion2014planarity}.
    \item When there exists a unique infinite cluster in $\mathcal C$ almost
        surely (i.e. for $p>p_c(q)$), we see that for any coloring parameter
        $r\in (0,1)$,
        \begin{align*}
	        \mu^0_{\Z^2,p,q,r}(R_\infty) = r > 0\;.
        \end{align*}
        In particular, there exists a unique infinite cluster in the fuzzy
        Potts model that is either red or blue. \item It remains to describe the
        geometry of the fuzzy Potts model for $p=p_c(q)$ and $q \in [1,4]$. In
        this case, one can use weaker RSW estimates (see
        \cite{duminil2017continuity}) to show that $\mathcal{C}$ almost surely
        contains infinitely many disjoint clusters surrounding the origin. Since
        infinitely many of these clusters will be colored in blue almost surely,
        it follows that for any coloring parameter $r\in (0,1)$,
        \begin{align*}
	        \mu^0_{\Z^2,p,q,r}(R_\infty)=0\;,
        \end{align*}
        and analogously, there is also no infinite blue fuzzy Potts cluster almost surely.
        We emphasize that this behavior differs drastically from the two
        previous cases.
\end{itemize}
What is achieved in this paper is the derivation of more precise results on the
geometry of the fuzzy Potts model for $p=p_c(q)$ and $q \in [1,4)$ as described
in the previous section. The case of $p=p_c(q)$ and $q=4$ is different since one
expects there to be a sharp phase transition at $r_c(p_c(4),4)=1/2$. A
construction of the conjectured near-critical scaling limits in this case
appears in \cite{cle-percolations, lehmkuehler-cle4}. The key difference to the
case $q<4$ is that one expects the scaling limit of the coloring to be different
from the coloring of the scaling limit when one considers $q=4$ (since the
scaling limit of the loop encoding of the model is conjectured to be a simple
CLE).

\subsection{Overview}

Let us now provide an outline of the paper. As mentioned above, there will be
three distinct steps which are needed to determine the exponents in Theorem
\ref{thm:main-bulk} and \ref{thm:main-boundary} and which
will be performed in this paper.
\begin{itemize}
    \item Section \ref{sec:background-fuzzy} contains all the relevant background
        information and references on the discrete objects (FK percolation and
        the fuzzy Potts model) as well as the continuum objects (CLE and BCLE).
        This is also the section within which we are precisely stating the
        conformal invariance conjecture as Conjecture \ref{conj:fk-to-cle}.
    \item In Section \ref{sec:discrete-fuzzy}, we develop arm separation and
        quasi-multiplicativity tools for the discrete fuzzy Potts model (which
        are instrumental for deducing the discrete from the continuum
        exponents). This section relies heavily on the recent paper
        \cite{duminil2021planar} which is valid for $q\in [1,4)$. We however
        expect that quasi-multiplicativity also holds when $q=4$ case, even
        though our proof techniques do not work in this case. Many techniques in
        this section are motivated by the seminal work \cite{Kesten1987}.
    \item In Section \ref{sec:convergence}, we show that under the conformal
        invariance assumption for FK percolation with parameter $q\in [1,4)$, the
        scaling limit of the (discrete) fuzzy Potts cluster boundaries is given
        by the continuum fuzzy Potts cluster boundaries as constructed in the
        CLE percolation paper by Miller, Sheffield and Werner
        \cite{cle-percolations}. We defer two technical lemmas on the topology
        involved to Appendix \ref{sec:app-loops}. This is the step where the
        situation is significantly different when $q=4$ instead of $q\in [1,4)$.
        We also obtain a new derivation of the scaling limit of the
        Ising model loops (see \cite{benoist-hongler-cle3} for the existing
        proof).
    \item In Section \ref{sec:continuum-exp}, we work on the continuum side, and
        we compute the critical exponents in the setting of the continuum fuzzy
        Potts model. More specifically, now that the interfaces are described in
        terms of the SLE variants from the papers by Miller, Sheffield
        and Werner, we need to compute the corresponding exponents for these
        processes. This will bear many similarities (and will rely on) the SLE
        computations by  Wu in \cite{wu-ising-arm}. Some technical results on
        conformal transformations relevant to this section are deferred to
        Appendix \ref{sec:app-conformal}.
    \item Finally, we will combine the above ingredients in Section
        \ref{sec:combine-results} to establish the main results, namely
        Theorem \ref{thm:main-bulk} and \ref{thm:main-boundary}.
\end{itemize}

\medspace

{\bf Notation.} Throughout, if $f,g\colon X\to [0,\infty]$ are functions on some
space $X$, we write $f(x)\lesssim g(x)$ if there is a constant $C\in (0,\infty)$
such that $f(x)\le Cg(x)$ for all $x\in X$. We write $f(x) \asymp g(x)$ if
$f(x)\lesssim g(x)$ and $g(x)\lesssim f(x)$ hold. We also let
$\Z_+=\{0,1,2,\dots\}$ and write $\log$ for the logarithm to the base $2$.
Finally, we write $B_r(z_0)=z_0+r\D$ and $(\!(a,b)\!)$ for the open
counterclockwise boundary arc from $a\in \partial\D$ to $b\in \partial\D$ along
$\partial\D$.

\medspace

{\bf Acknowledgments.} L.\,KS. has received funding from the European Research
Council (ERC) under the European Union's Horizon 2020 research and innovation
program (grant 851565). M.\,L. was supported by grant 175505 of the Swiss
National Science Foundation. Both authors are part of NCCR SwissMAP. We would
like to thank Wendelin Werner for proposing this project and we are grateful to
Vincent Tassion and Wendelin Werner for many insightful inputs and discussions.
We also thank Hugo Vanneuville for valuable comments on an earlier version of this paper.

\section{Background}
\label{sec:background-fuzzy}

\subsection{FK percolation}
\label{sec:FK-percolation}

In this subsection, we formally introduce critical FK percolation and we state
all preliminary results that will be used throughout the paper. We direct the
reader to \cite{grimmett-fk} for a broader introduction and to
\cite{duminil-fk-notes} for an exposition of recent results. Besides standard
properties of FK percolation, our work mostly relies on the strong crossing
estimates which were recently established in \cite{duminil2021planar}.

We consider the square lattice $(\Z^2,E(\Z^2))$ which is the graph with vertex
set $\Z^2$ and edges between nearest neighbors. Abusing notation slightly, we
often refer to the graph itself as $\Z^2$.

Let $G=(V,E)$ be a finite subgraph of $\Z^2$. We define its vertex boundary $\partial V = \{ v \in V : \deg_G(v)< \deg_{\Z^2}(v)=4\}$, where the  $\deg_G(v)=\abs{\{w \in V : vw \in E\}}$ is the number of neighbors of $v$ in $G$.
An element $\omega \in \{0,1\}^E$ encodes
a subgraph of $G$ with vertex set $V$ and edge set $o(\omega):=\{e\in E:
\omega_e = 1\}$, where we recall that an edge $e$ is called open (resp.\ closed)
if $\omega_e = 1$ (resp.\ $\omega_e=0$). 
We now consider a partition $\xi$ of
$\partial V$, called \emph{boundary condition}, and
we denote by $\omega^\xi$ the graph obtained from $\omega$ by adding additional
edges between vertices belonging to the same partition element. Let
$\mathcal{C}(\omega^\xi)$ denote the set of clusters of this graph. FK
percolation on $G=(V,E)$ with boundary condition $\xi$, edge weight $p\in[0,1]$
and cluster weight $q>0$ is the measure
\begin{align*}
	\phi_{G,p,q}^\xi (\omega) := \frac{1}{Z_{G,p,q}^\xi}
    p^{\abs{o(\omega)}} (1-p)^{\abs{E \setminus o(\omega)}}
    q^{\abs{\mathcal{C}(\omega^\xi)}}
	\quad\text{on}\; \{0,1\}^{E}\;,
\end{align*}
where $Z_{G,p,q}^\xi$ is a normalizing constant, called \emph{partition
function}.

The product ordering provides a partial order on $\{0,1\}^E$. A function
$f:\{0,1\}^E \to \R$ is said to be \emph{increasing} if $\omega \le \omega'$
implies $f(\omega) \le f(\omega')$, and an event $A$ is \emph{increasing} if its
indicator function $1_A$ is increasing. For cluster weights $q \ge 1$, FK
percolation is positively associated, i.e.\ for any two increasing functions
$f,g: \{0,1\}^E \to \R$,
\begin{equation*}
	\phi_{G,p,q}^\xi (f \cdot g) \ge \phi_{G,p,q}^\xi (f)
    \cdot \phi_{G,p,q}^\xi (g)\;.
\end{equation*}
This property is also referred to as the FKG inequality. Actually, it does not
hold for $q\in (0,1)$ and much less is known about the model in this case.
Throughout the paper, we will therefore only consider cluster weights $q \ge 1$.

We will now discuss boundary conditions in more detail. For two partitions
$\xi, \xi'$ of $\partial V$, we write $\xi \le \xi'$ if $\xi$ is a finer
partition than $\xi'$, i.e.\ every element in $\xi$ is a subset of an element in
$\xi'$. We refer to the finest partition, where $\partial V$ is partitioned into
singletons, as \emph{free} boundary conditions, denoted by $\phi_{G,p,q}^0$, and
to the coarsest partition, which consists only of the element $\partial V$, as
\emph{wired} boundary conditions, denoted by $\phi_{G,p,q}^1$. Boundary
conditions can easily be compared for increasing events $A \subset \{0,1\}^E$:
For any two boundary conditions
$\xi \le \xi'$,
\begin{equation*}
	 \phi_{G,p,q}^\xi (A) \le \phi_{G,p,q}^{\xi'}(A)\;.
\end{equation*}

The domain Markov property is another fundamental tool to study FK percolation.
Let $G'=(V',E')$ be a subgraph of the finite graph $G=(V,E)$ and $\xi$ be a boundary condition on $G$.
Given an event $A$ that is measurable with respect to the status of edges in
$E'$, it says that for
any $\psi \in \{0,1\}^{E \setminus E'}$,
\begin{equation*}
	\phi_{G,p,q}^\xi(A \mid \omega_e = \psi_e, \forall e \in
    E\setminus E') = \phi_{G',p,q}^{\xi'}(A)\;.
\end{equation*}
where $\xi'=\xi'(\psi,\xi)$ denotes the partition of $\partial V'$ that is
induced by $\xi$ and $\psi$, i.e.\ two vertices $x',y' \in \partial V'$ belong
to the same element of $\xi'$ if they can be connected to the same element of
$\xi$ or directly to each other using open edges in $\psi$.

Throughout the paper, we naturally embed the square lattice $(\Z^2,E(\Z^2))$ in the plane $\R^2$ by viewing edges as closed line segments.
Whenever $S\subset \R^2$ is closed and bounded, by definition, the finite subgraph
	\emph{induced} by $S$ is the subgraph of $\Z^2$ with vertex set $S\cap\Z^2$ and edge set $E(S)$ containing those edges that are completely contained in $S$. We often refer to the
	subgraph itself by $S$ when no confusion can arise, and define its vertex boundary $\partial S$ as before. 
	If the subgraph induced by $S$ is  contained in
	the upper halfplane, we also define its vertex boundary $\partial_+ S$ with
	respect to  $\Z \times \Z_+$, consisting of those vertices $v$ in $S$ with incomplete degree, i.e.\ $\deg_S(v) < \deg_{\Z\times\Z^+}(v) \in \{3,4\}$.

For $r,s >0$, we define the box $\Lambda_r = [-r,r]^2$ and the annulus
$\Lambda_{r,s}= [-s,s]^2 \setminus (-r,r)^2$. We also define their intersections
with the upper halfplane as $\Lambda_r^+ = [-r,r] \times [0,r]$ and
$\Lambda_{r,s}^+ = [-s,s]\times [0,s] \setminus (-r,r) \times [0,r)$. As
mentioned before, we will often use the same notation when referring to their
induced subgraphs. In particular, our notation is consistent with the
definitions given in Section \ref{sec:intro-fuzzy}.

To define FK percolation on the infinite graph $\mathbb{Z}^2$, one can consider
the sequence $(\Lambda_n)_{n\ge 1}$ and then take the weak
limit of the measures $\phi_{\Lambda_n,p,q}^\xi$ along this sequence for $\xi=0$
and $\xi=1$, respectively. While the limiting measures $\phi_{\Z^2,p,q}^0$ and
$\phi_{\Z^2,p,q}^1$ might a priori be different, it is straightforward to see
that they also satisfy the FKG inequality. In the same way, one can define the
FK percolation measure $\phi_{\Z\times\Z_+,p,q}^0$ on the upper halfplane
$\Z\times\Z_+$ by considering the sequence $(\Lambda_n^+)_{n\ge 1}$.

As explained before, this paper is concerned with critical FK percolation.
Therefore, we fix the edge weight to be $p_c(q)=\sqrt{q}/(1+\sqrt{q})$ from now
on and we drop it from our notation. It was proven in
\cite{duminil2017continuity} that for $q \in [1,4]$, the two extremal measures
$\phi_{\Z^2,q}^0$ and $\phi_{\Z^2,q}^1$ are in fact the same, and so we will
also drop the dependence on boundary conditions and simply write $\phi_{\Z^2,q}$
for the (critical) FK percolation measure on $\Z^2$ with cluster weight $q\in
[1,4]$.

A consequence of the crossing estimates in rectangles that were established in
\cite{duminil2017continuity} is the following mixing statement which says that
the boundary conditions do not affect the values of the probabilities
significantly (see \cite[Section 1.3.1]{duminil2017continuity}).

\begin{cor}\label{cor:mixing}
    Let $q \in [1,4]$. For all $c>0$, there exists a constant $C>0$ such that
    for any $n\ge 1$ with $\Lambda_{(1+c)n} \subset G$ and for any event $A$
    measurable with respect to the edges in $\Lambda_n$,
	\begin{equation*}
		C^{-1} \cdot \phi_{\Z^2,q}(A) \le \phi_{G,q}^\xi(A) \le
        C \cdot \phi_{\Z^2,q}(A)\;,
	\end{equation*}
    uniformly in the boundary condition $\xi$.
\end{cor}

The situation is similar in the case of the halfplane measure $\phi^0_{\Z\times
\Z_+,q}$. The following corollary can be proven analogously to the full plane case. This makes use of stronger crossing estimates as obtained in
\cite{duminil2021planar} and we therefore exclude the case $q=4$.

\begin{cor}\label{cor:mixing-half}
    Let $q\in [1,4)$. For all $c>0$, there exists $C>0$ such that for all $n\ge
    1$ with $\Lambda_{(1+c)n}^+\subset G\subset \Z\times \Z_+$ and for any event
    $A$ measurable with respect to the edges in $\Lambda^+_n$,
	\begin{equation*}
		C^{-1} \cdot \phi^0_{\Z\times \Z_+,q}(A) \le \phi_{G,q}^\xi(A) \le
        C \cdot \phi^0_{\Z\times \Z_+,q}(A)\;,
	\end{equation*}
    uniformly in boundary conditions $\xi$ for which $\{(i,0)\}\in \xi$ when
	$\abs{i}\le n$.
\end{cor}

Let us now present the stronger crossing estimates from \cite{duminil2021planar}
which allow for a precise understanding of the behavior at criticality. We
introduce a few notions which will also appear throughout the paper.

A path $\gamma=(\gamma_i)_{i=0}^n$ is a finite sequence of nearest-neighbor
vertices and it is called simple if the vertices are distinct. A loop is a path
$\lambda$ with $\lambda_n=\lambda_0$ and such that $(\lambda_i)_{i=1}^{n}$ is
simple. A path $\gamma=(\gamma_i)_{i=0}^n$ is said to be \emph{open} in a
percolation configuration $\omega$ if $\omega_{\gamma_{i-1}\gamma_i}=1$ for all
$1 \le  i \le n$. A discrete domain is a finite subgraph $\mathcal{D}=(V,E)$ of $\Z^2$
which is enclosed by a loop (all vertices and edges on the loop belong to the
domain, see also \cite{duminil2021planar}) and we write $\partial \mathcal{D}$
for the vertices on the boundary loop. If $a,b,c,d\in \partial \mathcal{D}$ are
distinct and ordered counterclockwise, we call $(\mathcal{D},a,b,c,d)$ a
\emph{discrete quad}. The boundary points define four boundary arcs $(ab)$,
$(bc)$, $(cd)$ and $(da)$ (which are subsets of $\partial \mathcal{D}$ and
contain their endpoints by convention).

Later, we will allow for a slightly more general type of domain, which we refer
to as an approximate discrete domain. By definition, it is a connected subgraph
$\mathcal{D}'=(V',E')$ obtained from a discrete domain $\mathcal{D}=(V,E)$ by
adding a set of edges, each of which has exactly one endpoint in $V$. Note that
$\mathcal{D}$ can be recovered from $\mathcal{D}'$. Moreover, if we consider
$a,b,c,d\in \partial \mathcal{D}$ as above then we define the boundary segments
by $(ab)'=\{v\in \partial V'\colon \min_{w\in (ab)} d_1(v,w)\le 1\},$ where $d_1$ denotes the graph metric
on $\Z^2$.

Let $\mathcal{D}'=(V',E')$ be as above.
We define the crossing event $\{(ab)' \xleftrightarrow{\mathcal{D}'} (cd)'\}$ to
be the event that there exists an open path from $(ab)'$ to $(cd)'$ in
$\mathcal{D}'$. The domain is then said to be \emph{crossed} and such open paths
are  called \emph{crossings}.

If the discrete domain $\mathcal{D}$ is defined in terms of a loop $\lambda$ (as
above), we can view $\lambda$ as a piecewise linear Jordan curve surrounding a
domain $\lambda^o$ in $\R^2$. By the Riemann mapping theorem, there exists
$\ell>0$ and a unique conformal transformation from $\lambda^o$ to $(0,1)\times
(0,\ell)$ such that its continuous extension maps $a, b, c, d$ to the corners of
$[0,1]\times[0,\ell]$ in counterclockwise order with $a$ being mapped to
$(0,0)$. This unique value $\ell$, denoted by $\ell_{\mathcal D}((ab),(cd))$, is
referred to as the \emph{extremal distance} from $(ab)$ to $(cd)$ in
$\mathcal{D}$.

\begin{thm}[\cite{duminil2021planar}]
    \label{thm:strongRSW}
    Let $q \in [1,4)$. For every $L>0$, there exists $\epsilon =
    \epsilon(L)\in (0,1)$ such that if $(\mathcal D, a, b, c,
    d)$ and $\mathcal{D}'$ are as above, then for any boundary condition $\xi$,
	\begin{itemize}
        \item if $\ell_{\mathcal D}((ab),(cd)) \le L$, then $\phi_{\mathcal
            D',q}^\xi ((ab)' \xleftrightarrow{\mathcal D'} (cd)') \ge \epsilon$,
        \item if $\ell_{\mathcal D}((ab),(cd)) \ge L^{-1}$, then $\phi_{\mathcal
            D',q}^\xi ((ab)' \xleftrightarrow{\mathcal D'} (cd)') \le
            1- \epsilon$.
	\end{itemize}
\end{thm}
These crossing estimates are the main result of \cite{duminil2021planar} and
they will serve as the key tool to study the critical fuzzy Potts model in
Sections \ref{sec:discrete-fuzzy} and \ref{sec:convergence}. Note that in
\cite{duminil2021planar} they are only stated in the case of discrete domains
but deducing the case for approximate discrete domains is straightforward. While
boundary conditions are irrelevant for Bernoulli percolation ($q=1$), it is
difficult to control the effect of boundary conditions for $q \in (1,4)$ and
these estimates are a major improvement compared with
\cite{duminil2017continuity} since they include crossings that touch unfavorable
boundary conditions in arbitrary domains.

We will now state some applications of Theorem \ref{thm:strongRSW} that were
presented (among others) in Sections 6 and 7 of \cite{duminil2021planar}. Many
of the arguments go back to \cite{Kesten1987} (see also \cite{Nolin2008}) in the
case of Bernoulli percolation ($q=1$) and have been extended to the case of
FK-Ising ($q=2$) in \cite{Chelkak2016}.

We begin with the universal arm exponents for FK percolation. We consider the
dual lattice $(\Z^2)^\star = (1/2,1/2) + \mathbb Z^2$ with edge set
$E^\star((\Z^2)^\star)$. It can naturally be embedded in the plane $\R^2$
together with the (primal) lattice $\Z^2$ by drawing its edges as straight line
segments. Every dual edge $e^\star \in E^\star((\Z^2)^\star)$ intersects a
unique (primal) edge $e \in E(\Z^2)$, and we define the dual percolation
configuration $\omega^\star$ by $\omega^\star_{e^\star}=1-\omega_e$. We say that
an edge $e^\star \in E^\star((\Z^2)^\star)$ is dual open if
$\omega^\star_{e^\star}=1$.  A type sequence $\tau=\tau_1\cdots\tau_k$ is a
finite length word with letters $\tau_1,\dots,\tau_k\in \{0,1\}$. An open simple
path in the annulus $\Lambda_{m,n}$ from the inner boundary $\partial \Lambda_m$
to the outer boundary $\partial \Lambda_m$  is called an \emph{arm of type 1}.
Similarly, an \emph{arm of type 0} from $\partial \Lambda_m$ to $\partial
\Lambda_m$ in $\Lambda_{m,n}$ denotes a dual open simple path on
$\Lambda_{m,n}^\star$ (the subgraph of $(\Z^2)^\star$ induced by the edges dual
to the edges of $\Lambda_{m,n}$) from
the inner boundary to the outer boundary.

\begin{defn}
    Let $1 \le m \le n$ and let $\tau$ be a type sequence. The arm event
    $A_\tau(m,n)$ denotes the existence of $\abs{\tau}$
    counterclockwise-ordered, disjoint arms
    $\gamma^1,\ldots,\gamma^{\abs{\tau}}$ from $\partial \Lambda_m$ to $\partial
    \Lambda_n$ in $\Lambda_{m,n}$ such that $\gamma^i$ has type $\tau_i$ for all
    $1\le i \le \abs{\tau}$. Similarly, we define the arm event $A^+_\tau(m,n)$
    by requiring that the arms stay in the upper halfplane $\Z\times\Z_+$ and
    that $\gamma^1,\dots,\gamma^{\abs{\tau}}$ starts with the rightmost arm.
\end{defn}

Note that the condition $m \ge \abs{\tau}$ is sufficient to guarantee that
$A^+_\tau(m,n)$ and $A_\tau(m,n)$ are non-empty events.

We will make use of the following results on critical exponents for FK
percolation. They appear in \cite[Proposition 6.6, Corollary
6.7]{duminil2021planar} and are stated for the measure $\phi_{\Z^2,q}$ there.
However, using the inclusions $A^+_{010}(m,n)\subset A^+_{010}(m,n/2)$ and
$A_{100100}(m,n)\subset A_{100100}(m,n/2)$ for $m\le n/2$ together with the mixing
property, we obtain the result as stated below.

\begin{cor}[\cite{duminil2021planar}]\label{cor:six-arms-FK-with-bc}
    Let $q \in [1,4)$. There exists a constant $\beta_1=\beta_1(q)>0$ such that
    for all $n \ge m \ge 1$,
	\begin{align*}
        \phi_{\Lambda_n,q}^\xi\left(A^+_{010}(m,n)\right) &\lesssim
        \left(\frac{m}{n}\right)^{1+\beta_1}\;,\\
        \phi_{\Lambda_n,q}^\xi\left(A_{100100}(m,n)\right) &\lesssim
        \left(\frac{m}{n}\right)^{2+\beta_1}\;,
	\end{align*}
	where the bounds are uniform in $\xi$ and in $m,n$.
\end{cor}

Another application of Theorem \ref{thm:strongRSW} will be needed in Section
\ref{sec:convergence} to prove the convergence of the (discrete) fuzzy Potts
cluster boundaries. A sequence of distinct clusters $C_1,\ldots,C_k$ such that
for every $1\le i< k$, the clusters $C_i$ and $C_{i+1}$ have graph distance 1,
is called a \emph{cluster chain} of length $k\ge 1$. For $\ell >0$ and $n \ge
1$, denote by $\Lambda_n^\ell$ the subgraph induced by $[-\ell n, \ell
n]\times[-n,n]$. The following theorem (see \cite[Theorem
7.5]{duminil2021planar}) states that with probability close to 1, the box
$\Lambda_n^\ell$ is crossed from left to right by a cluster chain consisting of
macroscopic clusters. The result in \cite{duminil2021planar} is stated with
respect to $\phi_{\Z^2,q}$ but the proof goes through with more general boundary
conditions (away from the box $\Lambda_n^\ell$ under consideration).

\begin{thm}[\cite{duminil2021planar}]\label{thm:large-cluster-chains}
    Let $q \in [1,4)$. For $\alpha > 0$ let $S(\alpha,n,\ell)$ be the event that
    there exists a cluster chain with each cluster having diameter at least
    $\alpha n$, that crosses $\Lambda_n^\ell$ from the left side to the right
    side. Then, for every $\epsilon ,\ell>0$ and $c>0$, there exists $\alpha >0$
    such that for every $n \ge 1$,
	\begin{equation*}
		\phi^\xi_{\Lambda_{(1+c)n}^{\ell},q}(S(\alpha,n,\ell))
        \ge 1-\epsilon \;.
	\end{equation*}
\end{thm}

\subsection{Fuzzy Potts model}
\label{sec:fuzzy-potts}

Having introduced critical FK percolation in the previous subsection, we are now
in the position to define the fuzzy Potts model. Let $G=(V,E)$ be a finite subgraph of
$\Z^2$. We consider the space $\{0,1\}^E \times \{R,B\}^V$ consisting
of elements $(\omega,\sigma)$ with $\omega$ representing a (bond) percolation
configuration and $\sigma$ representing a (site) coloring.  We say that a vertex
$v \in V$ is \emph{red} if $\sigma_v=R$ and \emph{blue} if $\sigma_v=B$. Fix a
cluster weight $q \in [1,4)$, a coloring parameter $r \in [0,1]$, and a boundary
condition $\xi$. The probability measure $P_{G,q,r}^\xi$ on $\{0,1\}^E \times
\{R,B\}^V$ is constructed in two steps:
\begin{enumerate}[(i)]
    \item Sample $\omega \in \{0,1\}^E$ according to the critical FK
        percolation measure $\phi_{G,q}^\xi$.
    \item Color every cluster $C \in \mathcal{C}(\omega^\xi)$ in red with
        probability $r$ and in blue with probability $1-r$, independently of the
        other clusters. Here, coloring a cluster $C$ in red (resp.\ blue) means
        to set $\sigma_v=R$ (resp.\ $\sigma_v=B$) for all vertices $v \in C$.
\end{enumerate}
In the above, recall that $\mathcal{C}(\omega^\xi)$ denotes the set of clusters
of the graph $\omega^\xi$ that is obtained from $\omega$ by adding additional
edges between vertices belonging to the same partition element of $\xi$.

The second marginal of $P^\xi_{G,q,r}$, denoted by $\mu_{G,q,r}^\xi$, is called
the fuzzy Potts measure.
Similarly, we obtain $P_{\Z^2,q,r}$ (resp.\ $P_{\Z\times \Z_+,q,r}^0$) and its
second marginal $\mu_{\Z^2,q,r}$ (resp.\ $\mu_{\Z\times\Z_+,r,q}^0$) from
$\phi_{\Z^2,q}$ (resp.\ $\smash{\phi^0_{\Z\times \Z_+,q}}$) using the analogous
construction.

To describe the geometry of the fuzzy Potts model, we again want to define the
notions of red (resp.\ blue) arms and clusters. Compared with FK percolation,
which is a measure on the status of edges of $\Z^2$, the fuzzy Potts model is a
measure on the color of vertices of $\Z^2$ and therefore, it is a priori not
clear which adjacency relation should be chosen to define arms and clusters.
Recall that we have naturally embedded $\Z^2$ in the plane $\R^2$ and denote by
the $d_1$ (resp.\ $d_\infty$) the metric induced by the 1-norm (resp.\ by the
$\infty$-norm) on $\R^2$. We call a finite sequence of vertices $\gamma =
(\gamma_i)_{i=0}^n$ a \emph{strong path} (resp.\ \emph{weak path}) if for all $1
\le i \le n$, $d_1(\gamma_{i-1},\gamma_i) = 1$ (resp.\
$d_\infty(\gamma_{i-1},\gamma_i) = 1$). Note that the notion of strong path
coincides with a nearest-neighbor path on the square lattice, whereas a weak
path is allowed to move along diagonals.

The reason why we are also introducing the notion of weak paths here (rather
than only working with strong ones) is the following duality statement (which
for conceptual clarity we are phrasing in the particular case of box crossings):
There is a weak blue path from top to bottom in $\Lambda_n$ if and only if there
is no strong red path from left to right in $\Lambda_n$. The analogous statement
holds when `red' and `blue' are interchanged. In Section
\ref{sec:discrete-fuzzy}, it will be convenient to pick the convention that two
points are in the same fuzzy Potts cluster if they are connected by a strong red
path or if they are connected by a weak blue path. Picking one particular
convention also appears in Section \ref{sec:loop-encoding} when we are encoding
fuzzy Potts configurations as collections of loops. As we will see in Section
\ref{sec:convergence} which convention we pick does not affect the scaling limit
and therefore also does not affect the arm exponents. This is essentially due to
the fact that there are no
pivotals in the scaling limit.

When studying arm events, we nonetheless need to specify whether arms correspond
to strong or weak paths. In the definition below, a color sequence
$\tau=\tau_1\cdots\tau_k$ is a finite length word with letters
$\tau_1,\dots,\tau_k\in \{R,B\}$.

\begin{defn}\label{def:arm-event-fuzzy-Potts}
    Let $1 \le m \le n$ and let $\tau$ be a color sequence. The arm event
    $A^{s}_\tau(m,n)$ denotes the existence of $\abs{\tau}$
    counterclockwise-ordered, disjoint strong paths
    $\gamma^1,\ldots,\gamma^{\abs{\tau}}$ from $\partial \Lambda_m$ to $\partial
    \Lambda_n$ in $\Lambda_{m,n}$ such that $\gamma^i$ has color $\tau_i$ for
    all $1\le i \le \abs{\tau}$. The arm event $A_\tau(m,n)$ is defined in the
    same way except that $\gamma^i$ is allowed to be a weak path if $\tau_i=B$.

    Similarly, we define the arm events $A^+_\tau(m,n)$ and $A^{+s}_\tau(m,n)$
    by requiring that the arms stay in the upper halfplane $\Z\times\Z_+$ and
    that $\gamma^1,\dots,\gamma^{\abs{\tau}}$ starts with the rightmost arm.
\end{defn}

Hence, the superscript $s$ indicates that all arms correspond to strong paths.
We do not introduce the arm event in which all arms correspond to weak paths
since in that case red and blue weak paths can cross each other. As in the case
of arm events for FK percolation, we remark that assuming $m \ge \abs{\tau}$ is
sufficient to guarantee that $A_\tau(m,n)$, $A^+_\tau(m,n)$, $A^{s}_\tau(m,n)$
and $A^{+s}_\tau(m,n)$ are non-empty.

Let us recall the definition of the number of `interfaces' associated to arm
events: Setting $\tau_{k+1}:=\tau_1$, it
was defined by
\begin{align*}
	I(\tau) = \#\{1\le i\le k\colon \tau_i\neq \tau_{i+1}\}\quad\text{and}\quad
	I^+(\tau) = 1+ \#\{1\le i< k\colon \tau_i\neq \tau_{i+1}\}\;.
\end{align*}
A color sequence $\tau$ is called \emph{alternating} if $\abs{\tau}=1$ or
$\abs{\tau} = I(\tau)$, and \emph{alternating for the halfplane} if $I^+(\tau)=
\abs{\tau}$. In both cases, it means that no color occurs twice subsequently
(with the only difference that the last and the first letter of $\tau$ are not
viewed as subsequent in the halfplane). For example, the color sequence $\tau =
RBR$ is not alternating but it is alternating for the halfplane.

\subsection{Loop (ensemble) topologies}

Since we will investigate the convergence of collections of loops in planar
domains, we will need to define the sense of convergence carefully and we use
the setup from \cite{benoist-hongler-cle3} (see this work for more details).

Let $C^s(\partial \D,\C)$ be the closure in the uniform topology of all
injective elements of $C(\partial \D, \C)$, i.e. closed curves that may
touch but not cross themselves. When we talk of loops in $\C$ in this work,
we refer to elements of $\mathcal{C} := C^s(\partial\D,\C)/\sim$, where the
equivalence relation $\sim$ is defined by $\eta\sim \eta\circ \phi$ whenever
$\phi$ is a homeomorphism from $\partial \D$ to itself. We define a metric
$d_{\mathcal{C}}$ on $\mathcal{C}$ by
\begin{align*}
	d_{\mathcal{C}}([\eta],[\eta']) = \inf_\phi \| \eta - \eta'\circ
		\phi\|_\infty
\end{align*}
where $[\eta]$ denotes the equivalence class associated to $\eta$ and the
infimum is taken over all homeomorphisms $\phi$ from $\partial \D$ to itself.
This turns $(\mathcal{C}, d_{\mathcal{C}})$ into a Polish space. By a small
abuse of notation, we will write $\eta$ instead of $[\eta]$ throughout.

We write $\diam(\eta)$ for the diameter of the image of $\eta$ and we also let
$\eta^o$ denote all the points in $\C$ around which are surrounded by $\eta$ or
more precisely all the points around which $\eta$ has winding $\pm 1$. We say
that $\eta$ surrounds $\eta'$ if $\eta'$ is contained in the closure of
$\eta^o$.

We will also want to consider collections of loops: Let $\mathcal{L}$ be the set
of countable subsets $\Gamma$ of $\mathcal{C}$ with the property that
$\Gamma_\epsilon := \{ \eta\in \Gamma\colon \diam(\eta)>\epsilon \}$ is finite
for all $\epsilon>0$ and the property that $\Gamma=\cup_{\epsilon > 0}
\Gamma_\epsilon$ (this is called local finiteness). We define
\begin{align*}
	d_{\mathcal{L}}(\Gamma, \Gamma') = \inf_{G,G',\pi} \sup_{\eta\in G}
		d_{\mathcal{C}}(\eta, \pi(\eta)) \vee
		\sup_{\eta\in \Gamma\setminus G} \diam(\eta) \vee
		\sup_{\eta'\in \Gamma'\setminus G'} \diam(\eta')
\end{align*}
where the infimum is taken over all $G\subset \Gamma$, $G'\subset \Gamma'$ and
all bijections $\pi\colon G\to G'$. Again, this definition turns $\mathcal{L}$ with the topology induced by $d_{\mathcal{L}}$ into a Polish space.

When $\Gamma\in \mathcal{L}$ has the property that for any two loops $\eta,\eta'\in
\Gamma$ with $\eta^o\cap (\eta')^o\neq\emptyset$, $\eta$ surrounds $\eta'$ or
$\eta'$ surrounds $\eta$, we can associate a nesting level $N_{\Gamma,\eta}$
to each loop $\eta\in \Gamma$. Indeed, we let $N_{\Gamma,\eta}=n+1$ where $n$ is
the number of distinct loops surrounding $\eta$ (not counting $\eta$ itself). In
particular, the outermost loops have nesting level $1$.

We will also introduce the space $\mathcal{C}'$ of curves. We let
$C^s([0,1],\C)$ be the closure in the uniform topology of all injective elements
of $C([0,1], \C)$. We call the space $\mathcal{C}'= C^s([0,1],\C)/\sim$ the
space of curves where we identify $\gamma\sim \gamma\circ \phi$ whenever
$\phi\colon [0,1]\to [0,1]$ is an increasing homeomorphism. We let
\begin{align*}
	d_{\mathcal{C}'}([\gamma],[\gamma']) = \inf_\phi \| \gamma - \gamma'\circ
	\phi\|_\infty
\end{align*}
where the infimum is taken over all increasing homeomorphisms $\phi$ from
$[0,1]$ to itself. Also in this case we obtain a Polish space
$(\mathcal{C}',d_{\mathcal{C}'})$. As in the case of loops we will write
$\gamma$ instead of $[\gamma]$ to avoid visual clutter.

When working with curves and collections of loops, it will always be understood
that we fix some arbitrary parametrization for them.

\subsection{Encoding percolation models as loop configurations}
\label{sec:loop-encoding}

In this section, we explain how both a bond and a site percolation can be
encoded as a collection of loops (in the sense of the previous section). Since
we will use these notions to talk about scaling limits in this paper, we will
consider subgraphs of $\epsilon\Z^2$ for $\epsilon>0$. Let $\mathcal{D}=(V,E)$
be a discrete domain in $\epsilon\Z^2$, so $\mathcal{D}$ is therefore a finite subgraph
of $\epsilon\Z^2$. We can view the boundary of $\mathcal{D}$ as a piecewise
linear simple loop $\lambda$ and write $\lambda^o$ for the open set of points
it surrounds.

To each bond percolation $\omega \in \{0,1\}^E$ and a cluster $C\in
\mathcal{C}(\omega)$ we can associate the set
\begin{align*}
    O^C_\omega = \bigcup_{v,w\in C\colon \omega_{vw}=1}\,R_{vw} \cup
	\bigcup_{v\in V} \left(v + [-\epsilon/4,\epsilon/4]^2\right)
\end{align*}
where $R_{vw} = \{tv + (1-t)w\colon t\in [0,1]\}+[-\epsilon/4,\epsilon/4]^2$ is
a rectangle with width $\epsilon/2$ and length $3\epsilon/2$ centered around the
edge $vw$.

\begin{figure}
	\centering
	\def\svgwidth{0.9\columnwidth}
    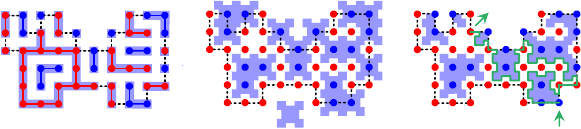
    \caption{\emph{Left.} The set $\cup_{C\in \mathcal{C}(\omega)} O^C_\omega$
        associated to a bond percolation configuration $\omega$ is shaded in
        blue and the clusters $\mathcal{C}(\omega)$ are colored to obtain a
        configuration $\sigma$. Edges which are part of the boundary of the
        domain but are not open with respect to $\omega$ are drawn as a dashed
        line. \emph{Center.} We consider a site percolation configuration
        $\sigma$ and draw a tile $X^+$ (as shown at the bottom) over each blue
        vertex and take the union to obtain the area shaded in blue.
        \emph{Right.} The set in the center figure is intersected with the
        closure of the domain to obtain $O^+_\sigma$. This step corresponds to
        the definition of blue clusters in terms of weak paths which stay within
        the discrete domain. The curve in green shows the curve
        $\gamma^+_{\sigma,a,b}$ from the boundary point $a$ (at the bottom) to
        $b$ (at the top).}
	\label{fig:loop-encoding}
\end{figure}

We write $\eta^C_\omega$ for the boundary of the unbounded connected component
of $\C\setminus O^C_\omega$ and $\Gamma^C_\omega$ for the collection of
boundaries of the bounded connected component of $\C\setminus O^C_\omega$. We
call the loop $\eta^C_\omega$ the outer boundary and view it as an element of
$\mathcal{C}$ and $\Gamma^C_\omega$ the collection of inner boundaries of $C$,
viewed as an element of $\mathcal{L}$ (see Figure \ref{fig:loop-encoding}).

We now define the loop encoding of the bond percolation configuration $\omega$ by
$\Gamma_\omega := \Gamma^O_\omega\cup \Gamma^I_\omega$ where
\begin{align*}
    \Gamma^O_\omega = \{\eta^C_\omega\colon C\in
    \mathcal{C}(\omega)\}\quad\text{and}\quad \Gamma^I_\omega =
    \bigcup_{C\in \mathcal{C}(\omega)} \Gamma^C_\omega
\end{align*}
are the sets of outer and inner boundaries respectively. Note that
$\Gamma^O_\omega$ (resp.\ $\Gamma^I_\omega$) is the set of loops in
$\Gamma_\omega$ with odd (resp.\ even) nesting level in $\Gamma_\omega$.

The constant $1/4$ in the definition here is completely arbitrary and any
constant $<1/2$ would work. In much of the literature, the loop encoding is
defined slightly differently but for statements about scaling limit results, the
differing definitions of the loop encodings do not matter.

We write $\Gamma^\partial_\omega$ for the loops in $\Gamma_\omega$ which do not
stay within $\lambda^o$ (these are the outer boundaries of clusters of $\omega$
intersecting $\partial\mathcal{D}$).

Let us now suppose that $\omega\in \{0,1\}^E$ and that $\sigma\in \{R,B\}^V$ is
obtained from $\omega$ by assigning colors to its clusters $\mathcal{C}(\omega)$
(we are therefore working implicitly with free boundary conditions). Let us
first make the following definitions:
\begin{align*}
    \Gamma^{OR}_{\omega,\sigma} &= \{ \eta^C_\omega\colon C\in
    \mathcal{C}(\omega)\,, \sigma_v=R\,,\ \forall v\in C\}\;,\\
    \Gamma^{IR}_{\omega,\sigma} &= \bigcup_{C\in \mathcal{C}(\omega)\colon
    \sigma_v=R\,\forall v\in C} \Gamma^C_\omega \;.
\end{align*}
So $\Gamma^{OR}_{\omega,\sigma}$ (resp.\ $\Gamma^{IR}_{\omega,\sigma}$) is the set of
outer (resp.\ inner) boundaries of red clusters. We analogously define
$\Gamma^{OB}_{\omega,\sigma}$ and $\Gamma^{IB}_{\omega,\sigma}$ as the outer and
inner boundaries of blue clusters. Now we encode the site percolation
configuration $\sigma$ in terms of a collection as loops as well. To this end,
let
\begin{align*}
    O^\pm_\sigma = \left(\,\bigcup_{v\in V\colon \sigma_v=B}\,(v+X^\pm)
    \right)\cap \overline{\lambda^o}
\end{align*}
where $X^+$ (resp.\ $X^-$) is a square with some boxes superimposed on (resp.\
removed from) its corners; more precisely, we define
\begin{align*}
    X^+ = [-\epsilon/2,\epsilon/2]^2 \cup
    (\{(\pm\epsilon/2,\pm\epsilon/2)\}+[-\epsilon/4,\epsilon/4]^2)\;,\\
    X^- = [-\epsilon/2,\epsilon/2]^2 \setminus
    (\{(\pm\epsilon/2,\pm\epsilon/2)\}+(-\epsilon/4,\epsilon/4)^2)\;.
\end{align*}
The boundary of $O^\pm_\sigma$ can be written as the disjoint union of simple
loops and we denote this set of simple loops by $\Sigma^\pm_\sigma$.

The convention in the definition of $\Sigma_\sigma^+$ corresponds to the
convention of defining blue fuzzy Potts clusters in terms of weak paths and
red fuzzy Potts clusters in terms of strong paths while the definition of
$\Sigma_\sigma^-$ arises from considering strong blue and weak red paths. It
will turn out that this choice does not matter when we are determining the
scaling limit of the critical fuzzy Potts measure with $q\in [1,4)$ since the
loops in the scaling limit will all be disjoint.

If $a,b\in \partial\mathcal{D}$ are distinct we also want to associate an
interface from $a$ to $b$ which is blue on its right and red on its left. We
make the following definition: The interface $\gamma^\pm_{\sigma,a,b}$ is given
by the unique simple curve from $a$ to $b$ which traces the counterclockwise
boundary arc of $\lambda$ from $a$ to $b$ except that it follows the boundary
of $O^\pm_\sigma$ in clockwise order whenever it hits $O^\pm_\sigma$. This
definition is illustrated in Figure \ref{fig:loop-encoding}.

Let us now quickly explain what we mean by the convergence of discrete domains
to continuum ones. This will appear in the formulation of the scaling limit
conjectures and results. Consider a Jordan domain $D\subset \C$ and write
$\lambda_\infty\in \mathcal{C}$ for its boundary. If $\epsilon_n\to 0$ we
consider discrete domains $\mathcal{D}_n = (V_n,E_n)$ in $\epsilon_n\Z^2$ and
associate to $\mathcal{D}_n$ its boundary curve $\lambda_n$ which we view as an
element of $\mathcal{C}$. We say that $\mathcal{D}_n$ converges to $D$ as
$n\to \infty$ if $d_\mathcal{C}(\lambda_n,\lambda_\infty)\to 0$ as $n\to \infty$.

\subsection{(Boundary) conformal loop ensembles}
\label{sec:cle-bcle}

We assume familiarity with Schramm-Loewner evolutions (SLE), see
\cite{schramm0, werner-notes, lawler-book}, conformal loop ensembles (CLE), see
\cite{shef-cle, shef-werner-cle, cle-percolations}, and boundary conformal loop
ensembles (BCLE), see \cite{cle-percolations}. However, we will briefly review
some of the key elements of the theory of CLEs and BCLEs for the reader's
convenience.

Let us begin by recalling the definition of $\SLE_\kappa(\rho_-,\rho_+)$ started
from $\xi_0$ with initial force points at $O_0^-\le \xi_0$ (on the left) and
$O_0^+\ge \xi_0$ (on the right) when $\kappa>0$ and $\rho_\pm>-2$.

It is the curve generated by the chordal Loewner chain with driving function
$\xi$ where $\xi$ is the unique weak solution to the SDE
\begin{align*}
	d\xi_t &= \sqrt{\kappa}\,dB_t + \frac{\rho_-\,dt}{\xi_t-O^-_t}
		+ \frac{\rho_+\,dt}{\xi_t-O^+_t}\;,\\
	dO^\pm_t &= \frac{2\,dt}{O^\pm_t-\xi_t}\;,
\end{align*}
with initial values $\xi_0$, $O^\pm_0$ and $O^-\le \xi\le O^+$ where $B$ is a
standard Brownian motion. The fact that a solution exists, uniqueness in law
holds and that the Loewner chain is generated by a continuous curve from $\xi_0$
to $\infty$ is not trivial, see \cite{ig1}.

If $\xi_0=O^\pm_0=0$ we obtain a scale invariant law on curves from $0$ to
$\infty$ which we simply call $\SLE_\kappa(\rho_-,\rho_+)$ without reference to
force or initial points. By applying a conformal transformation from $\H$ to any
Jordan domain, we can define $\SLE_\kappa(\rho_-,\rho_+)$ in any Jordan domain.
This is also well-defined in the case when the force points and the starting
point are all the same (and so the conformal transformation is not unique) by
the scale-invariance of the curve in this case.

\begin{figure}
	\centering
	\def\svgwidth{0.8\columnwidth}
\begingroup%
  \makeatletter%
  \providecommand\color[2][]{%
    \errmessage{(Inkscape) Color is used for the text in Inkscape, but the package 'color.sty' is not loaded}%
    \renewcommand\color[2][]{}%
  }%
  \providecommand\transparent[1]{%
    \errmessage{(Inkscape) Transparency is used (non-zero) for the text in Inkscape, but the package 'transparent.sty' is not loaded}%
    \renewcommand\transparent[1]{}%
  }%
  \providecommand\rotatebox[2]{#2}%
  \newcommand*\fsize{\dimexpr\f@size pt\relax}%
  \newcommand*\lineheight[1]{\fontsize{\fsize}{#1\fsize}\selectfont}%
  \ifx\svgwidth\undefined%
    \setlength{\unitlength}{302.57253409bp}%
    \ifx\svgscale\undefined%
      \relax%
    \else%
      \setlength{\unitlength}{\unitlength * \real{\svgscale}}%
    \fi%
  \else%
    \setlength{\unitlength}{\svgwidth}%
  \fi%
  \global\let\svgwidth\undefined%
  \global\let\svgscale\undefined%
  \makeatother%
  \begin{picture}(1,0.59415438)%
    \lineheight{1}%
    \setlength\tabcolsep{0pt}%
    \put(0,0){\includegraphics[width=\unitlength,page=1]{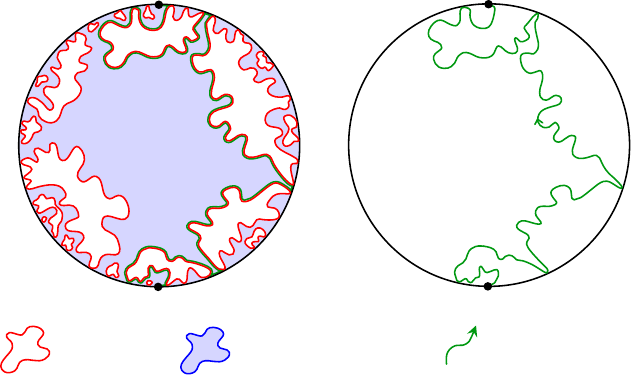}}%
    \put(0.11101951,0.0364732){\makebox(0,0)[lt]{\lineheight{1.25}\smash{\begin{tabular}[t]{l}\textcolor{red}{$\BCLE_\kappa(\rho)$}\end{tabular}}}}%
    \put(0.39855387,0.0364732){\makebox(0,0)[lt]{\lineheight{1.25}\smash{\begin{tabular}[t]{l}\textcolor{blue}{$\BCLE_\kappa(\kappa-6-\rho)$}\end{tabular}}}}%
    \put(0.78523804,0.0364732){\makebox(0,0)[lt]{\lineheight{1.25}\smash{\begin{tabular}[t]{l}\textcolor{mygreen}{$\SLE_\kappa(\rho,\kappa-6-\rho)$}\end{tabular}}}}%
  \end{picture}%
\endgroup%

	\caption{This figure illustrates the relation between $\Xi\sim
	\BCLE_\kappa(\rho)$ (in red), its false loops
	$\Xi^*\sim \BCLE_\kappa(\kappa-6-\rho)$ (the boundaries of the regions
	shaded in light blue) and the curve
		$\gamma\sim \SLE_\kappa(\rho,\kappa-6-\rho)$ (in green).}
	\label{fig:bcle-def}
\end{figure}
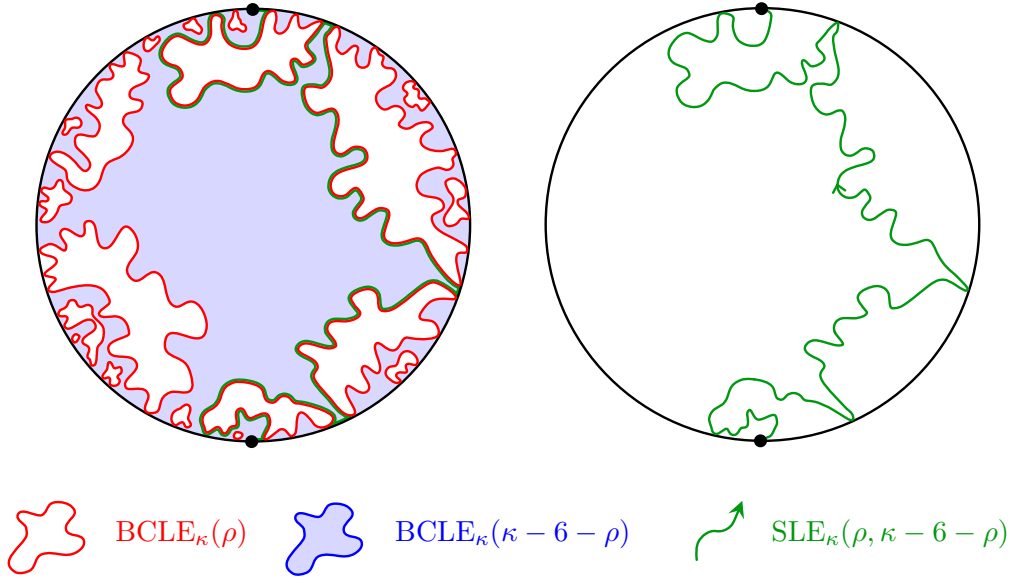

Boundary conformal loop ensembles are conformally invariant laws on
$\mathcal{L}$ in Jordan domains such that each loop intersects the
boundary of the domain. BCLEs consist either of simple or non-simple curves:
\begin{itemize}
	\item $\BCLE_\kappa(\rho)$ is defined for $\kappa\in (2,4]$ and
		$\rho\in (-2,\kappa-4)$ and is a conformally invariant law on
		collections of simple loops in the unit disk $\D$.
	\item $\BCLE_{\kappa'}(\rho')$ is defined for $\kappa'\in (4,8)$ and
		$\rho' \in [\kappa'/2-4,\kappa'/2-2]$ and is a conformally invariant law
		on collections of non-simple loops in $\D$.
\end{itemize}
The conformal invariance property implies that by applying conformal
transformations we obtain a well-defined notion of BCLE in any Jordan domain; we
restrict to Jordan domains here since this ensures that the BCLE loops in the
new domain are again continuous curves in the Euclidean topology (this fails for
domains with more pathological boundaries).

The definition of BCLEs further extends to domains which are disjoint unions of
Jordan domains by sampling independent BCLEs within each connected component
and the notion of false loops introduced in the paragraph below extends to this
setting as well.

If $\Xi\sim \BCLE_\kappa(\rho)$ then we can consider all the
`boundary to boundary' segments of the loops (i.e.\ loop segments which touch
the boundary at exactly two points) and observe that these loop segments in fact
encode another collection of loops which we will denote by $\Xi^*$ with the same
`boundary to boundary' loop segments (see Figure \ref{fig:bcle-def}). The
collection $\Xi^*$ is called the collection of false loops of $\Xi$ and
$\Xi^*\sim \BCLE_\kappa(\kappa-6-\rho)$. The analogous construction can be
performed in the case of non-simple BCLE.

We will not explain the full construction of BCLE but we will make use of the
following key characterizing property: Consider $\Xi\sim \BCLE_\kappa(\rho)$.
For any two distinct boundary points $a$ and $b$ there is a unique curve
$\gamma$ (generated by a Loewner chain) from $a$ to $b$ such that all the loops
in $\Xi$ touching the counterclockwise boundary arc from $a$ to $b$ are on one
side of $\gamma$ and all loops in $\Xi^*$ touching the clockwise boundary arc
from $a$ to $b$ are on the other side of $\gamma$. We call $\gamma$ the
interface from $a$ to $b$ associated to $\Xi$. The defining property is that
$\gamma\sim \SLE_\kappa(\rho,\kappa-6-\rho)$. Similarly, in the case of
non-simple BCLEs, one obtains $\SLE_{\kappa'}(\rho',\kappa'-6-\rho')$ curves
from a $\BCLE_{\kappa'}(\rho')$.

We will momentarily want to sample BCLE within the complementary components of
the loops of BCLE so we need to mention that the complementary components of
BCLE loops are again Jordan domains. More precisely, suppose that $\eta$ is a
loop within a BCLE. Then $\C\setminus \eta(\partial\D)$ can be decomposed into
its connected components. It turns out that by SLE duality (as established in
the required generality in \cite{ig1}) all the bounded connected components are
Jordan domains. Moreover, the boundary of the unbounded component of
$\C\setminus \eta(\partial\D)$ is a simple curve which we call the boundary of
the filling of $\eta$.

The object $\CLE_{\kappa'}$ (which is a law on $\mathcal{L}$ supported on
collections of
non-simple loops) can now be described as follows: If $\Gamma$ is a nested
$\CLE_{\kappa'}$ then the law of its boundary touching loops $\Xi$ is a
$\BCLE_{\kappa'}(0)$ and the conditional law of $\Gamma$ given $\Xi$ is given by
the union of $\Xi$ together with an independent nested $\CLE_{\kappa'}$ within
each loop of $\Xi$ and within each false loop of $\Xi^*$. This formulation can
also be phrased as an iterative construction. Just as for BCLEs, CLEs are
defined in any disjoint union of Jordan domains by conformal invariance and
taking independent samples in the connected components.

Recall from the previous Section \ref{sec:loop-encoding} that in the discrete
setting, a percolation configuration can be encoded as a collection of nested
loops. The following classical convergence conjecture states that nested CLE
arise in the limit from these collections of discrete loops; this is so far only
known for $q=2$, i.e.\ $\kappa'=16/3$ by
\cite{smirnov-ising, kemp-smirnov-fk-bdy, kemp-smirnov-fk-full}. We stress that
the encoding of percolation configurations in terms of loops is convenient
precisely because this description is explicit using SLE tools in the continuum.

\begin{conj}
    \label{conj:fk-to-cle} Consider $q\in (0,4]$ and let
    $\kappa'=4\pi/\arccos(-\sqrt{q}/2) \in [4,8)$. Suppose that $D$ is a Jordan
    domain, $\epsilon_n\to 0$ as $n\to \infty$ and that $\mathcal{D}_n$ is a
    discrete domain in $\epsilon_n\Z^2$ for each $n\ge 1$ such that
    $\mathcal{D}_n$ converges to $D$ as $n\to \infty$. Let $\omega^n \sim
    \phi^0_{\mathcal{D}_n,q}$, let $\Gamma$ be a nested $\CLE_{\kappa'}$ in $D$
    and let $\Gamma^\partial$ be the loops in $\Gamma$ intersecting $\partial
    D$. Then $(\Gamma_{\omega^n}\setminus
    \Gamma^\partial_{\omega^n},\Gamma^\partial_{\omega^n})$ converges in
    distribution to $(\Gamma\setminus \Gamma^\partial,\Gamma^\partial)$ with
    respect to the metric $d_\mathcal{L}$.
\end{conj}

This conjecture has a long history going back to
\cite{nienhuis-coulomb} in the physics literature (see also \cite{nienhuis-sle}
where the convergence conjecture for SLE was first stated) and the conjecture in
the form here appears first in \cite{shef-cle}.

\subsection{CLE Percolations}
\label{sec:cle-perco}

We now define the divide and color model in the continuum setting. This is based
on the seminal work \cite{cle-percolations}. Throughout this text, we will refer
to the construction performed in this section as the continuum fuzzy Potts
model.

In the discrete setting, we assigned colors to the percolation clusters and then
agglomerated them into divide and color clusters. In the continuum setting, we
will construct the coupling of divide and color cluster boundaries and
percolation cluster boundaries in one step. Fix $\kappa'\in (4,8)$ and $r\in
(0,1)$. We will first define all the relevant parameters for the continuum
construction. Let
\begin{align}
	\label{eq:numerical-rel}
	\begin{split}
	\kappa &= 16/\kappa' \in (2,4)\;,\\
	q(\kappa') &= 4\cos^2(4\pi/\kappa') \in (0,4)\;,\\
	\rho_B(\kappa',r) &= \frac{2}{\pi}\,\arctan\left( \frac{\sin(\pi\kappa/2)}{1
		+\cos(\pi\kappa/2)-1/r}\right) -2 \in (-2,\kappa-4)\;,\\
	\rho_R(\kappa',r) &= \kappa-6-\rho_B(\kappa',r) \\
		&= \frac{2}{\pi}\,\arctan\left( \frac{\sin(\pi\kappa/2)}{1
		+\cos(\pi\kappa/2)-1/(1-r)}\right) -2 \in (-2,\kappa-4)\;,\\
	\rho'_B(\kappa',r) &= \frac{-\kappa'}{4}(\rho_B+2) \in
	(\kappa'/2-4,0)\;,\\
	\rho'_R(\kappa',r) &= \frac{-\kappa'}{4}(\rho_R + 2) \in
	(\kappa'/2-4,0)\;.
	\end{split}
\end{align}
When $\kappa'$ and $r$ are clear from the context, we will sometimes drop them
from the notation. In the coupling constructed below, $\Gamma^{OR}$ (resp.\
$\Gamma^{IR}$) correspond to the outer (resp.\ inner) boundaries of percolation
clusters which have been colored red (similarly in the case when we consider
blue clusters) and $\Sigma$ will be the collection of continuum fuzzy Potts
interfaces.

We will now perform a rather involved construction. See Figure
\ref{fig:continuum-dac} for an illustration.

\begin{figure}
	\centering
	\def\svgwidth{0.85\columnwidth}
	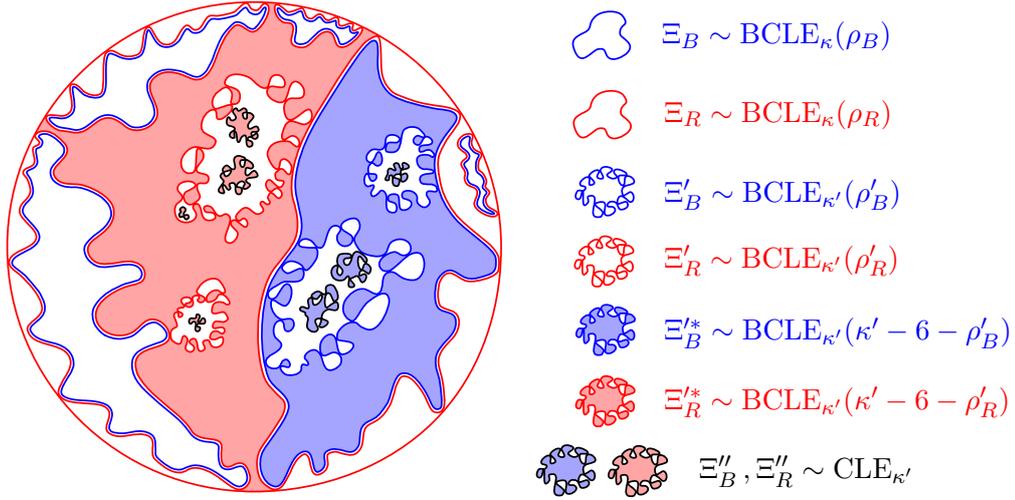
	\caption{The picture illustrates the constructions relevant for Theorems
		\ref{thm:msw-free} and \ref{thm:msw-wired}. The collection of $\Xi_B$ should be
		interpreted as the divide and color interfaces touching the boundary
		which are blue on the inside and red on the outside. $\Xi_B'$ are then
		the outer boundaries and $\Xi_B''$ the inner boundaries of the blue
		percolation clusters touching the divide and color interfaces from the
		inside; similarly, $\Xi_R'$ are the
		outer boundaries and $\Xi_R''$ are the inner boundaries of the
		percolation clusters touching the outside of the divide and color
		interfaces (and are hence contained in $\Xi_R:=\Xi_B^*$). The situation
		is symmetric when `red' is interchanged with `blue'. In each of the
		areas shaded in red (resp.\ blue), we now iterate with red (resp.\ blue)
		boundary conditions on the outside.}
	\label{fig:continuum-dac}
\end{figure}

Let $\Xi_B\sim \BCLE_\kappa(\rho_B)$ in the unit disk $\D$ and let $\Xi_R:=
\Xi_B^*\sim \BCLE_\kappa(\rho_R)$ be the collection of its false loops. Within
$\Xi_B$  we make the following definition.
\begin{itemize}
	\item Let $\Xi'_{B}\sim \BCLE_{\kappa'}(\rho'_B)$ in $\cup_{\eta\in
		\Xi_B} \eta^o$ and let $\Xi_B'^*$ be its false loops which then forms a
		$\BCLE_{\kappa'}(\kappa'-6-\rho'_B)$ in $\cup_{\eta\in \Xi_B} \eta^o$.
	\item Moreover, let $\Xi''_{B}$ be an non-nested $\CLE_{\kappa'}$ in
		$\cup_{\eta\in \Xi'_B} \eta^o$.
\end{itemize}
We make the analogous definition with $(\Xi_B,\rho'_B)$ replaced by
$(\Xi_R,\rho'_R)$ to obtain $(\Xi'_R,\Xi_R'^*,\Xi_R'')$. Using conformal
invariance and by taking independent samples in connected components, we can
sample the tuple $\Xi :=(\Xi_B,\Xi_B',\Xi_B'^*,\Xi_B'',
\Xi_R,\Xi_R',\Xi_R'^*,\Xi_R'')$ in any domain which is a disjoint union of
Jordan domains.

Fix a Jordan domain $D$ and write $\eta_\partial$ for the loop tracing the
boundary. To perform the iteration, we first let
$\Upsilon^B_0=\emptyset$ and $\Upsilon^R_0=\{\eta_\partial\}$. Suppose that we
have already constructed
$((\Sigma_i,\Upsilon^B_i,\Upsilon^R_i,\Gamma_i^{OB},\Gamma_i^{OR},\Gamma_i^{IB},
\Gamma_i^{IR})\colon 1\le i\le n-1)$ for $n\ge 1$. Then we proceed as follows:
\begin{itemize}
	\item Sample a copy of $\Xi$ within $\cup_{\eta\in \Upsilon^R_{n-1}}
        \eta^o$ and call it $\Xi_{R\to B}^n$. We also sample an independent copy
        of $\Xi$ in $\cup_{\eta\in \Upsilon^B_{n-1}} \eta^o$ and call the
        resulting tuple $\Xi_{B\to R}^n$. We write
		\begin{align*}
			\Xi_{R\to B}^n &=(\Xi_{RB}^n,\Xi_{RB}'^n,\Xi_{RB}'^{*n}, \Xi_{RB}''^{n},
			\Xi_{RR}^n,\Xi_{RR}'^n,\Xi_{RR}'^{*n},\Xi_{RR}''^n)\;,\\
			\Xi_{B\to R}^n &=(\Xi_{BB}^n,\Xi_{BB}'^n,\Xi_{BB}'^{*n}, \Xi_{BB}''^{n},
			\Xi_{BR}^n,\Xi_{BR}'^n,\Xi_{BR}'^{*n},\Xi_{BR}''^n)\;.
		\end{align*}
    \item Define $\Sigma_{n}$, $\Upsilon^B_{n}$, $\Upsilon^R_{n}$,
        $\Gamma^{OB}_n$, $\Gamma^{OR}_n$, $\Gamma^{IB}_n$ and $\Gamma^{IR}_{n}$
        by
		\begin{gather*}
			\Sigma_n = \Xi^n_{RB}\cup \Xi^n_{BR} \;, \\
			\Upsilon^B_n = \Xi_{RB}'^{*n} \cup \Xi_{RB}''^n \cup \Xi_{BB}'^{*n}
			    \cup \Xi_{BB}''^n \;, \quad
			\Upsilon^R_n = \Xi_{BR}'^{*n} \cup \Xi_{BR}''^n \cup \Xi_{RR}'^{*n}
		    	\cup \Xi_{RR}''^n \;, \\
            \Gamma_n^{OB} = \Xi_{BB}'^n\cup\Xi_{RB}'^n \;,\quad
            \Gamma_n^{OR} = \Xi_{BR}'^n \cup \Xi_{RR}'^n\;,\\
            \Gamma_n^{IB} = \Xi_{BB}''^n\cup \Xi_{RB}''^n \quad
            \Gamma_n^{IR} = \Xi_{BR}''^n \cup \Xi_{RR}''^n \;.
		\end{gather*}
\end{itemize}
The reader is encouraged to look at Figure \ref{fig:continuum-dac} where the
first step of this iteration is displayed and the inductive definition
mentioned. Lastly, we define $\Sigma$, $\Gamma^{OR}$, $\Gamma^{OB}$,
$\Gamma^{IR}$, $\Gamma^{IB}$ as the unions of the collection $(\Sigma_n)$,
$(\Gamma^{OR}_n)$, $(\Gamma^{OB}_n)$,
$(\Gamma^{IR}_n)$, $(\Gamma^{IB}_n)$ respectively. We also set
\begin{gather*}
    \Gamma^O = \Gamma^{OR}\cup \Gamma^{OB}\;,\quad
    \Gamma^I = \Gamma^{IR}\cup \Gamma^{IB}\;,\quad \Gamma = \Gamma^O \cup
    \Gamma^I \;.
\end{gather*}
From the construction, we see that $\Gamma^O$ (resp.\ $\Gamma^I$) is the
collection of loops in $\Gamma$ with odd (resp.\ even) nesting level. Moreover,
any loop in $\Gamma^I$ is in $\Gamma^{IR}$ (resp.\ $\Gamma^{IB}$) if and only if
its parent (i.e.\ the loop in $\Gamma^O$ surrounding it with maximal
nesting level) is in $\Gamma^{OR}$ (resp.\ $\Gamma^{OB}$). The interpretation of
$\Sigma$ as the set of continuum fuzzy Potts interfaces is justified by the fact
that in the construction above there are always blue loops on one side and red
loops on the other side of the interface.

The key is the following theorem, the proof of which is given in \cite[Theorem
7.2]{cle-percolations} in combination with \cite[Theorem 1.3]{msw-non-simple}.

\begin{thm}[\cite{cle-percolations, msw-non-simple}]
	\label{thm:msw-free}
    Suppose that $\kappa'\in (4,8)$ and $r\in (0,1)$. The collection $\Gamma$ is
    a nested $\CLE_{\kappa'}$ in $D$. The conditional law of
    $(\Gamma^{OR},\Gamma^{OB})$ given $\Gamma$ is given by adding each loop in
    $\Gamma^O$ independently to $\Gamma^{OR}$ with probability $r$ and to
    $\Gamma^{OB}$ otherwise.
\end{thm}

\begin{remark}
    \label{rk:interface-def}
    For $a,b\in \partial D$ we can define the interface $\gamma^{a,b}$ from $a$
    to $b$ associated to $\Xi^1_{RB}=\Sigma_1$. By definition, we marginally have
	$\gamma^{a,b}\sim \SLE_\kappa(\rho_B,\kappa-6-\rho_B)$.
\end{remark}

The next result is the continuum version of the classical Edwards-Sokal coupling
with wired boundary conditions. Indeed, the appearance of $\Gamma_0$ in the
statement corresponds on the discrete side to the inner boundaries of the
boundary cluster (when considering wired boundary conditions). The result was
proved in \cite[Theorem 7.7]{cle-percolations} with the inexplicit parameter $\rho = \rho(\beta,\kappa')$ 
there having been determined in \cite[Theorem 1.2]{msw-non-simple}.

\begin{thm}[\cite{cle-percolations, msw-non-simple}]
	\label{thm:msw-wired}
    Suppose that $\kappa'\in (4,6)$ and let $r=1/q(\kappa')$. Let $\Gamma_0$ be
    a non-nested $\CLE_{\kappa'}$ in a Jordan domain $D$ and for each connected
    component $C$ of $\eta^o$ where $\eta\in \Gamma_0$ we take an independent
    copy of $\Sigma$ and conformally map it into $C$. Let $\Sigma'$ be the union
    of all these collections of loops. Then $\Sigma'$ forms a nested
    $\CLE_\kappa$ in $D$.
\end{thm}

So far, these definitions are rather different from the ones made in the
discrete setting. In particular, nowhere in the constructions above are we
agglomerating CLE clusters into the continuum fuzzy Potts clusters. As part of
our proof of the convergence of the discrete fuzzy Potts clusters to the
continuum counterpart, we will need such a description. The type of input we
need is an approximation of the continuum fuzzy Potts interfaces by finite
chains of blue CLE loops on one side and finite chains of red CLE loops on the
other side. This is achieved by the result below which is a direct consequence
of the proof of \cite[Proposition 6.1]{cle-percolations} (by formalizing the
definition of $c_\epsilon^\pm$ appearing there). Below $\gamma^{a,b}$ is defined
as in Remark \ref{rk:interface-def}.

\begin{prop}[\cite{cle-percolations}]
    \label{prop:msw-interface-approx}
    Consider $D=\D$ and let $\Gamma_f^{OB}$ be the collection of boundaries of
    the fillings of the outermost loops in $\Gamma^{OB}$ (all oriented clockwise
    here by convention). For any distinct $a,b\in \partial\D$ the following
    statement holds almost surely:

    Consider $0\le s<t\le 1$ satisfying $\gamma^{a,b}_s,\gamma^{a,b}_t\in
    (\!(a,b)\!)$ and $\gamma^{a,b}((s,t))\cap (\!(a,b)\!)=\emptyset$. Then for
    all $\epsilon>0$ there are $\eta^1,\dots,\eta^n\in \Gamma_f^{OB}$ and
    $s^\pm_1,\dots,s^\pm_n\in \partial\D$ such that the concatenation of
	$\eta^1\lvert_{(\!(s^-_1,s^+_1)\!)},\dots,\eta^n\rvert_{(\!(s^-_n,s^+_n)\!)}$ defines
    a simple continuous curve $\widetilde{\gamma}$, viewed as a function on
    $[0,1]$, with
    \begin{align*}
        \widetilde{\gamma}_0\in B_\epsilon(\gamma^{a,b}_s)\cap \partial\D\;,\quad
        \widetilde{\gamma}_1\in B_\epsilon(\gamma^{a,b}_t)\cap\partial\D \;,\quad
        \widetilde{\gamma}([0,1])\subset \gamma^{a,b}([s,t])+B_\epsilon(0)
    \end{align*}
    and which lies right of the curve $\gamma^{a,b}$.
\end{prop}

The analogous result holds for approximating the interface $\gamma^{a,b}$ from
the left by red CLE cluster chains and its formulation is obtained by replacing
$\Gamma^{OB}$, $\Gamma^{OB}_f$, `clockwise', `right' and $(\!(a,b)\!)$ by
$\Gamma^{OR}$, $\Gamma^{OR}_f$, `counterclockwise', `left' and $(\!(b,a)\!)$
respectively.

\begin{remark}
    \label{rk:small-boundary-clusters}
    Let us make two additional observations about this proposition here: It is a
    consequence of the definition of $\CLE_{\kappa'}$ that a.s.\ no two distinct
    loops in $\Gamma$ contain the same boundary point and that no loop in
    $\Gamma$ hits a boundary point twice. Hence in the proposition above we have
    \begin{align*}
        \eta^i_{s^+_i}=\eta^{i+1}_{s^-_{i+1}}\notin\partial\D\quad
        \text{for all $i<n$}\;.
    \end{align*}
    This comment will also equally apply to Corollary \ref{cor:msw-loop-approx}.
    Furthermore, there is no loop in $\Gamma$ that lies right of $\gamma^{a,b}$ and which intersects the set
    $\gamma^{a,b}([0,1])\cap (\!(a,b)\!)$; this is a consequence of the inductive
    construction of $\Gamma$ in this section in terms of iterated BCLEs. Hence,
	by the local finiteness of $\Gamma$, for any given $\delta>0$, there exists $\epsilon >0$ such that $\eta^1$ and $\eta^n$ in Proposition \ref{prop:msw-interface-approx} also satisfy 
    $\diam(\eta^1),\diam(\eta^n)<\delta$.
\end{remark}

\begin{figure}
	\centering
	\def\svgwidth{0.8\columnwidth}
    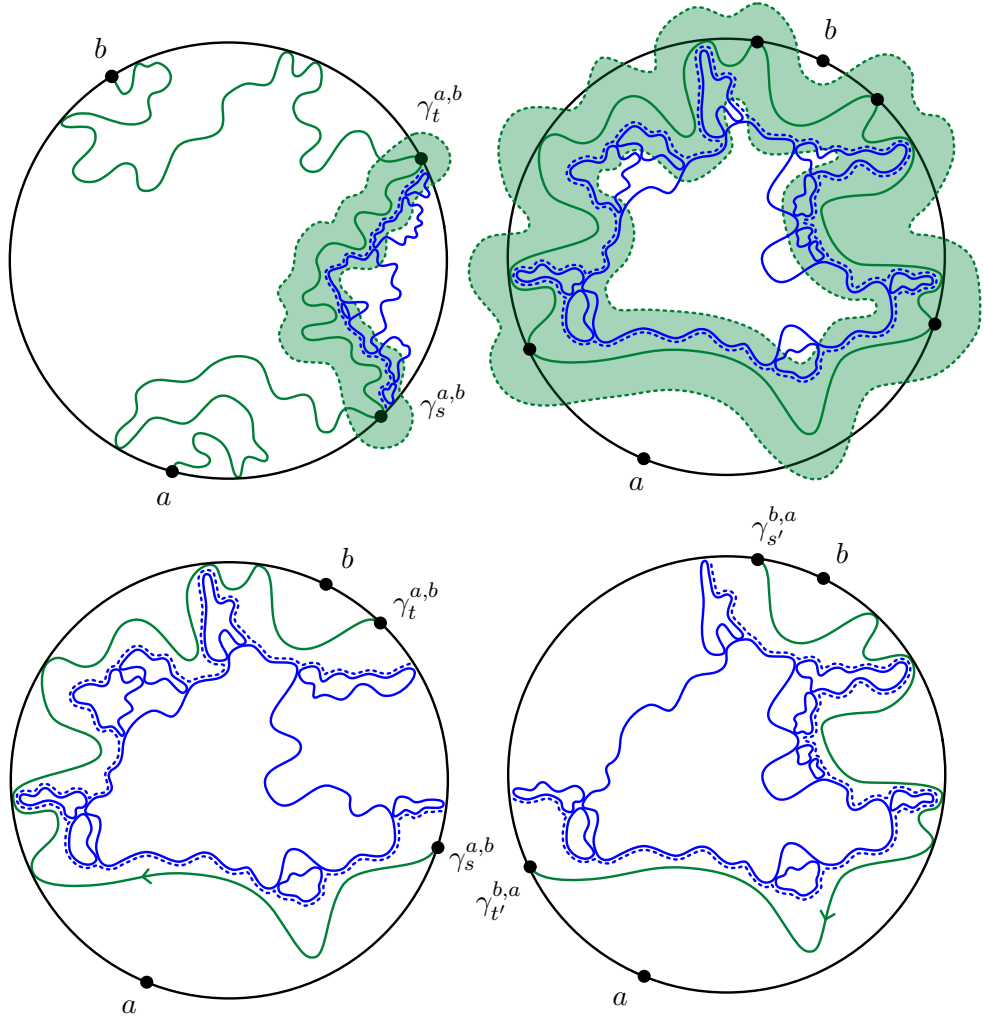
    \caption{\emph{Top left.} This figure illustrates Proposition \ref{prop:msw-interface-approx}; the green curve is $\gamma^{a,b}$, the area shaded in green is $\gamma^{a,b}([s,t])+B_\epsilon(0)$ and the dashed blue curve is $\widetilde{\gamma}$. \emph{Top right.} This graphic explains Corollary \ref{cor:msw-loop-approx}. The green loop is $\eta$, the shaded are in green is $\eta(\partial\D)+B_\epsilon(0)$ and the dashed blue loop is $\widetilde{\eta}$. The bottom row illustrates the argument which can be used to derive Corollary \ref{cor:msw-loop-approx} from Proposition \ref{prop:msw-interface-approx}. \emph{Bottom left.} The green curve is $\gamma^{a,b}$ restricted to $[s,t]$ and the dashed blue curve is the approximating curve appearing in the proposition (see also the top left part of this figure). Crucially, the restriction of $\gamma^{a,b}$ restricted to $[s,t]$ yields a segment of the loop $\eta$. \emph{Bottom right.} The green curve is $\gamma^{b,a}$ restricted to $[s',t']$ and the dashed blue curve is again the approximating curve as in the proposition. Again this curve segment forms part of $\eta$. The key is that the approximating curves appearing in the two bottom figures intersect which readily implies the corollary.}
	\label{fig:approximate-continuum}
\end{figure}

There is also a version of the result above in the case of the actual BCLE loops
rather than the interfaces. Indeed, still in the setting $D=\D$, for any
$\eta\in \Xi^1_{RB}$ we can take $a,b\in \mathcal{Q}:=\{e^{i\theta}\colon
\theta\in \Q\}$ such that $\eta(\partial\D)$ intersects both $(\!(a,b)\!)$ and
$(\!(b,a)\!)$.

Then $\eta(\partial\D)=\gamma^{a,b}([s,t])\cup
\gamma^{b,a}([s',t'])$ where
\begin{gather*}
    s = \inf\{u\ge 0\colon \gamma^{a,b}_u\in (\!(a,b)\!)\cap \eta(\partial\D)\}\;,\quad
    t = \sup\{u\ge 0\colon \gamma^{a,b}_u\in (\!(a,b)\!)\cap \eta(\partial\D)\}\;,\\
    s' = \inf\{u\ge 0\colon \gamma^{b,a}_u\in (\!(b,a)\!)\cap \eta(\partial\D)\}\;,\quad
    t' = \sup\{u\ge 0\colon \gamma^{b,a}_u\in (\!(b,a)\!)\cap \eta(\partial\D)\}\;.
\end{gather*}
By applying Proposition \ref{prop:msw-interface-approx} to the segments
$\gamma^{a,b}([s,t])$ and $\gamma^{b,a}([s',t'])$ we obtain curves
$\widetilde{\gamma}^{a,b}$ and $\widetilde{\gamma}^{b,a}$ (for any prespecified
$\epsilon>0$). By noting that $\widetilde{\gamma}^{a,b}$ and
$\widetilde{\gamma}^{b,a}$ intersect for $\epsilon>0$ sufficiently small, we
readily obtain the following corollary. See also Figure \ref{fig:approximate-continuum} where this is illustrated.

\begin{cor}
    \label{cor:msw-loop-approx}
    Consider $D=\D$ and let $\Gamma^{OB}_f$ be defined as in Proposition
    \ref{prop:msw-interface-approx}. Then the following statement is almost
    surely true: Suppose that $\eta\in \Xi^1_{RB}$ surrounds a point $z\in \D$
    and fix $\epsilon>0$. Then there are $\eta^1,\dots,\eta^n\in \Gamma^{OB}_f$
    (all oriented counterclockwise say) and $s_1^\pm,\dots,s_n^\pm\in
    \partial\D$ such that the concatenation of the curves
	$\eta^1\lvert_{(\!(s^-_1,s^+_1)\!)},\dots,\eta^n\rvert_{(\!(s^-_n,s^+_n)\!)}$ defines
    a simple loop $\widetilde{\eta}$, viewed as a function on $\partial\D$, such
    that $\eta$ surrounds the curve $\widetilde{\eta}$, $\widetilde{\eta}$
    surrounds $z$ and $\widetilde{\eta}(\partial\D)\subset \eta(\partial\D) +
    B_\epsilon(0)$.
\end{cor}

\subsection{Imaginary geometry results}
\label{sec:ig-lemmas}

In this short section, we collect some results from the theory of imaginary
geometry which will be used in Section \ref{sec:continuum-exp}. We encourage the
reader to skip this section and refer back to it whenever necessary when reading
Section \ref{sec:continuum-exp}.

\begin{figure}
	\centering
	\def\svgwidth{0.8\columnwidth}
	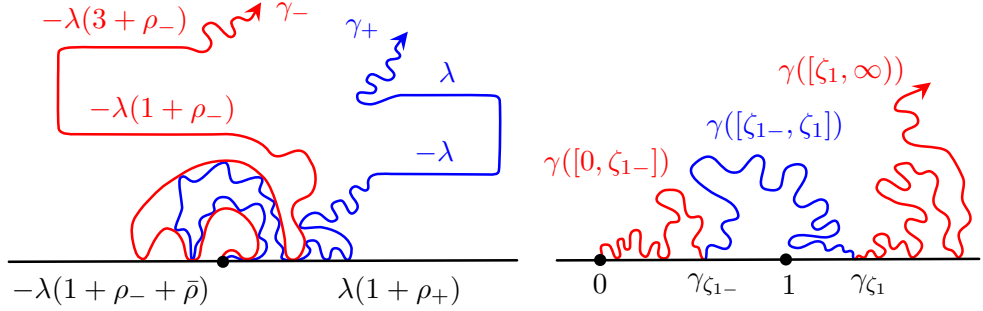
	\caption{\emph{Left.} This figure illustrates Lemma
        \ref{lem:imaginary-map-in}. The relevant imaginary geometry boundary
        conditions have been drawn in for the reader's convenience where
        $\lambda = \pi/\sqrt{\kappa}$. \emph{Right.} Lemma
        \ref{lem:sle-excursion} is illustrated: The conditional law of the blue
		curve given the red curves is an $\SLE_\kappa(0,\rho')$ in the
        complementary domain (up to time reparametrization). Both figures are
        drawn in the upper halfplane rather than the unit disk for
        illustrational purposes.}
	\label{fig:imaginary-lemmas}
\end{figure}

The following two results appear as part of \cite{ig1} (see also
\cite[Section 8]{cle-percolations}) and in \cite[Theorem 5.6]{wedges}
respectively. They are illustrated in Figure \ref{fig:imaginary-lemmas}.

\begin{lemma}[\cite{ig1}]
	\label{lem:imaginary-map-in}
	Let $\kappa\in (0,4)$, $\rho_-,\rho_+>-2$ and $\bar{\rho}>0$. Then one can
	couple two curves $\gamma_- \sim \SLE_\kappa(\bar{\rho}-2,2+\rho_-+\rho_+)$
	and $\gamma_+ \sim \SLE_\kappa(\rho_-+\bar{\rho},\rho_+)$ right of
    $\gamma_-$ (both from $-i$ to $i$ in $\D$) such that conditionally on
    $\gamma_-$, the restrictions of $\gamma_+$ to the components right of
	$\gamma_-$ are independent $\SLE_\kappa(\rho_-,\rho_+)$ curves.
\end{lemma}

\begin{lemma}[\cite{wedges}]
	\label{lem:sle-excursion}
	Fix $\kappa\in (0,4)$, $\rho \in (-2,\kappa/2-2)$ and define
	$\rho'=\kappa-4-\rho$. Let $\gamma\sim \SLE_\kappa(0,\rho)$ from $-i$ to $i$
    in $\D$, write $\zeta_1$
    \begin{align*}
        \zeta_1 = \inf\{t\ge 0\colon \gamma_t \in (\!(1,i)\!) \}\;,\quad
	    \zeta_{1-}=\sup\{t<\zeta_1\colon \gamma_t\in (\!(-i,1)\!) \}\;.
    \end{align*}
    Then the conditional law of $\gamma([\zeta_{1-},\zeta_1])$ given
    $\gamma([0,\zeta_{1-}]\cup [\zeta_1,1])$ is that of the image of a
	$\SLE_\kappa(0,\rho')$ from $\gamma_{\zeta_{1-}}$ to $\gamma_{\zeta_1}$ in
    the component of $\D\setminus \gamma([0,\zeta_{1-}]\cup [\zeta_1,1])$
    with $1$ on its boundary.
\end{lemma}

Let us remark that the final lemma makes sense from the Loewner chain
description as well since an $\SLE_\kappa(0,\rho)$ can be encoded in terms of a
Bessel process of dimension $\delta = 1+2(\rho+2)/\kappa$ and the excursion
measure of such a Bessel process can be constructed from Bessel processes of
dimension $4-\delta=1+2(\rho'+2)/\kappa$ (i.e., the dimension is reflected
around $2$).

Another input on relations between different SLE curves that we will state here
is the following change of coordinates property of
$\SLE_\kappa(\rho,\kappa-6-\rho)$ curves (see \cite[Section 7]{cle-percolations}
and \cite{sw-coord}). This is related to the target invariance property of
$\SLE_\kappa(\rho,\kappa-6-\rho)$ curves which has been instrumental in
\cite{cle-percolations}.

\begin{lemma}[\cite{cle-percolations}]
	\label{lem:coordinate-change}
    Consider $\kappa\in (2,4)$ and $\rho\in (-2,\kappa-4)$. One can couple
	$\gamma\sim \SLE_\kappa(\rho,\kappa-6-\rho)$ from $-i$ to $i$ in $\D$ with
    initial force points at $-i$ (on the left) and $x\in (\!(-i,i)\!)$ with
	$\gamma'\sim \SLE_\kappa(\rho,0)$ from $-i$ to $x$ in $\D$ such that the
    following is true: If $\zeta_x=\inf\{t\ge 0\colon \gamma_t\in (\!(x,i)\!)\}$
    and $\zeta'_x =\inf\{t\ge 0\colon \gamma'_t\in (\!(x,i)\!)\}$ then
    $\gamma([0,\zeta_1])=\gamma'([0,\zeta'_1])$.
\end{lemma}

The final two lemmas we state here say that SLE curves stay within prespecified
tubes with positive probability.

The proofs of both results are based on imaginary geometry (IG) techniques. The
first result is a special case of \cite[Lemma 2.5]{miller-wu-dim} and the second
one follows from a straightforward adaptation of the proof of
\cite[Lemma 2.3]{miller-wu-dim}.\footnote{\,In the
setting of \cite{miller-wu-dim}, one considers a GFF $h$ in $\D$ with IG
boundary conditions $-\lambda(1+\rho_-)$ on $(\!(i,x_-)\!)$, $-\lambda$ on
$(\!(x_-,-i)\!)$, $\lambda$ on $(\!(-i,x_+)\!)$ and $\lambda(1+\rho_+)$ on
$(\!(x_+,i)\!)$, and the flow line $\gamma$ from $-i$ to $i$. One also considers
a GFF $h'$ in the domain between $\nu_-$ and $\nu_+$ which has the same IG
boundary conditions on $(\!(\nu^-_0,\nu^+_0)\!)\cup (\!(\nu^+_1,\nu^-_1)\!)$,
$-\lambda$ on $\nu^-([0,1])$ and $\lambda$ on $\nu^+([0,1])$. The flow line
$\gamma'$ from $-i$ to $i$ associated to $h'$ does not hit
$\nu^-([0,1])\cup\nu^+([0,1])$ a.s.\ and the proof proceeds as in
\cite{miller-wu-dim} by showing absolute continuity of $h$ and $h'$ when both
are restricted to a connected set which is a positive distance away from
$\nu^-([0,1])\cup\nu^+([0,1])$.}

\begin{lemma}[\cite{miller-wu-dim}]
    \label{lem:positive-hit}
	Consider $\kappa\in (0,4)$, $\rho_-\in (-2,\kappa/2-2)$ and $\rho_+>-2$. Let
	$\gamma\sim \SLE_\kappa(\rho_-,\rho_+)$ in $\D$ from $-i$ to $i$ with
    initial force points at $x_-\in (\!(i,-i)\!)\cup\{-i\}$ and $x_+ \in
    (\!(-i,i)\!)\cup \{-i\}$. Let $\nu^\pm\colon [0,1]\to \overline{\D}$ be two
    simple curves with disjoint images such that only their endpoints are on
    $\partial\D$. Also assume that $\nu^-_0\in (\!(i,-i)\!)$, $\nu^-_1 \in
    (\!(i,\nu^-_0)\!)\cap (\!(i,x_-)\!)$, $\nu^+_0 \in (\!(-i,i)\!)$ and
    $\nu^+_1 \in (\!(i,\nu^-_1)\!)$. Then with positive probability, $\gamma$
    hits $(\!(\nu^+_1, \nu^-_1)\!)$ before $\nu^-([0,1])\cup \nu^+([0,1])$.
\end{lemma}

\begin{lemma}[\cite{miller-wu-dim}]
    \label{lem:positive-whole}
	Consider $\kappa\in (0,4)$ and $\rho_\pm > -2$. Let $\gamma\sim
	\SLE_\kappa(\rho_-,\rho_+)$ from $-i$ to $i$ in $\D$ with initial force
    points at $x_- \in (\!(i,-i)\!)\cup\{-i\}$ and $x_+\in
    (\!(-i,i)\!)\cup\{-i\}$. Let $\nu^\pm \colon
	[0,1]\to \overline{\D}$ be disjoint simple curves with only their endpoints
	on $\partial\D$. Assume that $\nu^-_0\in (\!(i,-i)\!)$, $\nu^-_1 \in
	(\!(i,\nu^-_0)\!)$, $\nu^+_0 \in (\!(-i,i)\!)$ and $\nu^+_1\in
	(\!(\nu^+_1,i)\!)$. Then with positive probability,
	$\gamma$ does not hit $\nu^-([0,1])\cup \nu^+([0,1])$.
\end{lemma}

\section{Fuzzy Potts model: Quasi-multiplicativity and other discrete ingredients}
\label{sec:discrete-fuzzy}

In this section, we study the fuzzy Potts model based on the precise
understanding of critical FK percolation, summarized in Section
\ref{sec:FK-percolation}, that relies to a large extend on
\cite{duminil2021planar}. Throughout the section, the cluster weight $q\in
[1,4)$ and the coloring parameter $r\in (0,1)$ are fixed and we drop them from
the notation. We will prove the `discrete' ingredients needed to establish our
main results. First, we will show separation, extension and localization
properties of arms and these lead to quasi-multiplicativity for arm event
probabilities. This part follows the strategy known for Bernoulli percolation
(see \cite{Kesten1987,Nolin2008}) and FK percolation (see
\cite{Chelkak2016,duminil2020scaling,duminil2021planar}). Secondly, we will bound the alternating
six-arm exponent for the fuzzy Potts model and explain how to reduce the study
of arm exponents in the fuzzy Potts model to the case of alternating color
sequences.

In comparison to the study of FK percolation, there are several difficulties
arising due to the additional layer of randomness. To explain them and to
motivate our approach, let us describe a natural first attempt: One could try to
establish an analogue of the strong crossing estimates from Theorem
\ref{thm:strongRSW} for the fuzzy Potts model, in order to then follow in much
the same way the arguments used for FK percolation. In particular, one would
show that conditionally on any color configuration on the boundary of the
rectangle, the probability of a red path crossing a rectangle from left to right
is bounded away from $0$ and $1$, with the bounds only depending on the
rectangle's aspect ratio. This turns out to be wrong: It is expected (and known
in the case of the Ising model -- this can be deduced from \cite{benoist-hongler-cle3}, more specifically from the fact that $\CLE_{3}$, the scaling limit of the Ising interfaces, is simple) that if we condition on red boundary conditions,
then the crossing probability will converge to $1$ along a sequence of
increasing rectangles with the same aspect ratio. Another aspect which makes the
study of the fuzzy Potts model more challenging in general is the lack of a
domain Markov property (except in special cases like Potts models).

Let us mention that softer Russo-Seymour-Welsh arguments (as proven in
\cite{rsw-general-kst}) can be applied directly to the fuzzy Potts model using
its positive association (see \cite[Theorem 4.2]{haggstrom-color-percolation} for the Bernoulli case $q=1$ due to Häggström and Schramm and \cite{haggstrom-fuzzy-positive,kahn-fuzzy-fkg} for the general case $q\ge 1$)
but they yield no control of boundary effects and so we will not make use of
them here.

Our guiding principle in the study of the fuzzy Potts model will be the
following: Whenever we want to condition on the color configuration on a subset
$S\subset \Z^2$, we instead condition on the percolation configuration in $S$
and the color of clusters in the interior of $S$. Importantly, we never
condition on the color of a cluster that intersects the boundary $\partial S$.
In this way, we only obtain information about the induced FK boundary condition,
i.e.\ the partition of $\partial S$, and we can apply the domain Markov property
for FK percolation, the mixing property for FK percolation, as well as the crossing estimates of Theorem
\ref{thm:strongRSW} to extend (resp.\ block) clusters using open (resp.\ dual
open) paths.

In Section \ref{sec:fuzzy-potts}, we have introduced two closely related
versions of arm events, $A_\tau^s(m,n)$ and $A_\tau(m,n)$, which denote the
existence of $\abs{\tau}$ arms in the annulus $\Lambda_{m,n}$ with colors
prescribed by $\tau$. In the case of $A_\tau^s(m,n)$, all arms correspond to
strong paths, whereas for $A_\tau(m,n)$, red arms correspond to strong paths and
blue arms to weak paths. In this section, we choose to present most proofs for
the latter version but it will be clear that the arguments also apply to the
first version (which involves only strong paths).

\subsection{Almost-arm events: Definitions and first results} \label{sec:almost-arms}

The occurrence of $A_\tau(m,n)$ provides information about the colors at the
inner boundary $\partial \Lambda_m$ and the outer boundary $\partial \Lambda_n$,
and it is therefore convenient to introduce a variant of the arm event
$A_\tau(m,n)$ that does not include such information.

Let us consider the subgraph induced by $S \subset \R^2$ and denote its edge set
by $E(S)$. For a vertex $v \in S$ and a percolation configuration $\omega$, we
denote by $C_{v,S}(\omega)$ the $\omega$-cluster of $v$ in $S$, i.e.\ the
connected component of $v$ in the subgraph with edge set $\{e\in
E(S):\omega(e)=1\}$. For a path $\gamma \subset S$, we define its \emph{cluster
hull}
\begin{equation*}
	H_{\gamma,S}(\omega):= \bigcup_{v \in \gamma} C_{v,S}(\omega)
\end{equation*}
as the union of all $\omega$-clusters intersecting $\gamma$. We will write it
simply $H_{\gamma}(\omega)$ when no confusion can arise.

\begin{defn}[Almost-arm]
    Let $S \subset \Z^2$ and $A,B \subset \partial S$. A finite sequence of
    vertices $\gamma = (\gamma_i)_{i=0}^\ell \subset S$ with $\gamma_0 \in A$
    and $\gamma_\ell \in B$ is called a \emph{red (resp.\ blue) almost-arm} from
    $A$ to $B$ in $S$ with respect to $\Z^2$ if there exist $0 \le k_A \le k_B
    \le \ell$ such that
    \begin{enumerate}[(i)]
        \item $(\gamma_0,\ldots,\gamma_{k_A})$ is an $\omega$-open path,
            $(\gamma_{k_A},\ldots,\gamma_{k_B})$ is a strong (resp.\ weak) path,
            and $(\gamma_{k_B},\ldots,\gamma_\ell)$ is an $\omega$-open path;
        \item the vertices in
            $\widetilde{\gamma}:=(\gamma_{k_A+1},\ldots,\gamma_{k_B-1})$ are
            colored in red (resp.\ blue);
        \item the $\omega$-clusters of the colored vertices in
            $\widetilde{\gamma}$ do not intersect the boundary of $S$ with
            respect to $\Z^2$, i.e.\ $$H_{\widetilde{\gamma},S}(\omega) \cap
            \partial S = \emptyset.$$
	\end{enumerate}
    For $S \subset \Z\times \Z_+$ and $A,B \subset \partial_+ S$, we analogously
    define a \emph{red (resp.\ blue) almost-arm} from $A$ to $B$ in $S$ with
    respect to $\Z\times\Z_+$ by requiring that the $\omega$-clusters of the
    colored vertices in $\widetilde{\gamma}$ do not intersect the boundary of
    $S$ with respect to $\Z\times\Z_+$, i.e.\ $H_{\widetilde{\gamma},S}(\omega)
    \cap \partial_+ S = \emptyset$.
\end{defn}

In other words, a red almost-arm is a concatenation of an (FK) arm of type 1
starting at $A$, a (fuzzy Potts) red arm using only interior clusters, and an
(FK) arm of type 1 ending at $B$. We refer to Figure \ref{fig:almost-arm} for an
illustration.  Our definition is inspired by the notion of almost-crossing
introduced in \cite[Chapter 3]{tassion2014planarity} and the idea of considering
cluster hulls was already present in \cite{balint-dac-sharp}.

\begin{figure}
	\centering
	\def\svgwidth{0.4\columnwidth}
	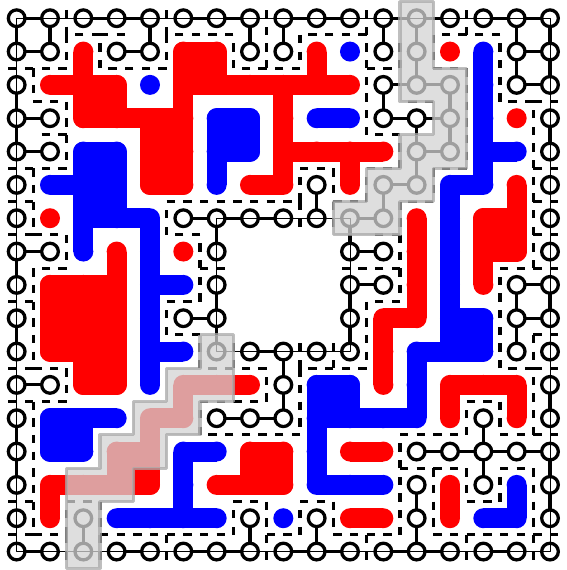
    \caption{Two examples of red almost-arms from $\partial \Lambda_m$ to
    $\partial \Lambda_n$ in $\Lambda_{m,n}$ (with respect to $\Z^2$) are
    highlighted in gray.}
	\label{fig:almost-arm}
\end{figure}

Let us describe a simple way to check if a finite sequence of vertices $\gamma = (\gamma_i)_{i=0}^\ell$ is actually a red (resp.\ blue) almost-arm:
\begin{enumerate}[(i)]
    \item Starting from $\gamma_0 \in A$, explore $\gamma$ until the first $\gamma_i\gamma_{i+1}$ that is \emph{not} an $\omega$-open edge (either because it is closed or a `diagonal step'). If $\gamma_i\gamma_{i+1}$ does not exist, $\gamma$ is an $\omega$-open path from $A$ to $B$, hence an almost-arm. Otherwise, set $k_A := i$.
	\item Starting from $\gamma_\ell \in B$, explore $\gamma$ backwards until the first $\gamma_j\gamma_{j+1}$ that is \emph{not} an $\omega$-open edge, and set $k_B := j+1$. If $(\gamma_{k_A},\ldots,\gamma_{k_B})$ is not a strong (resp.\ weak) path, $\gamma$ is not an almost-arm.
	\item For every $ v \in \widetilde{\gamma}=(\gamma_{k_A+1},\ldots,\gamma_{k_B-1})$, explore the $\omega$-cluster of $v$ in $S$. If  $C_{v,S}(\omega) \cap \partial S \neq \emptyset$ for some vertex $v \in \widetilde{\gamma}$, $\gamma$ is not an almost-arm. Otherwise, explore the colors of these clusters. If they are all red (resp.\ blue), $\gamma$ is an almost-arm.
\end{enumerate}
Consequently, the event `$\gamma$ is an almost-arm' depends only on the status of the edges in $\gamma$, the edges incident to clusters visited by $\gamma$ (except the clusters of $\gamma_0$ and $\gamma_\ell$), and the colors of the interior clusters among them. The fact that it does not depend on the color of boundary-touching clusters is the main advantage of working with almost-arms compared to arms.

In Sections \ref{sec:quasi-multiplicatity} and \ref{sec:arm-separation}, we will naturally encounter almost-arms with respect to $\Z^2$ when studying arm events under the measures $\mu_{\Z^2}$ and $P_{\Z^2}$. In Section \ref{sec:arm-separation-halfplane}, we study arm events under the halfplane measures $\mu_{\Z\times\Z_+}^0$ and $P_{\Z\times\Z_+}^0$ with free boundary conditions and in this context, we will also encounter almost-arms with respect to $\Z\times\Z^+$. Therefore, we also make the following definitions.

\begin{defn}[Almost-arm event in the plane] \label{def:almost-arm-plane}
	Let $1 \le m \le n$ and let $\tau$ be a color sequence. The almost-arm event $B_\tau(m,n)$ denotes the existence of $\abs{\tau}$ counterclockwise-ordered, disjoint almost-arms $\gamma^1,\ldots,\gamma^{\abs{\tau}}$ from $\partial \Lambda_m$ to $\partial \Lambda_n$ in $\Lambda_{m,n}$ (with respect to $\Z^2$) such that
    \begin{enumerate}[(i)]
		\item the almost-arm $\gamma^i$ has color $\tau_i$, for every $1\le i \le \abs{\tau}$;
		\item the cluster hulls $H_{\gamma^1,\Lambda_{m,n}}(\omega),\ldots, H_{\gamma^{\abs{\tau}},\Lambda_{m,n}}(\omega)$ are pairwise disjoint.
	\end{enumerate}
\end{defn}

\begin{defn}[Almost-arm events in the halfplane] \label{def:almost-arm-halfplane}
	Let $1 \le m \le n$ and let $\tau$ be a color sequence. The almost-arm event $B^{+}_\tau(m,n)$ (resp.\ $B^{++}_\tau(m,n)$) denotes the existence of $\abs{\tau}$ counterclockwise-ordered (starting with the rightmost), disjoint almost-arms $\gamma^1,\ldots,\gamma^{\abs{\tau}}$ from $\partial \Lambda_m$ to $\partial \Lambda_n$ in $\Lambda^{+}_{m,n}$ with respect to $\Z \times \Z_+$ (resp.\ $\Z^2$) such that
    \begin{enumerate}[(i)]
	\item the almost-arm $\gamma^i$ has color $\tau_i$, for every $1\le i \le \abs{\tau}$;
	\item the cluster hulls $H_{\gamma^1,\Lambda^+_{m,n}}(\omega),\ldots, H_{\gamma^{\abs{\tau}},\Lambda^+_{m,n}}(\omega)$ are pairwise disjoint.
\end{enumerate}
\end{defn}
We emphasize the difference between the two almost-arm events in the halfplane: $B^{++}_\tau(m,n)$ involves almost-arms with respect to $\Z^2$ which implies that clusters intersecting the segments $[-n,-m]\times\{0\}$ and $[m,n]\times\{0\}$ are considered as boundary-touching clusters. This is not the case for $B^{+}_\tau(m,n)$, but clearly $B^{++}_\tau(m,n)$ is a subevent of $B^{+}_\tau(m,n)$.
Furthermore, it is important to note that two almost-arms are required to have disjoint cluster hulls (even if they have the same color). As before, we remark that assuming $m \ge \abs{\tau}$ is sufficient to guarantee that  $B_\tau(m,n)$, $B^+_\tau(m,n)$ and $B^{++}_\tau(m,n)$ are non-empty.
We begin our study of almost-arms with an upper bound on the probability that there exist a red and a blue almost-arm with respect to $\Z^2$ in the upper halfplane.

\begin{prop}\label{prop:two-arm-fuzzy-potts}
For all $1 \le m \le n$,
	\begin{equation*}
		 P_{\Z^2}(B^{++}_{RB}(m,n)) \lesssim  \frac{m}{n},
	\end{equation*}
	where the bound is uniform in $r$ and $m,n$.
\end{prop}
The principal idea of the proof is standard and well-known in the case of Bernoulli percolation, where it is one part of showing that the two-arm exponent in the halfplane is universal and equal to 1. We remark that these arguments (see, e.g., \cite[p.\ 13--14]{werner-perc-notes} for an outline) can also be used to deduce
certain universal arm exponents for the fuzzy Potts model.
\begin{remark}\label{rk:universal-exponents-fuzzy-potts} Uniformly in $n \ge 1$, it holds that
		\begin{equation*}
		\mu_{\Z^2}(A^+_{RB}(1,n)) \asymp \mu_{\Z\times\Z_+}^0(A^+_{RB}(1,n)) \asymp  n^{-1} \quad \text{and} \quad \mu_{\Z^2}(A^+_{RBR}(1,n)) \asymp  n^{-2}.
	\end{equation*}
\end{remark}
Let us emphasize that the three-arm exponent ($RBR$) in the halfplane is only universal under the measure $\mu_{\Z^2}$ and so there is no contradiction with Theorem \ref{thm:main-boundary}.
We include the proof of the proposition since working with almost-arms requires additional care compared with arms. One indication for this is the (maybe surprising) observation that the standard proof of the corresponding lower bound does not apply to the almost-arm event $B^{++}_{RB}(m,n)$. In any case, the upper bound stated above will be sufficient for our purposes.
\begin{figure}
	\centering
	\def\svgwidth{0.9\columnwidth}
	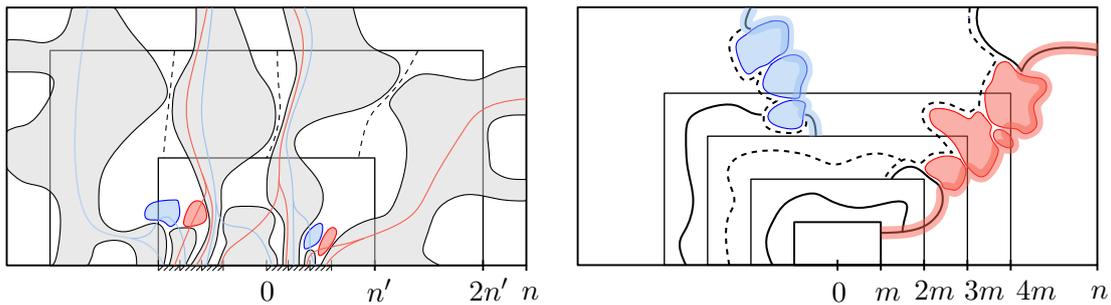
    \caption{Illustrations of the proof of Proposition \ref{prop:two-arm-fuzzy-potts}. \emph{Left.} Six translated versions of the event $B'(m,n)$ occur ($\abs{X}=6$); inducing three disjoint dual open paths that cross the semi-annulus $\Lambda^+_{n',2n'}$. \emph{Right.} Construction of the event $B'(4m,n)$ from the event $B(m,n)$. The almost-arms $\Gamma_R$ and $\Gamma_B$ are highlighted. }
    \label{fig:illustration-2-almost-arm}
\end{figure}

\begin{proof}
	Throughout the proof, we abbreviate $P_{\Z^2}$ by  $P$ and $B^{++}_{RB}(m,n)$ by $B(m,n)$.
    
    Let us define a variant of the event $B(m,n)$ by considering $\Lambda_n^+$ instead of $\Lambda_{m,n}^+$ and by requiring the two almost-arms (with respect to $\Z^2$) to start from vertices on the line segment $[-m,m]\times\{0\}$, and  denote it by $B'(m,n)$.

    \emph{Step 1: $P(B'(m,n)) \lesssim m/n$.}

	We set $n':=\lfloor n/3 \rfloor$ and we may assume $n' \ge 4m$ without loss of generality. Let $N_{n'}$ be the number of disjoint dual open paths that cross the semi-annulus $\Lambda_{n',2n'}^+$ from $\partial \Lambda_{n'}$ to $\partial \Lambda_{2n'}$.

    First, let us argue that
	\begin{equation}\label{eq:almost-arm-upper-bound}
		\sum_{x \in 2m\Z\, :\,  \abs{x}\le n'-m } 1_{(x,0) + B'(m,n)} \le 2 \cdot N_{n'},
	\end{equation}
	where we use the notation $(x,y)+F$ to denote the translate of the event $F$ by the vector $(x,y) \in \Z^2$.
	In words, it says that if there are many pairs of $RB$ almost-arms in $\Lambda_{2n'}^+$ from $[-n',n']\times \{0\}$ to $\partial \Lambda_{2n'}$, then this induces many dual open paths crossing the semi-annulus $\Lambda_{n',2n'}^+$ from the inner to the outer boundary. For the reader familiar with the proof in the case of Bernoulli percolation, we mention that the factor $2$ on the right-hand side is necessary since we work with almost-arms (see the left of Figure \ref{fig:illustration-2-almost-arm}).

	Let $X \subset \{x \in 2m\Z\, :\,  \abs{x}\le n'-m \}$ be the set of points for which the event  $(x,0) + B'(m,n)$ occurs . For $x \in X$, let $\gamma_{B,x}$ denote the leftmost blue almost-arm  in $\Lambda_{2n'}^+$ from $[x-m,x+m]\times\{0\}$ to $\partial \Lambda_{2n'}$. Clearly, there exists a red almost-arm $\gamma_{R,x}$ to the right of $\gamma_{B,x}$ that also crosses $\Lambda_{2n'}^+$ and so that the cluster hulls are disjoint. Hence, the dual open path $\gamma_{d,x}$ which follows the right boundary of the cluster hull of $\gamma_{B,x}$ is sandwiched between $\gamma_{B,x}$ and $\gamma_{R,x}$.
	While arms of different colors cannot cross each other, this might happen for almost-arms. Indeed, the uncolored segments at the beginning (resp.\ the end) which correspond to $\omega$-open paths might intersect (this is exactly why we have insisted on disjoint cluster hulls in Definitions \ref{def:almost-arm-plane} and \ref{def:almost-arm-halfplane}).
	Therefore, $x<y \in X$ does not imply that $\gamma_{R,x}$ stays to the left of $\gamma_{B,y}$ (if true, it would have allowed us to deduce directly that $\gamma_{d,x}$ is strictly to the left of $\gamma_{d,y}$). 
	Instead, we will show that for all $x<y<z \in X$, the dual open path $\gamma_{d,x}$ is strictly to the left of $\gamma_{d,z}$, which then readily implies $N_{n'} \ge \abs{X}/2$, and so equation \eqref{eq:almost-arm-upper-bound} follows.
	
	It suffices to argue that the cluster hulls of $\gamma_{B,x}$ and $\gamma_{B,z}$ must be disjoint. To this end, denote by $C_{B,x}^1,\ldots,C_{B,x}^i$ the clusters of $\gamma_{B,x}$, ordered from the one touching $[x-m,x+m]\times\{0\}$ to the one touching $\partial \Lambda_{2n'}$. Analogously, denote by $C_{R,x}^1,\ldots,C_{R,x}^j$ the clusters of $\gamma_{R,x}$ and by $C_{B,z}^1,\ldots,C_{B,z}^k$ the clusters of $\gamma_{B,z}$. First, we observe that the existence of $\gamma_{d,y}$ implies that $C_{B,x}^1$ cannot touch $[z-m,z+m]\times\{0\}$ and $C_{B,z}^1$ cannot touch $[x-m,x+m]\times\{0\}$. Consequently, the boundary-touching clusters $C_{B,x}^1$, $C_{R,x}^1$, and $C_{B,z}^1$ are disjoint. Second, the existence of $\gamma_{R,x}$, whose cluster hull lies to the right and is disjoint from the one of $\gamma_{B,x}$, implies that the clusters  $C_{B,x}^i$ and  $C_{B,z}^k$ must be disjoint. Indeed, if not, the clusters $C_{B,x}^1,\ldots,C_{B,x}^i = C_{B,z}^k, \ldots, C_{B,z}^1$ would form a semicircuit from $[x-m,x+m]$ to $[z-m,z+m]$ separating $C_{R,x}^1$ from $\partial \Lambda_{2n'}$. The same argument also implies that the intermediate clusters of $\gamma_{B,x}$ and $\gamma_{B,z}$ must be disjoint.

	Taking expectation on both sides in equation \eqref{eq:almost-arm-upper-bound}, we get
	\begin{equation}
		\frac{n'}{2m} \cdot  P \left(B'(m,n)\right) \le \left(\frac{n'}{m}-2\right)\cdot  P \left(B'(m,n)\right) \le 2 \cdot \phi_{\Z^2} \left(N_{n'}\right).
	\end{equation}
	To complete Step 1, it is now sufficient to observe $\phi_{\Z^2} (N_{n'}) \le C < \infty$, which is a standard consequence of the domain Markov property and the crossing estimates in Theorem \ref{thm:strongRSW}.
	
    \emph{Step 2: $P(B(m,n)) \lesssim  P(B'(4m,n))$.}

	In this step, we show that the probabilities of $B(m,n)$ and $B'(m,n)$ are comparable. One direction follows readily from the inclusion $B'(m,n) \subset B(m,n)$ and we will now establish the other direction.
	Let us work on the event $B(m,n)$. In this case, we can introduce the following notation (see also the right of Figure \ref{fig:illustration-2-almost-arm}):
    \begin{enumerate}[(i)]
		\item Let $\Gamma_R$ be the rightmost red almost-arm in $\Lambda_{m,n}^+$ from $\partial\Lambda_{m}$ to $\partial\Lambda_{n}$, and let $E(\Gamma_R)$ denote the set of edges with respect to which $\Gamma_R$ is measurable, i.e.\ all edges in $\Lambda_{m,n}^+$ that are on or to the right of $\Gamma_R$ and all edges that are
		incident to interior clusters visited by $\Gamma_R$.
		\item Let $\Gamma_{d}$ be the rightmost dual open simple path in $\Lambda_{2m,n}^+$ from $\partial\Lambda_{2m}$ to $\partial\Lambda_{n}$ that is to the left of $\Gamma_R$, and let $E(\Gamma_{d})$ denote all primal edges in $\Lambda_{2m,n}^+$ that intersect or are to the right of $\Gamma_{d}$.
		\item Let $\Gamma_B$ be the rightmost blue almost-arm in $\Lambda_{3m,n}^+$ from $\partial\Lambda_{3m}$ to $\partial\Lambda_{n}$ that is to the left of $\Gamma_{d}$, and let $E(\Gamma_B)$ denote the set of edges with respect to which $\Gamma_B$ is measurable, i.e.\ all edges in $\Lambda_{3m,n}^+$ that are on or to the right of $\Gamma_B$ and all edges that are incident to interior clusters visited by $\Gamma_B$.
	\end{enumerate}
	Let us condition on $\Gamma_R, \Gamma_d, \Gamma_B$, the status of the edges in $E(\Gamma_R)\cup E(\Gamma_{d}) \cup E(\Gamma_B)$ and the colors of the interior clusters in the subgraph induced by this edge set.

    As shown in right of Figure \ref{fig:illustration-2-almost-arm}, the occurrence of the event $B'(4m,n)$ can now be guaranteed by the existence of an open path in $\Lambda_{m,2m}^+ \setminus E(\Gamma_R)$, a dual open path in $\Lambda_{2m,3m}^+ \setminus E(\Gamma_d)$ and an open path in $\Lambda_{3m,4m}^+ \setminus E(\Gamma_B)$, and we denote by $F(m,4m)$ the event that such paths exist.

    Furthermore, the domain Markov property implies that the configuration outside of $E(\Gamma_R)\cup E(\Gamma_{d}) \cup E(\Gamma_B)$ depends on the conditioning only through the induced FK boundary conditions. Therefore, using that the extremal distances of the considered discrete domains
    are bounded below by the extremal distance of $\Lambda_{3m,4m}^+$ (from $[-4m,-3m]\times\{0\}$ to $[3m,4m]\times\{0\}$), the strong crossing estimates of Theorem \ref{thm:strongRSW} imply	that the conditional probability of $F(m,4m)$ is bounded below by a constant $\epsilon>0$ which does not depend on $m,n$, on $r$, or on the configuration we have conditioned on.\footnote{\,More precisely, let $\mathcal{D}=(V,E)$ be the largest discrete domain in $\Lambda_{m,2m}^+ \setminus E(\Gamma_R)$ (where we allow $\partial \mathcal{D}$ to use edges of $\gamma_R$ that correspond to the open paths touching the boundary of $\Lambda_{m,n}^+$). We then consider the approximate discrete domain $\mathcal{D}'$ obtained from it by adding all edges in $\Lambda_{m,2m}^+ \setminus E(\Gamma_R)$ with exactly one endpoint in $V$. One can then check that an open path from $[-2m,-m]\times\{0\}$ to the segment of $\partial \mathcal{D}'$ neighboring $E(\Gamma_R)$ guarantees the existence of a red almost-arm from $[-2m,-m]\times\{0\}$ to $\partial \Lambda_n$. The approximate discrete domains in $\Lambda_{2m,3m}$ and $\Lambda_{3m,4m}$ can be defined similarly.} By plugging this lower bound into the previous equation, we get $P(B'(4m,n)) \ge \epsilon \cdot P(B(m,n))$ which completes the second step and thereby the proof.
\end{proof}

The following useful corollary is a standard consequence of the previous
proposition. We refer to Figure \ref{fig:3-almost-arm-halfplane} for an illustration of the main idea of the proof.

\begin{cor}\label{cor:three-arm-fuzzy-potts}
	Let $\tau  \notin \{RRR,BBB\}$ be a color sequence of length $3$. There exists a constant $\beta_2=\beta_2(q)>0$  such that for all $1 \le m \le n$,
	\begin{equation*}
		 P_{\Z^2}(B^{++}_{\tau}(m,n)) \lesssim  \left(\frac{m}{n}\right)^{1+\beta_2},
	\end{equation*}
	where the bound is uniform in $r$ and $m,n$.
\end{cor}
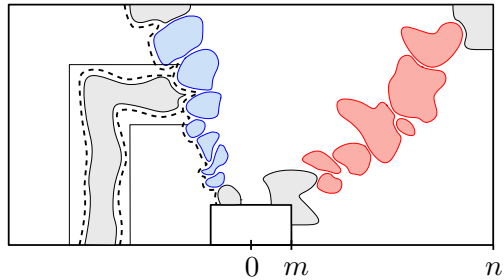
\begin{figure}
	\centering
	\def\svgwidth{0.4\columnwidth}
\begingroup%
  \makeatletter%
  \providecommand\color[2][]{%
    \errmessage{(Inkscape) Color is used for the text in Inkscape, but the package 'color.sty' is not loaded}%
    \renewcommand\color[2][]{}%
  }%
  \providecommand\transparent[1]{%
    \errmessage{(Inkscape) Transparency is used (non-zero) for the text in Inkscape, but the package 'transparent.sty' is not loaded}%
    \renewcommand\transparent[1]{}%
  }%
  \providecommand\rotatebox[2]{#2}%
  \newcommand*\fsize{\dimexpr\f@size pt\relax}%
  \newcommand*\lineheight[1]{\fontsize{\fsize}{#1\fsize}\selectfont}%
  \ifx\svgwidth\undefined%
    \setlength{\unitlength}{385.50161438bp}%
    \ifx\svgscale\undefined%
      \relax%
    \else%
      \setlength{\unitlength}{\unitlength * \real{\svgscale}}%
    \fi%
  \else%
    \setlength{\unitlength}{\svgwidth}%
  \fi%
  \global\let\svgwidth\undefined%
  \global\let\svgscale\undefined%
  \makeatother%
  \begin{picture}(1,0.58445751)%
    \lineheight{1}%
    \setlength\tabcolsep{0pt}%
    \put(0,0){\includegraphics[width=\unitlength,page=1]{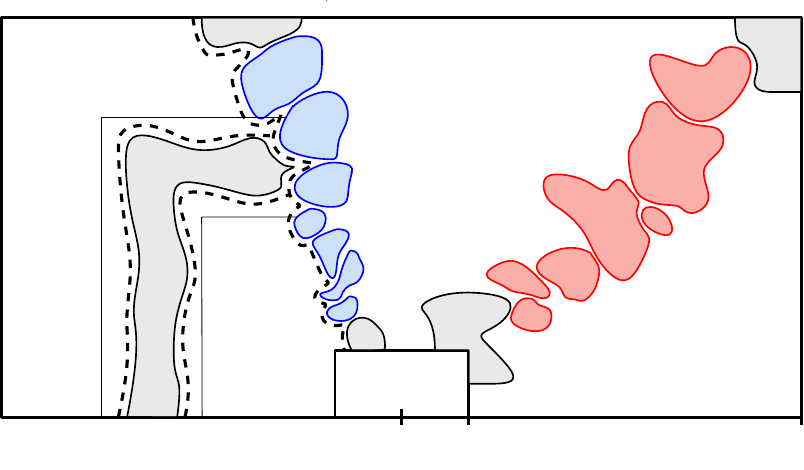}}%
    \put(0.48723121,0.00030091){\makebox(0,0)[lt]{\lineheight{1.25}\smash{\begin{tabular}[t]{l}$0$\end{tabular}}}}%
    \put(0.56680706,0.00309457){\makebox(0,0)[lt]{\lineheight{1.25}\smash{\begin{tabular}[t]{l}$m$\end{tabular}}}}%
    \put(0.98151663,0.00248774){\makebox(0,0)[lt]{\lineheight{1.25}\smash{\begin{tabular}[t]{l}$n$\end{tabular}}}}%
  \end{picture}%
\endgroup%

	\caption{Conditional on the exploration of the rightmost red almost-arm and the rightmost blue almost-arm to its left,  the probability of an additional almost-arm in the remaining domain is polynomially small (using the strong crossing estimates of Theorem \ref{thm:strongRSW} and the mixing property for FK percolation of Corollary \ref{cor:mixing} in dyadic semi-annuli). The bound is uniform in $r$ since the existence of a boundary-touching FK cluster (as drawn) suffices to prevent an additional almost-arm.}
	\label{fig:3-almost-arm-halfplane}
\end{figure}

\begin{remark}
	To obtain the same upper bound for the almost-arm event corresponding to $A_\tau^{+s}(m,n)$, denoted by $B^{++s}_{\tau}(m,n)$, it suffices to note that $B^{++s}_{\tau}(m,n) \subset B^{++}_{\tau}(m,n)$.
\end{remark}

\subsection{Quasi-multiplicativity} \label{sec:quasi-multiplicatity}

In this subsection, we prove quasi-multiplicativity for the fuzzy Potts model, which is the main `discrete' ingredient needed to establish our main results.

\begin{thm}[Quasi-multiplicativity] \label{thm:quasi-multiplicativity}
	Let $\tau$ be an alternating color sequence. For all $\abs{\tau} \le \ell \le m \le n$ we have
	\begin{equation*}
	\mu_{\Z^2}(A_\tau(\ell,n)) \asymp \mu_{\Z^2}(A_\tau(\ell,m)) \cdot  \mu_{\Z^2}(A_\tau(m,n)),
	\end{equation*}
	where the bounds are uniform in $\ell,m,n$. The same holds true for the arm event version $A^{s}_\tau$.
\end{thm}

We present all proofs in this section for the arm event version $A_\tau$ but the same reasoning also applies to the arm event version $A_\tau^s$. We will always work with the infinite-volume measures $P_{\Z^2}$ and $\mu_{\Z^2}$, so we also drop $\Z^2$ from the notation and simply write $P$ and $\mu$.

The quasi-multiplicativity of arm events is usually obtained by first proving that arms can be separated, chosen to land at prescribed boundary segments and extended.

This strategy has been developed for Bernoulli percolation in \cite{Kesten1987} and it has been extended to FK percolation in \cite{Chelkak2016} and \cite{duminil2021planar}. As we will see in the following subsection, arm separation in itself is a useful tool, e.g.\ to compare arm exponents in the plane.
Our proof follows closely the well-known approach of Kesten \cite{Kesten1987} (see also \cite{Nolin2008,Manolescu2012}) and we have also been inspired by the presentation in \cite{DuminilCopin2021-scaling-without}. The main novelty is the use of almost-arm events, which allows to extend this approach to the fuzzy Potts model.

We begin with introducing some notation. Fix $\beta := \beta_1\wedge \beta_2\wedge
(1/2)$ according to Corollaries \ref{cor:six-arms-FK-with-bc} and  \ref{cor:three-arm-fuzzy-potts}.  We  say that a finite sequence $I^{(s)} \subset \R^2$ of counterclockwise-ordered points with $\norm{x}_\infty = s$, i.e.\ on the boundary of the square $[-s,s]^2$, is \emph{$\delta$-separated} for $\delta >0$ if the points in $I^{(s)}$ are at distance at least $2\delta^{\beta/2}s$ from each other and at distance at least $\delta^{\beta/4}s$ from the four corners of the square.

Let $x \in \R^2$  with $\norm{x}_\infty = s$. For $t >0$, we denote by $R^{in}_{t}(x)$ the rectangle $x + [-t,0] \times [0,2t]$ (resp.\ $x + [-2t,0] \times [-t,0]$, $x + [0,t] \times [-2t,0]$, $x + [0,2t] \times [0,t]$) if $x$ is on the right (resp.\ top, left, bottom) side of the boundary of the square $[-s,s]^2$. These rectangles are contained in $\Lambda_s$, have one corner at $x$ and `stretch' from there in counterclockwise order.
We denote by $R^{out}_{t}(x)$ the analogous rectangle that is positioned outside of $\Lambda_s$. It is obtained by interchanging  $[0,t]$ and $[-t,0]$ in the previous definition.
Finally, we denote by $S^{in}_{t}(x)$ and $S^{out}_{t}(x)$ the slightly translated versions of $R^{in}_{t}(x)$ and $R^{out}_{t}(x)$ that are obtained if we replace $[0,2t]$ by $[2t,4t]$ and $[-2t,0]$ by $[-4t,-2t]$ in the previous definitions.

\begin{figure}
	\centering
	\def\svgwidth{0.9\columnwidth}
    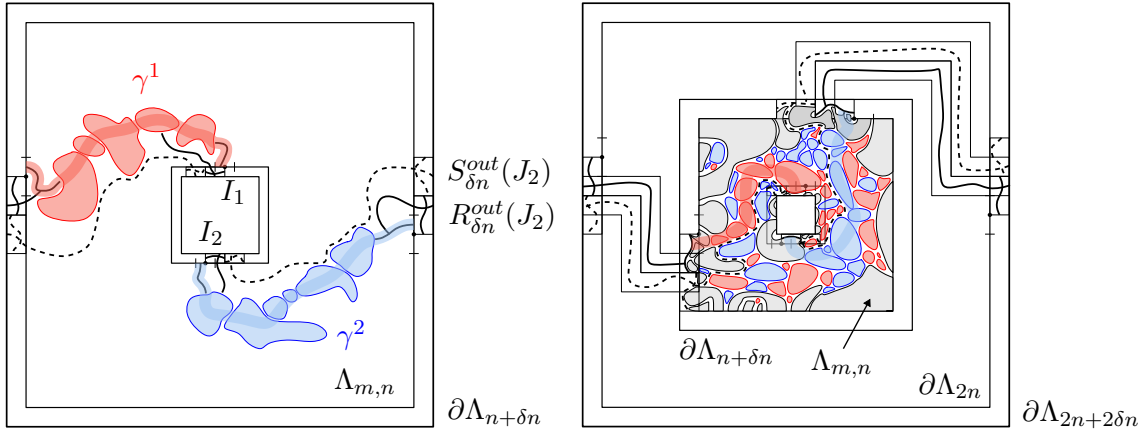
    \caption{\emph{Left.} The well-separated almost-arm event with extensions $\hat{B}^{\delta,I,J}_{RB}(m,n)$. The chosen almost-arms $\gamma^1$ and $\gamma^2$ are highlighted. \emph{Right.} Illustrations of the proof of Proposition \ref{prop:extendability}. The extensions $\hat{\gamma}^1$ and $\hat{\gamma}^2$ are highlighted. }
    \label{fig:separation-extension-arm}
\end{figure}

\begin{defn}\label{def:well-seperated-almost-arm-extensions}
	Let $1 \le m \le n$ and $\tau=\tau_1\cdots\tau_k$ be a color sequence. Let $\delta >0$ and let $I^{(m)}=(I_1,\ldots,I_k)$, $J^{(n)}=(J_1,\ldots,J_k)$ be $\delta$-separated. The \emph{well-separated almost-arm event with extensions} $\hat{B}^{\delta,I,J}_\tau(m,n)$ denotes the subset of $B_\tau (m,n)$ for which the almost-arms $\gamma^i$ for $1\le i \le k$ can be chosen as follows:
    \begin{enumerate}[(i)]
		\item The almost-arm $\gamma^i$ starts at some $x_i \in \partial \Lambda_m$ with $d_\infty(x_i,I_i) \le \delta m$ and ends at some  $y_i \in \partial \Lambda_n$ with $d_\infty(y_i,J_i) \le \delta n$.
		\item There is an almost-arm $\hat{\gamma}^i$ in $\Lambda_{m,n} \cup R^{in}_{\delta m}(I_i) \cup R^{out}_{\delta n}(J_i)$ from $\partial \Lambda_{m - \delta m}$ to $\partial \Lambda_{n + \delta n}$ whose colored vertices belong to the interior clusters of $\gamma^i$ (we call $\hat{\gamma}^i$ an extension of $\gamma^i$) and the rectangles $ R^{in}_{\delta m}(I_i)$ and $ R^{out}_{\delta n} (J_i)$ are crossed in the long direction.
		\item There is a dual open simple path $\gamma_d^i$ in $\Lambda_{m,n} \cup  S^{in}_{\delta m}(I_i) \cup S^{out}_{\delta n}(J_i)$ from $\partial \Lambda_{m - \delta m}$ to $\partial \Lambda_{n + \delta n}$ that separates $\hat{\gamma}^i$ from $\hat{\gamma}^{i+1}$ and the rectangles $ S^{in}_{\delta m}(I_i)$ and $ S^{out}_{\delta n}(J_i)$ are dual crossed in the long direction.
	\end{enumerate}
\end{defn}
We refer to the left of Figure \ref{fig:separation-extension-arm} for an illustration. Importantly, extensions of almost-arms use only $\omega$-open edges in $\Lambda_{m-\delta m,m}$ and $\Lambda_{n, n + \delta n}$, and this allows for the relevant gluing constructions.

\begin{prop}[Extendability for $\hat{B}_\tau$] \label{prop:extendability}
	Let $\tau$ be a color sequence and let $\delta>0$. For all $\abs{\tau} \le m \le n$ and for all $I'^{(m/2)}$, $I^{(m)}$, $J^{(n)}$, $J'^{(2n)}$ $\delta$-separated,
	\begin{align*}
		&P\left(\hat{B}_\tau^{\delta,I,J}(m,n)\right) \lesssim P\left(\hat{B}_\tau^{\delta,I',J}(m/2,n)\right), \\
		&P\left(\hat{B}_\tau^{\delta,I,J}(m,n)\right) \lesssim P\left(\hat{B}_\tau^{\delta,I,J'}(m,2 n)\right),
	\end{align*}
	uniformly in $I', I, J, J'$ and in $m,n$.
\end{prop}

\begin{proof}
	The two bounds can be proven in the same way and we only present the argument for the second bound here.
	The idea is to use the extensions of the almost-arms in $\Lambda_{n, n + \delta n}$ to further extend them to the boundary segments of $\partial \Lambda_{2n}$ prescribed by $J'^{(2n)}$. To this end, we assume that the event $ \hat{B}_\tau^{\delta,I,J}(m,n)$ occurs and we condition on the percolation configuration $\omega$ in $\Lambda_{m-\delta m,n}$ as well as the colors of the interior clusters. For each $1 \le i \le \abs{\tau}$, we now want to explore a suitable extension of $\gamma^i$. We refer to the right of Figure \ref{fig:separation-extension-arm} for an illustration. For simplicity, assume that $J_i$ is a point on the right side of $[-n,n]^2$. In the rectangle $R_i :=  R^{out}_{\delta n}(J_i)$, we explore the percolation configuration $\omega$ to find the leftmost crossing $\lambda$ from bottom to top. In this way, the edges to the right of $\lambda$ remain unexplored, and by the definition of $\hat{B}_\tau^{\delta,I,J}(m,n)$, the extension $\hat{\gamma}^i$ must intersect $\lambda$.  Analogously, we can explore the leftmost dual crossing $\lambda_d$ in $S_i := S_{\delta n}^{out}(J_i)$ which by definition intersects the dual open path $\gamma_d^i$. It becomes clear from the right of Figure \ref{fig:separation-extension-arm} that  $2\abs{\tau}$ suitable tubes of thickness $\delta n$ can be placed in $\Lambda_{n,2n+2\delta n}$ such that for each $1 \le i \le \abs{{\tau}}$, the almost-arm $\hat{\gamma}^i$ and the dual open path $\gamma_d^i$ can be extended inside the tubes to the boundary segment nearby $J'_i$ in order to guarantee the occurrence of the event $\hat{B}_\tau^{\delta,I,J'}(m,2n)$. Under the previous conditioning, the probability of the open resp.\ dual open crossings inside the tubes can be bounded from below using Theorem \ref{thm:strongRSW}. This shows for some $\epsilon'=\epsilon'(\delta,\tau)>0$,
	\begin{equation*}
		P\left(\left.\hat{B}_\tau^{\delta,I,J'}(m,2n) \right\rvert \hat{B}_\tau^{\delta,I,J}(m,n)\right) \ge \epsilon' ,
	\end{equation*}
	and thereby concludes the proof.
\end{proof}

The bounds corresponding to the other direction in Proposition \ref{prop:extendability} can also be proven. In fact, much more is true: $\delta$  can be chosen so that the probabilities of  arm events and well-separated almost-arm events with extensions are comparable. This is the most technical result in this section and we postpone the proof to Section \ref{sec:arm-separation}.

\begin{thm}[Arm separation] \label{thm:arm-separation}
	Let $\tau$ be an alternating color sequence. There exists $\delta_0 >0$ such that for all $\abs{\tau} \le m \le n$, for all $\delta \in (0,\delta_0)$
	 and for all $I^{(m)}$, $J^{(n)}$ $\delta$-separated,
	\begin{equation*}
		P\left(\hat{B}_\tau^{\delta,I,J}(m,n)\right) \asymp  \mu\left(A_\tau(m,n)\right) \asymp P\left(B_\tau(m,n)\right),
	\end{equation*}
	uniformly in $I, J$ and in $m,n$.
\end{thm}
\begin{remark}\label{rk:arm-separation-non-alternating}
	The theorem is stated under the assumption that $\tau$ is alternating but this is only to ensure that on the event $A_\tau(m,n)$, all arms have disjoint cluster hulls. More precisely
	\begin{equation*}
		P\left(\hat{B}_\tau^{\delta,I,J}(m,n)\right) \asymp P\left(B_\tau(m,n)\right)
	\end{equation*}
    for any color sequence $\tau$ as the cluster hulls of almost-arms are disjoint by definition.
\end{remark}
\begin{remark}
	As for every other result in this section, arm separation also holds true for the arm event version $A^s_\tau(m,n)$ and the corresponding versions of the almost-arm events.
\end{remark}

By combining the previous two results, we obtain the following direct corollary.

\begin{cor}[Extendability for $A_\tau$ and $B_\tau$] \label{cor:extendability}
	Let $\tau$ be an alternating color sequence. For all $\abs{\tau} \le m \le n$,
	\begin{align*}
		\mu\left(A_\tau(m,n)\right) \asymp \mu\left(A_\tau(m/2,2n)\right) \quad \text{and}
		\quad P\left(B_\tau(m,n)\right) \asymp P\left(B_\tau(m/2,2n)\right),
	\end{align*}
	uniformly in $m,n$.
\end{cor}

We are now ready to prove quasi-multiplicativity.

\begin{proof}[Proof of Theorem \ref{thm:quasi-multiplicativity}]
Fix $\delta > 0$ for which Theorem \ref{thm:arm-separation} applies.
First, we prove the upper bound
\begin{equation*}
	\mu(A_\tau(\ell,n)) \lesssim \mu(A_\tau(\ell,m)) \cdot  \mu(A_\tau(m,n)).
\end{equation*}
In the case of Bernoulli percolation, it follows trivially from independence. Here, we apply the mixing property for FK percolation to almost-arm events, which is possible since conditioning on an almost-arm event provides no information about the colors on the boundary but only about the FK boundary condition. More precisely,
\begin{align*}
	\mu\left(A_\tau(\ell,n)\right) &\le P\left(B_\tau(\ell,m/2) \cap B_\tau(m,n)\right) \asymp  P\left(B_\tau(\ell,m/2)  \right)\cdot  P\left(B_\tau(m,n)\right) \\
	& \asymp \mu\left(A_\tau(\ell,m/2)\right) \cdot \mu\left(A_\tau(m,n)\right)  \asymp \mu\left(A_\tau(\ell,m)\right) \cdot \mu\left(A_\tau(m,n)\right),
\end{align*}
where we have used the mixing property (Corollary \ref{cor:mixing}) in the second comparison as well as arm separation and extendability in the third and fourth comparison.

To prove the lower bound, we fix any  $I^{(\ell)}$, $J^{(m/2)}$, $I'^{(m)}$, $J'^{(n)}$ $\delta$-separated. Then,
\begin{align*}
	\mu\left(A_\tau(\ell,m)\right) \cdot \mu\left(A_\tau(m,n)\right) &\asymp
	P\left(\hat{B}_\tau^{\delta,I,J}(\ell,m/2)\right) \cdot 	P\left(\hat{B}_\tau^{\delta,I',J'}(m,n)\right) \\
	&\asymp P\left(\hat{B}_\tau^{\delta,I,J}(\ell,m/2) \cap  \hat{B}_\tau^{\delta,I',J'}(m,n)\right),
\end{align*}
where we have applied arm separation and extendability in the first comparison and the mixing property (Corollary \ref{cor:mixing}) in the second comparison. Now, one can use the similar gluing constructions as in the proof of Proposition \ref{prop:extendability} to connect the almost-arms (resp.\ dual open paths) in the inner annulus $\Lambda_{\ell,m/2}$ with the corresponding almost-arms (resp.\ dual open paths) in the outer annulus $\Lambda_{m,n}$. By coloring in red resp.\ blue the $\abs{\tau}$ interior clusters intersecting the intermediate annulus $\Lambda_{m/2,m}$ that have been created to connect the almost-arms, one can guarantee the occurrence of the event $\hat{B}_\tau^{\delta,I,J'}(\ell,n)$. We note that the coloring step has probability at least $(r(1-r))^{\abs{\tau}}$. Hence, we have just argued that
\begin{equation*}
	P\left(\hat{B}_\tau^{\delta,I,J}(\ell,m/2) \cap  \hat{B}_\tau^{\delta,I',J'}(m,n)\right) \lesssim P\left(\hat{B}_\tau^{\delta,I,J}(\ell,n)\right).
\end{equation*}
By plugging this bound into the previous comparison, we conclude  that
\begin{equation*}
	\mu\left(A_\tau(\ell,m)\right) \cdot \mu\left(A_\tau(m,n)\right) \lesssim  P\left(\hat{B}_\tau^{\delta,I,J}(\ell,n)\right) \asymp \mu\left(A_\tau(\ell,n)\right),
\end{equation*}
where we have once more used arm separation in the last comparison.
\end{proof}

\subsection{Arm separation}
\label{sec:arm-separation}

In this section, we present the proof of Theorem \ref{thm:arm-separation}. The
very rough idea of the proof is to apply the a priori bound on the three-arm event
appearing in Corollary \ref{cor:three-arm-fuzzy-potts} at each scale to control the
probability of two arms landing close to each other. By iterating over the scales, it is possible to deduce that almost-arm events with extensions are comparable up to constants with arm events.

\begin{proof}[Proof of Theorem \ref{thm:arm-separation}]
First, we note that the occurrence of the arm event $A_\tau(m,n)$ implies the occurrence of the almost-arm event $B_\tau(m,n)$ since the color sequence $\tau$ is alternating. Second, if the event $\hat{B}_\tau^{\delta,I,J}(m,n)$ occurs, then one can use a similar gluing construction as in the proof of Proposition \ref{prop:extendability} to connect the dual open paths $\gamma_d^i$, $1 \le i \le \abs{\tau}$, to a dual open circuit in $\Lambda_{m-2\delta m, m-\delta m}$ and to a dual open circuit in $\Lambda_{n+ \delta n, n + 2\delta n}$. Conditional on the event $\hat{B}_\tau^{\delta,I,J}(m,n)$, this construction has probability bounded below by some $\epsilon'=\epsilon'(\delta,\tau)>0$ according to Theorem \ref{thm:strongRSW} and it ensures that the boundary-touching clusters (in $\Lambda_{m,n}$) of different almost-arms $\gamma^i, 1 \le i \le \abs{\tau}$,  are disjoint.
By coloring the boundary-touching clusters of $\gamma^i$ with color $\tau_i$ for each $1 \le i \le \abs{\tau}$, one can guarantee the occurrence of the arm event $A_\tau(m,n)$. This coloring step has probability at least $(r(1-r))^{2\abs{\tau}}$. In summary, we have obtained
	\begin{equation*}
		P\left(\hat{B}_\tau^{\delta,I,J}(m,n)\right) \lesssim \mu\left(A_\tau(m,n)\right) \le P\left(B_\tau(m,n)\right).
	\end{equation*}

	It remains to prove that there exists $\delta_0 >0$ such that for every $\delta \in (0,\delta_0)$ and every  $I^{(m)}$, $J^{(n)}$ $\delta$-separated,
	\begin{equation}\label{eq:arm-separation-main-step}
		P\left(B_\tau(m,n)\right) \lesssim P\left(\hat{B}_\tau^{\delta,I,J}(m,n)\right).
	\end{equation}
	Since the separation and extension of the almost-arms at the inner boundary $\partial \Lambda_m$ resp.\ the outer boundary $\partial \Lambda_n$ follows from the same arguments, we focus on the latter. More precisely, we want to show that for every $\delta$-separated $J^{(n)}$,
	\begin{equation}\label{eq:arm-separation-reduced}
		P\left(B_\tau(m,n)\right) \lesssim P\left(\hat{B}_\tau^{\delta,J}(m,n)\right),
	\end{equation}
	where we write $\hat{B}_\tau^{\delta,J}(m,n)$ for the well-separated almost-arm event with extensions, which only requires separation and extensions at the outer boundary $\partial \Lambda_n$.

	Throughout the rest of the proof, we assume that $n$ is of the form $n = 2^{k} m$ for some $k\ge 1$. Once we have established the result for $n$ of this form, the general case follows from the gluing constructions as in the proof of Proposition \ref{prop:extendability}.
	From now onward, it will be important to keep track which constants depend on $\delta$. Therefore, we write $\epsilon_i = \epsilon_i(\delta,\tau), i\ge 1$ for constants depending on $\delta$ and $c_i=c_i(\tau),i\ge 1$ for constants not depending on $\delta$.

	We begin with the following generalization of Proposition \ref{prop:extendability} (see Fig.\ 9 and the accompanying text in \cite{Kesten1987} for the first appearance of this type of inequality): There exist $\epsilon_1>0$ and $c_1>0$  such that for every $1 \le i \le\log(n/m)$ and for any $\delta$-separated $J,J'$,
	\begin{equation}\label{eq:arm-separation-apriori-bound}
		P\left(\hat{B}_\tau^{\delta,J}(m,n)\right) \ge \epsilon_1  \cdot (c_1)^{i} \cdot   P\left(\hat{B}_\tau^{\delta,J'}(m,n/2^i)\right).
	\end{equation}
	This lower bound can be obtained by iterative extension of the almost-arms as in the proof of Proposition \ref{prop:extendability}. Importantly, the dependence on $\delta$ is only via the constant $\epsilon_1$ because only the first and the last extension require tubes of size $\delta$ and for the intermediate extensions, it is possible to work with tubes of some fixed size $\delta'$ (only depending on $\tau$). We note that equation \eqref{eq:arm-separation-apriori-bound} implies equation \eqref{eq:arm-separation-reduced} when the ratio $n/m$ is bounded by a constant.

	Let us introduce two intermediate types of almost-arm events.
	First, $B^{\delta,J}_\tau(m,n)$ denotes the subset of $B_\tau (m,n)$ such that for every $1\le i \le \abs{\tau}$, the almost-arm $\gamma^i$ can be chosen to end at some  $y_i \in \partial \Lambda_n$ with $d_\infty(y_i,J_i) \le \delta n$ and there is a dual open simple path $\gamma_d^i$ in $\Lambda_{m,n}$ from $\partial \Lambda_{m}$ to $\partial \Lambda_{n}$ that ends next to $y_i$ (in counterclockwise order). We refer to the left part of Figure \ref{fig:almost-arm-events} for an illustration.
	Second, for an almost-arm $\gamma$ in $\Lambda_{m,n}$, let us denote by $\tilde{H}_\gamma$ the cluster hull of its interior clusters. The event $\tilde{B}^{\delta,J}_\tau(m,n)$ is the subset of $B^{\delta,J}_\tau(m,n)$ such that for any $1 \le i \neq j \le \abs{\tau}$,
	\begin{equation*}
		(\gamma^i \cup \tilde{H}_{\gamma^i}) \cap (J_j + \Lambda_{4\delta n}) = \emptyset,
	\end{equation*}
	where $\gamma^i$ is the almost-arm ending at $y_i$. In words, the almost-arms (and their interior clusters) stay away from the endpoints of the other almost-arms (see the center of Figure \ref{fig:almost-arm-events} for an illustration). While the event $\hat{B}^{\delta,J}_\tau(m,n)$ was chosen so that the arms can be extended, the event $\tilde{B}^{\delta,J}_\tau(m,n)$ is chosen in such a way that an analogue of the three-arm event in the halfplane occurs whenever the event $B_\tau(m,n)$, but none of the events $\tilde{B}^{\delta,J}_\tau(m,n), J^{(n)} \subset \mathcal{J}^{(n)}$ holds, where $\mathcal{J}^{(n)} \subset \R^2$  denotes a sequence of counterclockwise-ordered points on the boundary of $[-n,n]^2$ that are at distance $2\delta n$ from each other.
	This will be explained in more detail in Steps 2.2 and 2.3. To begin with, let
	us show how to construct the well-separated almost-arm event with extensions $\hat{B}^{\delta,J}_\tau(m,n)$ from the event $\tilde{B}^{\delta,J}_\tau(m,n)$.

	\begin{figure}
		\centering
	    \def\svgwidth{0.95\columnwidth}
        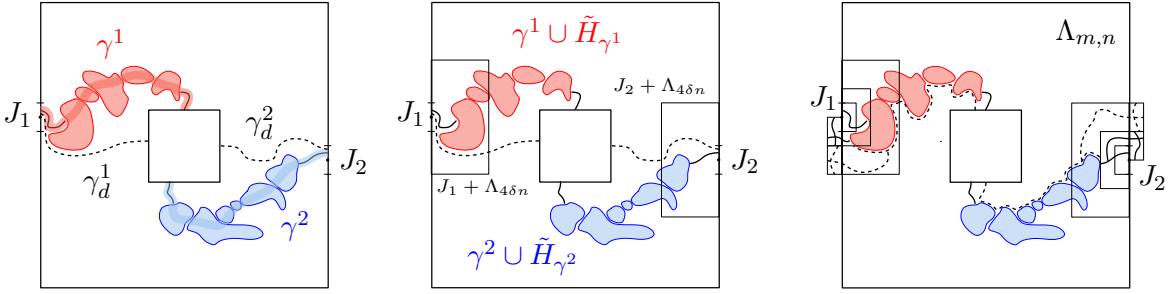
        \caption{\emph{Left.} The event $B^{\delta,J}_{RB}(m,n)$. \emph{Center.} The event $\tilde{B}^{\delta,J}_{RB}(m,n)$. \emph{Right.} This shows the construction of the event $\hat{B}^{\delta,J}_\tau(m,n)$.}
        \label{fig:almost-arm-events}
	\end{figure}

	\emph{Step 1: $P(\hat{B}^{\delta,J}_\tau(m,n)) \ge \epsilon_2 \cdot P(\tilde{B}^{\delta,J}_\tau(m,n))$ for some $\epsilon_2>0$.}

	The condition on the event $\tilde{B}^{\delta,J}_\tau(m,n)$ ensures that every almost-arm $\gamma^i$, $1\le i \le \abs{\tau}$, can be chosen so that it is measurable with respect to the percolation configuration $\omega$ restricted to $\Lambda_{m,n} \setminus \cup_{j \neq i} (J_j + \Lambda_{4\delta n})$ and the colors of the interior clusters of $\omega$. We proceed as follows.

	Assume that the event $\tilde{B}^{\delta,J}_\tau(m,n)$ occurs and condition on the percolation configuration $\omega$ and the colors of the interior clusters in $ \Lambda_{m,n} \setminus \cup_i (J_i + \Lambda_{4\delta n})$. We now construct the extensions for each $i$ separately and refer to the right of Figure \ref{fig:almost-arm-events} for an illustration. Let $1 \le i \le \abs{\tau}$ and assume for simplicity that $J_i$ belongs to the right side of the boundary of $[-n,n]^2$. We explore the lowest almost-arm $\gamma^i$ of color $\sigma_i$ from $\partial \Lambda_m$ to the boundary segment of length $2\delta n$ centered at $J_i$, for which $\gamma^i \cup \tilde{H}_{\gamma^i}$ is contained in $\Lambda_{m,n} \setminus \cup_{j \neq i} (J_j + \Lambda_{4\delta n})$. This can be done as follows: First, we explore the lowest almost-arm (i.e.\ the first in clockwise order)  of color $\sigma_i$ in $(J_i + \Lambda_{4\delta n}) \cap \Lambda_{m,n}$ from the boundary segment of length $2\delta n$ centered at $J_i$ to the boundary of $J_i + \Lambda_{4\delta n}$. After this exploration, we might have fully explored some clusters, which touch the boundary of $J_i + \Lambda_{4\delta n}$, and for these interior clusters of $\Lambda_{m,n} \setminus \cup_{j \neq i} (J_j + \Lambda_{4\delta n})$, we then explore their colors. Now, we have either found the desired almost-arm $\gamma^i$ from $\partial \Lambda_m$ to $\partial \Lambda_n$ or we repeat the procedure by exploring the next almost-arm (in clockwise order)  of color $\sigma_i$ in $(J_i + \Lambda_{4\delta n}) \cap \Lambda_{m,n}$.
	
	Finally, we explore the lowest dual open simple path from $\partial \Lambda_m$ to $\partial(J_i + \Lambda_{2\delta n})$ that is above $\gamma^i$. As shown on the right in Figure \ref{fig:almost-arm-events}, it is now easy to construct an $\omega$-open path inside $J_i + \Lambda_{\delta n, 2\delta n}$ to $\partial \Lambda_{n+\delta n}$ together with an $\omega$-open path inside $ R^{out}_{\delta n}(J_i)$ from top to bottom. In particular, this creates an extension of $\gamma^i$ and a crossing of $R^{out}_{\delta n}(J_i)$ in the long direction. In the same way, we
	create a dual open path inside $J_i + \Lambda_{2\delta n, 4\delta n}$ to $\partial \Lambda_{n+\delta n}$ and a crossing of $S_{\delta n}^{out}(J_i)$ in the long direction. Under the conditioning, the probability of these constructions are bounded below by $\epsilon_2>0$ thanks to the crossing estimates of Theorem \ref{thm:strongRSW} and the FKG inequality. This completes Step 1.

	In the following Steps 2.1, 2.2 and 2.3, we argue that when the event $B_\tau(m,n) \setminus \tilde{B}_\tau^{\delta,J}(m,n)$ occurs, then an event of small probability (going to $0$ as $\delta \to 0$) must occur in the outermost dyadic annulus $\Lambda_{n/2,n}$.

	\emph{Step 2.1: Almost-arms stay away from corners.}

	We denote by $E_1^\delta(n)$ the event that there exists an almost-arm in $\Lambda_{3n/4,n}$ from the inner to the outer boundary that ends at distance $< \delta^{\beta/4}n$ from a corner of $\Lambda_n$.
    
    Let us argue that the probability of $E_1^\delta(n)$ goes to $0$ as $\delta \to 0$.
	For a corner $z$ of $\Lambda_n$, we consider the quarter-annulus $(z + \Lambda_{\delta^{\beta/4}n,n/4}) \cap \Lambda_{n}$. If there is an $\omega$-open path $\lambda$ which crosses the quarter-annulus from one boundary segment of $\partial \Lambda_n$ to the other and if the cluster hull $H_{\lambda,\Lambda_n}$ is entirely contained in the quarter-annulus (which is the case if $\lambda$ lies between two dual-open paths), then there cannot exist an almost-arm from $\partial \Lambda_{3n/4}$ to $\partial \Lambda_n$ that ends at distance $< \delta^{\beta/4}n$ from $z$. Indeed, this would mean that the almost-arm intersects $\lambda$ but does not connect to the boundary using $H_{\lambda,\Lambda_{n}}$, contradicting the definition of an almost-arm. By the crossing estimates of Theorem \ref{thm:strongRSW}, the probability of an $\omega$-open (resp.\ dual-open) circuit in a dyadic annulus $\Lambda_{n/2^{i+1}, n/2^{i}}$ is uniformly bounded from below. Hence, $P(E_1^\delta(n))\to 0$ follows since the number of dyadic annuli in $\Lambda_{\delta^{\beta/4}n,n/4}$ goes to $\infty$ as $\delta \to 0$.

	\emph{Step 2.2: Almost-arms end at some $\delta$-separated $J^{(n)}$.}

	To partition $\partial \Lambda_n$ into boundary segments of length $2\delta n$, we fix a sequence of counterclockwise-ordered points $\mathcal{J}^{(n)} \subset \R^2$ on the boundary of $[-n,n]^2$ that are at distance $2\delta n$ from each other. In addition, we denote by $X \subset \R^2$ a fixed sequence of counterclockwise-ordered points  on the boundary of $[-n,n]^2$ that are at distance $2\delta^{\beta/2} n$ from each other and from which we have removed all points that are at distance $< \delta^{\beta/4}n$ from a corner.
	We define
	\begin{equation*}
		E_2^\delta(n):= B_\tau(m,n) \setminus \left( \bigcup_{J^{(n)} \subset  \mathcal{J}^{(n)}\, \delta\text{-sep.}} B_\tau^{\delta,J}(m,n) \cup E_1^\delta(n)\right).
	\end{equation*}
	We now argue that $P(E_2^\delta(n)) \to 0$ as $\delta \to 0$. The idea is as follows: Either the almost-arms can be chosen to go to well-separated boundary segments,  there exists an almost-arm that ends close to a corner, or there occurs a variant of an alternating three-almost-arm event somewhere at the boundary $\partial \Lambda_n$.  We refer to the left of Figure \ref{fig:almost-arm-steps} for an illustration. We also point out that we always choose $\gamma^i$ to end at the furthest point of its boundary-touching cluster (in counterclockwise-order), thereby guaranteeing  the existence of a dual open path $\gamma^i_d$ ending next to the endpoint of $\gamma^i$.

	Let us explain why the occurrence of the event $E_2^\delta(n)$ implies that for some $x \in X$, there occurs a variant of an alternating three-almost-arm event in the semi-annulus $(x + \Lambda_{3 \delta^{\beta/2} n,n/4}) \cap \Lambda_n$: Assume that the first $i-1$ almost-arms were chosen to end at well-separated boundary segments around the points $J_1,\ldots,J_{i-1} \in \mathcal{J}^{(n)}$. If the almost-arm $\gamma^i$ \emph{must} be chosen to end at the boundary segment of some $J_i$ that is at distance less than $2\delta^{\beta/2}n$ from the previously chosen points, we denote the point in $X$ that is closest to $J_i$ by $x_i$ and observe that for some $n'=2^\ell 3\delta^{\beta/2} n$ with $3\delta^{\beta/2} n \le n' < n/4$, there occurs an alternating three-almost-arm event $B^{++}_{\tau'}( 2n',n/4)$ in the semi-annulus $(J_i + \Lambda_{2n', n/4}) \cap \Lambda_{n}$ (with $\tau' \in \{RBR,BRB\}$ depending on $\tau_i$) and an alternating FK three-arm event $A_{010}^{+}(3\delta^{\beta/2} n, n')$ in the semi-annulus $(J_i + \Lambda_{3\delta^{\beta/2} n,n'}) \cap \Lambda_{n}$. Here, $\ell = 0$ corresponds to the case that the boundary-touching cluster of $\gamma^i$ is contained in $(J_i + \Lambda_{3\delta^{\beta/2} n}) \cap \Lambda_{n}$.
	In summary, we obtain
	\begin{align*}
		&P(E^{\delta}_2(n)) \\
		&\le \sum_{x \in X} \sum_{\ell = 0}^{\log(\delta^{-\beta/2}/12)} C \cdot  \underbrace{P\left(A_{010}^{+}(3\delta^{\beta/2}n, 2^\ell 3\delta^{\beta/2} n)\right)}_{\le c_2^{-1} (1/2^{\ell})^{1+\beta_1}} \cdot \underbrace{P\left(B^{++}_{RBR}( 2^{\ell+1} 3\delta^{\beta/2} n,n/4)\right)}_{\le c_2^{-1} (2^{\ell+1}/2^{\log(\delta^{-\beta/2}/12)})^{1+\beta_2}} \\
		&\le c_3^{-1} \cdot \abs{X} \cdot \log(\delta^{-\beta/2}) \cdot \delta^{(\beta/2)(1+\beta)}\\
		& \le c_4^{-1} \cdot \log(\delta^{-\beta/2}) \cdot \delta^{\beta^2/2} \to 0 \quad \text{as} \; \delta \to 0,
	\end{align*}
	where we have used the mixing property of Corollary \ref{cor:mixing-half} in the first inequality, and then the upper bounds from Corollaries \ref{cor:six-arms-FK-with-bc} and \ref{cor:three-arm-fuzzy-potts} and also $\abs{X} = O(\delta^{-\beta/2})$.

	\begin{figure}
		\centering
	    \def\svgwidth{0.9\columnwidth}
        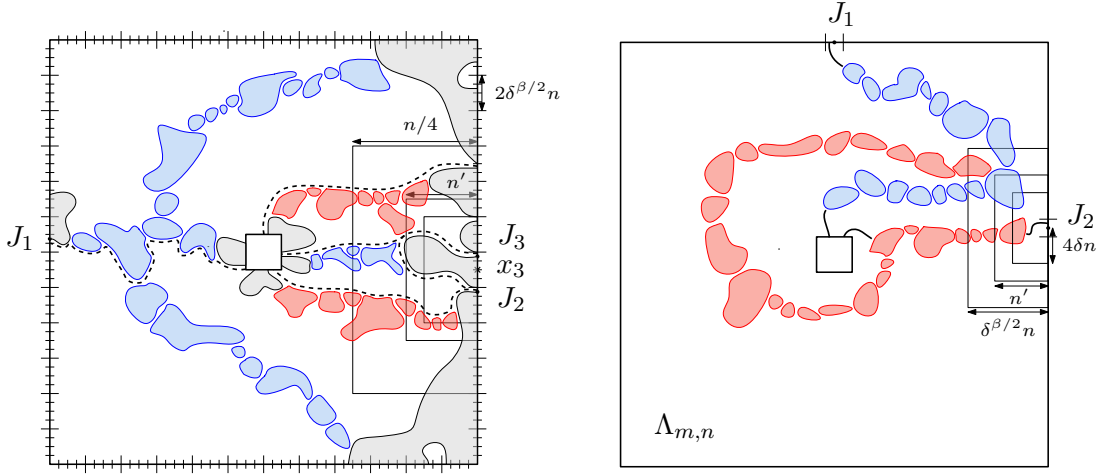
        \caption{\emph{Left.} If the almost-arms cannot be chosen to end at some $\delta$-separated $J^{(n)} \subset \mathcal{J}^{(n)}$, then for some $x \in X$, a variant of an alternating three-almost-arm event occurs in the semi-annulus $(x + \Lambda_{3 \delta^{\beta/2} n,n/4}) \cap \Lambda_n$.  \emph{Right.} If the almost-arms cannot be chosen to stay away from each others endpoint, then for some $J_i \in \mathcal{J}^{(n)}$, a variant of an alternating three-almost-arm event occurs in the semi-annulus $(J_i +  \Lambda_{4\delta n, \delta^{\beta/2}n}) \cap \Lambda_n $.}
        \label{fig:almost-arm-steps}
	\end{figure}

	\emph{Step 2.3: Almost-arms stay away from endpoints of the other almost-arms.}

	For $J^{(n)}=(J_1,\ldots,J_ {\abs{\tau}})$ $\delta$-separated, we denote
	\begin{equation*}
		E_3^{\delta,J}(n) := B_\tau^{\delta,J}(m,n) \setminus \tilde{B}_\tau^{\delta,J}(m,n)  \quad \text{and} \quad E_3^\delta(n):= \bigcup_{J^{(n)} \subset \mathcal{J}^{(n)}\, \delta\text{-sep.}} E^{\delta,J}_3(n).
	\end{equation*}
	Let us now argue that $P( E^{\delta}_3(n) ) \to 0$ as $\delta \to 0$. If the event $E_3^{\delta,J}(n)$ occurs, then there exists $1 \le i \le \abs{\tau}$ such that for any choice of the almost-arm $\gamma^i$,
	\begin{equation*}
	 (\gamma^i \cup \tilde{H}_{\gamma^i}) \cap 	\bigcup_{j \neq i}  (J_j + \Lambda_{4\delta n}) \neq \emptyset.
	\end{equation*}
	As before, this implies a variant of an alternating three-almost-arm event in the semi-annulus $(J_j \cap \Lambda_{4\delta n, \delta^{\beta/2}n})\cap \Lambda_n$ for some $1\le j \le \abs{\tau}$ (see the right of Figure \ref{fig:almost-arm-steps} for an illustration). If the almost-arm $\gamma^i$ must be chosen to intersect $J_j + \Lambda_{4\delta n}$, then we simply observe an alternating three-almost-arm event $B^{++}_{\tau'}(4\delta n, \delta^{\beta/2}n)$ (with $\tau' \in \{RBR,BRB\}$ depending on $\tau_i$) in the semi-annulus. But actually, the intersection of $\gamma^i \cup \tilde{H}_{\gamma^i}$ with $J_j + \Lambda_{4\delta n}$ might also be due to an interior cluster of $\gamma^i$. In any case, one can show that for some $n'=2^\ell 4\delta n$ with $4\delta n \le n' < \delta^{\beta/2} n$, there occurs an alternating three-almost-arm event $B^{++}_{\tau'}( 2n',\delta^{\beta/2}n)$  in the semi-annulus $(J_j + \Lambda_{2n',\delta^{\beta/2} n}) \cap \Lambda_{n}$ and an alternating FK three-arm event $A_{010}^{+}(4\delta n, n')$ in the semi-annulus $(J_j + \Lambda_{4\delta n,n'}) \cap \Lambda_{n}$.
	In summary, we obtain
	\begin{align*}
		&P(E^{\delta}_3(n)) \\
		&\le \sum_{x \in \mathcal{J}^{(n)}} \sum_{\ell = 0}^{\log(\delta^{\beta/2 -1}/8)} C \cdot  \underbrace{P\left(A_{010}^{+}(4\delta n, 2^\ell 4\delta n)\right)}_{\le c_5^{-1} (1/2^{\ell})^{1+\beta_1}} \cdot \underbrace{P\left(B^{++}_{RBR}( 2^{\ell+1} 4\delta n, \delta^{\beta/2} n )\right)}_{\le c_5^{-1} (2^{\ell+1}/2^{\log(\delta^{\beta/2 -1}/8)})^{1+\beta_2}} \\
		&\le c_6^{-1} \cdot \abs{\mathcal{J}^{(n)}} \cdot \log(\delta^{\beta/2 -1}) \cdot \delta^{(1-\beta/2)(1+\beta)}\\
		& \le c_7^{-1} \cdot \log(\delta^{\beta/2 -1}) \cdot \delta^{(1-\beta)\beta/2} \to 0 \quad \text{as} \; \delta \to 0,
	\end{align*}
	where we have used the mixing property of Corollary \ref{cor:mixing-half} in the first inequality, and then the upper bounds from Corollaries \ref{cor:six-arms-FK-with-bc} and \ref{cor:three-arm-fuzzy-potts} and also $\abs{\mathcal{J}^{(n)}} = O(\delta^{-1})$.

	\emph{Step 3: Conclusion.}

	In Steps 2.1, 2.2 and 2.3, we have shown that for any $n\ge 2m\ge 1$,
	\begin{align*}
		B_\tau(m,n) \subset & \bigcup_{J^{(n)} \subset  \mathcal{J}^{(n)}\, \delta\text{-sep.}} \tilde{B}_\tau^{\delta,J}(m,n)  \\ &\cup \left( B_\tau(m,n/2) \cap \left(E_1^\delta(n) \cup E_2^\delta(n) \cup E_3^\delta(n)\right)\right)
	\end{align*}
	with $P(E_1^\delta(n) \cup E_2^\delta(n) \cup E_3^\delta(n)) \to 0$ as $\delta \to 0$.
	Hence, the mixing property for FK percolation (which is applicable since we consider almost-arm events instead of arm events) and the bound of Step 1 imply
	\begin{equation} \label{eq:arm-separation-bound}
		P\left(B_\tau(m,n)\right) \le \sum_{J^{(n)} \subset  \mathcal{J}^{(n)}\, \delta\text{-sep.}} \epsilon_2^{-1} \cdot P\left(\hat{B}_\tau^{\delta,J}(m,n)\right)  +  P\left( B_\tau(m,n/2)\right) \cdot  f(\delta),
	\end{equation}
	where the function $f$ is chosen such that $C \cdot P(E_1^\delta(n) \cup E_2^\delta(n) \cup E_3^\delta(n))\le f(\delta) \to 0$ uniformly in $n$ as $\delta\to 0$ (here $C$ is the constant from Corollary \ref{cor:mixing}).

	We would now like to iteratively apply equation \eqref{eq:arm-separation-bound}. Fix any $\delta$-separated $J_\ast^{(n)} \subset \mathcal{J}^{(n)}$. By applying equation \eqref{eq:arm-separation-bound} to $(m,n/2)$, we obtain
	\begin{align*}
		P\left(B_\tau(m,n)\right) &\le P\left(B_\tau(m,n/2)\right) \\
		&\le \sum_{J^{(n/2)} \subset  \mathcal{J}^{(n/2)}\, \delta\text{-sep.}} \epsilon_2^{-1} \cdot P\left(\hat{B}_\tau^{\delta,J}(m,n/2)\right)  +  P\left( B_\tau(m,n/4)\right) \cdot  f(\delta) \\
		&\le (\epsilon_1\epsilon_2\epsilon_3)^{-1} c_1^{-1} \cdot P\left(\hat{B}_\tau^{\delta,J_\ast}(m,n)\right)  +  P\left( B_\tau(m,n/4)\right) \cdot  f(\delta)\;,
	\end{align*}
	where we have applied equation \eqref{eq:arm-separation-apriori-bound} for each $J^{(n/2)}$ in the second inequality and the constant $\epsilon_3$ accounts for the number of terms in the sum. By iteration, we get
	\begin{equation*}
			P\left(B_\tau(m,n)\right) \le   (\epsilon_1\epsilon_2\epsilon_3)^{-1} P\left( \hat{B}_\tau^{\delta,J_\ast}(m,n)\right) \cdot \sum_{i=0}^{\log(n/m)-1} c_1^{-(i+1)}f(\delta)^i
	\end{equation*}
	Finally, we fix $\delta_0>0$ sufficiently small such that $f(\delta_0)\le c_1/2$. Then for $\delta < \delta_0$, this implies
		\begin{equation*}
			P\left(B_\tau(m,n)\right) \le  2c_1^{-1} (\epsilon_1\epsilon_2\epsilon_3)^{-1} P\left( \hat{B}_\tau^{\delta,J_\ast}(m,n)\right),
	\end{equation*}
	and thereby completes the proof of equation \eqref{eq:arm-separation-reduced}.
	In the same way, one can show that for any $\delta$-separated $I_\ast^{(m)}$,
	\begin{equation}
		 P\left(\hat{B}_\tau^{\delta,J_\ast}(m,n)\right) \lesssim P\left(\hat{B}_\tau^{\delta,I_\ast,J_\ast}(m,n)\right),
	\end{equation}
	which implies equation \eqref{eq:arm-separation-main-step} when combined with equation \eqref{eq:arm-separation-reduced} and thereby concludes the proof of the the theorem.
\end{proof}

\subsection{Quasi-multiplicativity and arm separation for the halfplane}
	\label{sec:arm-separation-halfplane}

	In this section, we briefly explain how to establish the following theorem.

	\begin{thm}[Quasi-multiplicativity for the halfplane] \label{thm:quasi-multiplicativity-halfplane}
		Let $\tau$ be an alternating color sequence for the halfplane.  For all $\abs{\tau} \le \ell \le m \le n$,
		\begin{equation*}
			\mu_{\Z\times\Z_+}^0(A^+_\tau(\ell,n)) \asymp \mu_{\Z\times\Z_+}^0(A^+_\tau(\ell,m)) \cdot  \mu_{\Z\times\Z_+}^0(A^+_\tau(m,n)),
		\end{equation*}
		where the bounds are uniform in $\ell,m,n$. The same holds true for the arm event version $A^{s}_\tau$.
	\end{thm}

	In Sections \ref{sec:quasi-multiplicatity} and \ref{sec:arm-separation}, we always worked in the plane and we have therefore only considered almost-arms with respect to $\Z^2$. In the halfplane, it is natural to consider almost-arms with respect to $\Z\times\Z_+$, i.e. the almost-arm event $B^+_\tau(m,n)$. The first step is then to define the
	well-separated almost-arm event with extensions $\hat{B}^{+ \,\delta,I,J}_\tau(m,n)$ in the halfplane as a subset of $B^+_\tau (m,n)$. The definition is identical to Definition \ref{def:well-seperated-almost-arm-extensions}, except that no dual open simple path is needed to separate the last almost-arm $\gamma^{\abs{\tau}}$ from the first almost-arm $\gamma^{1}$  since this separation will be guaranteed by the free boundary condition. Using the same reasoning as in the previous subsections, one then obtains:

	\begin{thm}[Arm separation for the halfplane] \label{thm:arm-separation-halfplane}
		Let $\tau$ be an alternating color sequence for the halfplane. There exists $\delta_0 >0$ such that for all $\abs{\tau} \le m \le n$, for all $\delta \in (0,\delta_0)$
		and for all $I^{(m)}$, $J^{(n)}$ $\delta$-separated,
		\begin{equation*}
			P_{\Z\times\Z_+}^0\left(\hat{B}_\tau^{+ \,\delta,I,J}(m,n)\right) \asymp  \mu_{\Z\times\Z_+}^0\left(A^+_\tau(m,n)\right) \asymp 	P_{\Z\times\Z_+}^0\left(B^+_\tau(m,n)\right),
		\end{equation*}
		uniformly in $I, J$ and in $m,n$. The same holds true for the arm event version $A^{+s}_\tau(m,n)$ and the corresponding versions of the almost-arm events.
	\end{thm}

	Actually, the proof can be simplified in this case since arms can be explored in counterclockwise order starting from the right end of the semi-annulus as explained in Step 2 of the proof of Proposition \ref{prop:two-arm-fuzzy-potts}.

\subsection{Bounding and comparing arm exponents}
\label{sec:additional-discrete-results}

In this section, we apply arm separation to prove several results about arm events in the fuzzy Potts model, which will be needed in Sections \ref{sec:convergence} and \ref{sec:combine-results}. First of all, arm separation also allows us to prove a simple upper bound on the alternating six-almost-arm event that will be useful in Section \ref{sec:convergence}.

\begin{prop}\label{prop:6-almost-arm-bound}
	There exists a constant $\beta_3=\beta_3(r,q)>0$ such that for all $n \ge m
	\ge 1$,
	\begin{equation*}
		P_{\Lambda_n}^\xi\left(B_{RBRBRB}(m,n)\right) \lesssim \left(\frac{m}{n}\right)^{2+\beta_3},
	\end{equation*}
	where the bounds are uniform in $\xi$ and in $m,n$.
\end{prop}

\begin{proof}
	First, we show that uniformly for all $1\le m \le n$,
	\begin{equation}\label{eq:5-arm-upper-bound}
		P_{\Z^2}\left(B_{RBRBR}(m,n)\right) \lesssim \left(\frac{m}{n}\right)^{2}.
	\end{equation}
	For some (FK) cluster $C_0 \in \mathcal{C}$ with  $C_0 \subset \Lambda_{n}$, we denote by $E(C_0;2n)$ the event that there exist five disjoint sequences of clusters $C^{(j)}=(C_i^{(j)})_{i=1}^{k_j}$, $j \in \{1,\ldots,5\}$, in $\Lambda_{2n}$ such that
	\begin{itemize}
		\item $C_0$ and all clusters in $C^{(1)}$, $C^{(3)}$ and $C^{(5)}$ are red, all clusters in $C^{(2)}$ and $C^{(4)}$ are blue,
		\item the sequences of clusters start next to $C_0$ and intersect $\partial \Lambda_{2n}$, i.e.\ for $j = 1,3,5$,
		\begin{equation*}
			d_1(C_0,C_1^{(j)}) = 1 \quad \text{and} \quad d_1(C^{(j)}_i,C_{i+1}^{(j)}) = 1, \; \forall 1 \le i < k_j,
		\end{equation*}
		and for $j=2,4$,
		\begin{equation*}
			d_\infty(C_0,C_1^{(j)}) = 1 \quad \text{and} \quad d_\infty(C^{(j)}_i,C_{i+1}^{(j)}) = 1, \; \forall 1 \le i < k_j,
		\end{equation*}
		and for each $j$, only the last cluster $C^{(j)}_{k_j}$ touches the boundary $\partial \Lambda_{2n}$,
		\item $C^{(1)}_{k_1}$ and $C^{(5)}_{k_5}$ touch the top, $C^{(2)}_{k_2}$ touches the left, $C^{(3)}_{k_3}$ touches the bottom, and $C^{(4)}_{k_4}$ touches the right of the boundary $\partial \Lambda_{2n}$.
	\end{itemize}
	It is easy to verify that the  event $E(C_0;2n)$ occurs for at most one cluster $C_0 \in \mathcal C$, i.e.\
		\begin{equation*}
		1 \ge \sum_{C_0 \subset \Lambda_{n}} 1_{E(C_0;2n)}.
	\end{equation*}
	The bound \eqref{eq:5-arm-upper-bound} for $m=1$ follows by taking expectations and in the case of Bernoulli percolation, this is a well-known step in proving that the five-arm exponent is universal and equals 2. For general $m$, we cover $\Lambda_{n}$ with boxes of size $m$ to deduce
	\begin{equation*}
		1 \ge \sum_{x \in 2m \Z^2\, :\, x + \Lambda_m \subset \Lambda_{n}} 1_{E(C_0;2n)\; \text{holds for some}\; C_0 \subset x+\Lambda_m}\,,
	\end{equation*}
	and \eqref{eq:5-arm-upper-bound} will follow by taking expectations  once we have established that for some $\delta>0$ and $I^{(m)}$, $J^{(n)}$ $\delta$-separated as in Theorem \ref{thm:arm-separation},
	\begin{align*}
        P_{\Z^2}\left(E(C_0;2n)\; \text{holds for some}\; C_0 \subset x+\Lambda_m\right) &\gtrsim 	P_{\Z^2}\left( \hat{B}_{RBRBR}^{\delta,I,J}(m,n)\right) \\
        &\asymp P_{\Z^2}\left( B_{RBRBR}(m,n)\right)\;.
	\end{align*}
	The second comparison follows from arm separation (see Remark \ref{rk:arm-separation-non-alternating}) and the first bound follows from the following construction: Assume that the event $\hat{B}_{RBRBR}^{\delta,I,J}(m,n)$ centered at $x$ occurs. We extend the five almost-arms from the boundary of $x+\Lambda_n$ to the respective sides of $\partial\Lambda_{2n}$. Inside $x + \Lambda_m$, we explore the innermost $\omega$-open circuit in $x + \Lambda_{m/8,m/4}$ (which exists with positive probability) and then its cluster, called $C_0$ (which is contained in $x + \Lambda_{m/2}$ with positive probability). We then extend the five almost-arms from the boundary of $x+\Lambda_{m}$ so that they end next to the cluster $C_0$. By coloring $C_0$ and the clusters corresponding to the extensions of the five almost-arms in red and blue, we have constructed an instance of the event where $E(C_0;2n)$ holds for some $C_0 \subset x+\Lambda_m$.

	Second, the bound \eqref{eq:5-arm-upper-bound} implies that there exists $\beta_3>0$ such that uniformly for all $1\le m \le n$,
	\begin{equation}\label{eq:6-arm-upper-bound}
		P_{\Z^2}\left(B_{RBRBRB}(m,n)\right) \lesssim \left(\frac{m}{n}\right)^{2+\beta_3}.
	\end{equation}
	The argument is standard. Indeed, using arm separation, the endpoints of the almost-arms can be specified, which allows to explore the first five almost-arms ($RBRBR$) in such a way that the sixth almost-arm ($B$) must lie in the unexplored part. The probability of the sixth almost-arm is then polynomially small in $m/n$ since it can be blocked in each dyadic annuli with positive probability.
	To deduce the result from \eqref{eq:6-arm-upper-bound}, we use the inclusion
    $B_{RBRBRB}(m,n)\subset B_{RBRBRB}(m,n/2)$ for $n>2m$ and the mixing property of FK percolation.
\end{proof}

The last proposition of this section explains how to deduce non-alternating arm
exponents from alternating ones by showing that the arm event probabilities only
changes by a constant factor when replacing a red (resp.\ blue) arm by several
subsequent red (resp.\ blue) arms. It  relies crucially on the
result that the alternating four-arm exponent is strictly larger than 2, which will be
established in Section \ref{sec:combine-results}.

Let $\tau$ be any color sequence. We formally define a reduced version $\hat{\tau}$ for the plane and a reduced version
$\hat{\tau}^{+}$ for the halfplane as follows. Whenever the letter $R$ (resp.\ $B$) occurs subsequently more than once, we replace this subsequence by a single letter. The last and the first letter are viewed as subsequent in the plane but not in the halfplane.
Clearly, $\hat{\tau}$ is alternating and $\hat{\tau}^+$ is alternating for the halfplane.

\begin{prop}\label{prop:alternating-non-alternating}
	Assume that there exist $\beta_4 >0$ such that uniformly for all $n\ge 1$,
	\begin{equation*}
       \mu_{\Z^2}\left(A_{RBRB}(1,n)\right) \lesssim n^{-(2+\beta_4)}.
	\end{equation*}
	Let $\hat{\tau}$ (resp.\ $\hat{\tau}^+$) be the reduction of $\tau$ in the plane (resp.\ halfplane). It then holds that uniformly for all $n\ge m \ge \lvert\tau\rvert$,
	\begin{equation*}
		\mu_{\Z^2}\left(A_\tau(m,n)\right) \asymp \mu_{\Z^2}\left(A_{\hat{\tau}}(m,n)\right).
	\end{equation*}
    The same holds true for the arm event version $A^{s}_\tau$.
\end{prop}

\begin{remark}
    \label{rk:halfplane-non-alternating}
	The proof of Proposition \ref{prop:alternating-non-alternating} can analogously be applied in the halfplane case. It yields that uniformly for all $n\ge m \ge \lvert\tau\rvert$,
	\begin{equation*}
		\mu_{\Z\times\Z_+}^0\left(A^+_\tau(m,n)\right) \asymp \mu_{\Z\times\Z_+}^0\left(A^+_{\hat{\tau}^+}(m,n)\right),
	\end{equation*}
	and the same holds true for the arm event version $A^{+s}_\tau$.
	The only additional ingredient needed in the halfplane case is that there exist $\beta_5 >0$ such that uniformly for all $n\ge 1$,
	\begin{equation*}
		\mu_{\Z\times\Z_+}^0\left(A^+_{RBR}(1,n)\right) \lesssim n^{-(1+\beta_5)}.
	\end{equation*}
	This follows directly from the comparison with the alternating two-arm exponent in the halfplane (see Remark \ref{rk:universal-exponents-fuzzy-potts}).
    
    It is needed since the defects (see below for more details) might be close to $\Z \times \{0\}$ and thereby induce an alternating three-arm event in the halfplane instead of an alternating four-arm event in the plane.
\end{remark}

In the proof of this result, we will naturally encounter paths that look like red or blue arms except for a few defects. Therefore, we introduce $A_\tau^{(d)}(m,n)$, the arm event with $d$ defects, containing all colorings $\sigma$ for which there is some $\sigma' \in A_\tau(m,n)$ with $\abs{\{v : \sigma'_v \neq \sigma_v\}} \le d$.

Since almost-arm events depend both on the percolation configuration and on the coloring, one might consider different notions of defects. In analogy with the previous definition, the almost-arm event with $d$ defects, denoted by $B_\tau^{(d)}(m,n)$, contains all configurations $(\omega,\sigma)$ for which changing the color of at most $d$ vertices (formally by closing the four adjacent edges and coloring the isolated vertex) results in a configuration $(\omega',\sigma') \in B_\tau(m,n)$.

\begin{prop} \label{prop:arm-events-with-defects}
	Let $\tau$ be an alternating color sequence and $d\ge 1$ be the number of defects. Uniformly for all $n \ge m \ge \abs{\tau}$, it holds that
	\begin{equation*}
		\mu_{\Z^2}\left(A_\tau^{(d)}(m,n)\right) \lesssim \left(1+\log\left(\frac{n}{m}\right)\right)^d \mu_{\Z^2}\left(A_\tau(m,n)\right),
	\end{equation*}
	and
	\begin{equation*}
		P_{\Z^2}\left(B_\tau^{(d)}(m,n)\right) \lesssim \left(1+\log\left(\frac{n}{m}\right)\right)^d P_{\Z^2}\left(B_\tau(m,n)\right).
	\end{equation*}
	The same holds true for the arm (resp.\ almost-arm) event versions $A^{s}_\tau$ and $A^{+s}_\tau$.
\end{prop}
The upper bound on (almost-)arm events with defects is a standard consequence of the arm separation techniques that were established in the previous subsections and we refer the reader to \cite[Proposition 18]{Nolin2008} for a proof in the case of Bernoulli percolation that can directly be adapted to our setting.

\begin{proof}[Proof of Proposition \ref{prop:alternating-non-alternating}]
	Let us present the proof for the events $A_\tau(m,n)$ and $A_{\hat{\tau}}(m,n)$. The other cases can be proven in exactly the same way. Hence, we work with the measures $P_{\Z^2}$ and $\mu_{\Z^2}$ throughout the proof and abbreviate them by $P$ and $\mu$.

	The upper bound $\mu(A_\tau(m,n)) \le \mu(A_{\hat{\tau}}(m,n))$ is trivial by inclusion. For the lower bound, it will be sufficient to show that
	\begin{equation}\label{eq:alternating-non-alternating-lower-bound}
		\mu\left(A_\tau(m,n)\right) \gtrsim \mu\left(A_{\hat{\tau}}(m,n)\right)
	\end{equation}
	for $m,n$ of the form $2^k$ and for $\tau$ of the form $\tau=\hat{\tau}_1\cdots\hat{\tau}_1\,\cdots\,\hat{\tau}_\ell\cdots\hat{\tau}_\ell$, where $\ell = \abs{\hat{\tau}}$ and each letter of $\hat{\tau}$ is repeated $d$ times for some fixed $d\ge 1$. Then, the result follows for general $m,n$ using the extendability for $A_{\hat{\tau}}$ and for general $\tau$ by inclusion.

	Depending on the length of $\hat{\tau}$, we define the event $E_{\hat{\tau}}(m,n)$ (resp. $E'_{\hat{\tau}}(m,n)$) as follows:
	\begin{itemize}
		\item If $\abs{\hat{\tau}} \le 2$, it denotes the existence of a red (resp.\ blue) arm from $\partial \Lambda_{m/4}$ to $\partial \Lambda_{4n}$ in $\Lambda_{m/4,4n}$ and a closed weak (resp.\ strong) path surrounding $\Lambda_m$ inside $\Lambda_{m,n}$ that is blue (resp.\ red) except for at most $d-1$ defects.
		\item If $\abs{\hat{\tau}} > 2$, it denotes the existence of three counterclockwise ordered arms $\hat{\gamma}^1$, $\hat{\gamma}^2$, $\hat{\gamma}^3$ of colors $BRB$ (resp.\ $RBR$) from $\partial \Lambda_{m/4}$ to $\partial \Lambda_{4n}$ in $\Lambda_{m/4,4n}$ and a weak (resp.\ strong) path connecting $\hat{\gamma}^1$ and  $\hat{\gamma}^3$ inside $\Lambda_{m,n}$ in counterclockwise order  that is blue (resp.\ red) except for at most $d-1$ defects.
	\end{itemize}
	It follows from Menger's theorem that if the maximal number of disjoint red (resp.\ blue) arms crossing the annulus $\Lambda_{m,n}$ from the inner to the outer boundary is $j$, then there is a blue (resp.\ red) circuit with  $j$ defects in the annulus. For $\abs{\hat{\tau}} \le 2$, we directly deduce that
	\begin{equation*}
		A_{\hat{\tau}}(m/4,4n) \subset A_\tau(m,n) \cup \left(E_{\hat{\tau}}(m,n) \cap A_{\hat{\tau}}(m/4,4n)\right) \cup \left(E'_{\hat{\tau}}(m,n) \cap A_{\hat{\tau}}(m/4,4n)\right),
	\end{equation*}
	and for $\abs{\hat{\tau}} > 2$, the same conclusion is obtained by considering the connection between the arms $\hat{\gamma}^1$ and  $\hat{\gamma}^3$ with the least number of defects.
	
	In the following, we will show that there exists some $m_0$ such that for $n\ge m \ge m_0$,
	\begin{equation}\label{eq:alternating-non-alternating-main-bound}
		\mu\left(E_{\hat{\tau}}(m,n) \cap A_{\hat{\tau}}(m/4,4n)\right) + \mu\left(E'_{\hat{\tau}}(m,n) \cap A_{\hat{\tau}}(m/4,4n)\right) \le \frac{1}{2} \mu \left(A_{\hat{\tau}}(m/4,4n)\right).
	\end{equation}
	Once we have established this inequality, it will follow that for $n\ge m \ge m_0$,
	\begin{equation*}
		\mu\left(A_\tau(m,n)\right) \ge \frac{1}{2} \mu \left(A_{\hat{\tau}}(m/4,4n)\right) \asymp \mu \left(A_{\hat{\tau}}(m,n)\right),
	\end{equation*}
	and this directly implies the lower bound \eqref{eq:alternating-non-alternating-lower-bound} for general $n\ge m \ge \abs{\tau}$ since $\mu(A_\tau(\abs{\tau},n)) \gtrsim \mu(A_\tau(m_0,n))$ can be deduced from the finite energy property of FK percolation.

	To show \eqref{eq:alternating-non-alternating-main-bound}, we note that on the event $E_{\hat{\tau}}(m,n)$, a four-arm event with at most $d-2$ defects occurs locally around at least one defect in $E_{\hat{\tau}}(m,n)$. Thus, we have
	\begin{align*}
		E_{\hat{\tau}}(m,n) \subset \bigcup_{x \in \Lambda_{m,n}} \left\{x + A^{(d-2)}_{RBRB}(1,2^{k-1})\right\}.
	\end{align*}
	Decomposing $\Lambda_{m,n}$ into dyadic annuli, it follows that
	\begin{align*}
		&E_{\hat{\tau}}(m,n) \cap A_{\hat{\tau}}(m/4,4n) \\
		&\subset \bigcup_{k=\log(m)}^{\log(n)-1} A_{\hat{\tau}}(m/4,2^{k-2})\cap \left(\bigcup_{x \in \Lambda_{2^k, 2^{k+1}}} \left\{x + A^{(d-2)}_{RBRB}(1,2^{k-1})\right\} \right) \cap A_{\hat{\tau}}(2^{k+3},4n).
	\end{align*}
	We emphasize that the intermediate event depends only on the coloring in $\Lambda_{2^{k-1},2^{k+2}}$.

    To apply the mixing property of Corollary \ref{cor:mixing}, we  consider almost-arm events instead of arm events. Thereby, we obtain
	\begin{align*}
		&\mu\left(E_{\hat{\tau}}(m,n) \cap A_{\hat{\tau}}(m/4,4n)\right) \\
		&\le C^2 \sum_{k=\log(m)}^{\log(n)-1} P\left(B_{\hat{\tau}}(m/4,2^{k-2})\right) \cdot \abs{\Lambda_{2^k, 2^{k+1}}}  \\
		& \hspace{2.75cm} \cdot P\left(B^{(d-2)}_{RBRB}(1,2^{k-1})\right) 	 \cdot P\left(B_{\hat{\tau}}(2^{k+3},4n)\right) \\
		&\lesssim  \mu\left(A_{\hat{\tau}}(m/4,4n)\right) \cdot \sum_{k=\log(m)}^{\log(n)-1} \abs{\Lambda_{2^k, 2^{k+1}}} \cdot P\left(B^{(d-2)}_{RBRB}(1,2^{k-1})\right),
	\end{align*}
	where we have used extendability and quasi-multiplicativity in the second comparison. Using Proposition \ref{prop:arm-events-with-defects} and the assumption on the alternating four-arm event, it follows that
	 \begin{align*}
	 	&\sum_{k=\log(m)}^{\log(n)-1} \abs{\Lambda_{2^k, 2^{k+1}}} \cdot P\left(B^{(d-2)}_{RBRB}(1,2^{k-1})\right) \\ &\lesssim \sum_{k=\log(m)}^{\log(n)-1} 4^k \cdot k^{d-2} \cdot \underbrace{P\left(B_{RBRB}(1,2^{k-1})\right)}_{\asymp \mu(A_{RBRB}(1,2^{k}))\lesssim (2^k)^{-(2+\beta_4)}} \lesssim  \sum_{k=\log(m)}^{\infty}   k^{d-2} \left(2^{-\beta_4}\right)^k < \infty.
	 \end{align*}
 	In particular, there exists $m_0$ sufficiently large such that the sum is smaller than $1/4$ for $m\ge m_0$. The upper bound for $E_{\hat{\tau}}(m,n) \cap A_{\hat{\tau}}(m/4,4n)$ follows analogously. This establishes \eqref{eq:alternating-non-alternating-main-bound} and completes the proof.
\end{proof}

\section{Convergence results}
\label{sec:convergence}

As announced in the introduction, the aim of this section will be to prove the
convergence of the (discrete) fuzzy Potts model to its continuum counterpart
conditional on the conformal invariance conjecture for FK percolation (stated as
Conjecture \ref{conj:fk-to-cle}). There are two types of results that we
present: Firstly, we will describe the scaling limit of a single fuzzy Potts
interface and secondly, we will determine the scaling limit of the loop encoding
of an entire discrete fuzzy Potts model. We will only make use of the former
result to transport continuum critical exponents back to the discrete setting
but emphasize that the latter result is needed to transport the one-arm exponent
back but the corresponding continuum exponent is not known at the time of writing.

Throughout this section, we fix $q\in [1,4)$ and assume Conjecture
\ref{conj:fk-to-cle} for this value $q$ and let $\kappa'\in (4,6]$ be the
parameter appearing in the statement of the conjecture. We will also fix $r\in
(0,1)$ and let
$\Gamma,\Gamma^\partial,\Gamma^O,\Gamma^I,\Gamma^{OR},\Gamma^{OB},
\Gamma^{IR}, \Gamma^{IB}, \Sigma$ and $(\gamma^{a,b})$ be as in Section
\ref{sec:cle-perco}. Recall also that analogous objects have been defined in the
discrete setting in Section \ref{sec:loop-encoding}.

We also fix a Jordan domain $D$, $\epsilon_n\to 0$ and
discrete domains $\mathcal{D}_n=(V_n,E_n)$ in $\epsilon_n\Z^2$ converging to $D$
(see Section \ref{sec:loop-encoding} for the sense of convergence). Recall that
we wrote $\partial\mathcal{D}_n$ for the set of boundary vertices of
$\mathcal{D}_n$. By Skorokhod's representation theorem and the scaling limit
conjecture we can therefore couple $\omega^n\sim \phi^0_{\mathcal{D}_n,q}$ with
$\Gamma$ such that
\begin{align*}
    (\Gamma_{\omega^n}\setminus
    \Gamma_{\omega^n}^\partial,\Gamma_{\omega^n}^\partial) \to (\Gamma\setminus
    \Gamma^\partial,\Gamma^\partial)
    \quad\text{a.s.}
\end{align*}
with respect to $d_\mathcal{L}$. It is straightforward to see that nesting
levels are preserved by limits with respect to $d_\mathcal{L}$ and in particular
$\Gamma_{\omega^n}^O\setminus \Gamma_{\omega^n}^\partial \to \Gamma^O\setminus
\Gamma^\partial$ with respect to $d_\mathcal{L}$. In this section we would like
to consider fuzzy Potts configurations $(\omega^n,\sigma^n)\sim
P^0_{\mathcal{D}_n,q}$ which are coupled together such that
\begin{align}
    \label{eq:coupling-colors}
    \begin{split}
    (\Gamma_{\omega^n,\sigma^n}^{OR}\cap
    \Gamma^\partial_{\omega^n},\Gamma^{OB}_{\omega^n,\sigma^n}\cap
    \Gamma^\partial_{\omega^n}) &\to (\Gamma^{OR}\cap \Gamma^\partial,
    \Gamma^{OB}\cap \Gamma^\partial)\;,\\
    (\Gamma_{\omega^n,\sigma^n}^{OR}\setminus
    \Gamma^\partial_{\omega^n},\Gamma^{OB}_{\omega^n,\sigma^n}\setminus
    \Gamma^\partial_{\omega^n}) &\to (\Gamma^{OR}\setminus \Gamma^\partial,
    \Gamma^{OB}\setminus \Gamma^\partial)\;,\\
    (\Gamma_{\omega^n,\sigma^n}^{IR},\Gamma^{IB}_{\omega^n,\sigma^n}) &\to
    (\Gamma^{IR}, \Gamma^{IB}) \quad\text{a.s.}
    \end{split}
\end{align}
Let us quickly explain how to construct such a coupling. By the definition of
$d_\mathcal{L}$ we may take $G_n\subset \Gamma^\partial_{\omega^n}$, $F_n\subset
\Gamma^\partial$, $G_n'\subset \Gamma^O_{\omega^n}\setminus
\Gamma^\partial_{\omega^n}$, $F_n'\subset \Gamma^O\setminus \Gamma^\partial$ and
bijections $\pi_n\colon G_n\to F_n$, $\pi'_n\colon G_n'\to F_n'$ such that
\begin{align*}
    \sup_{\eta\in G_n} d_\mathcal{C}(\eta,\pi_n(\eta))\vee
    \sup_{\eta\notin G_n} \diam(\eta)\vee
    \sup_{\eta\notin F_n} \diam(\eta) &\le
    2\,d_\mathcal{L}(\Gamma^O_{\omega^n}\setminus \Gamma^\partial_{\omega^n},
    \Gamma^O\setminus \Gamma^\partial)\;,\\
    \sup_{\eta\in G'_n} d_\mathcal{C}(\eta,\pi'_n(\eta))\vee
    \sup_{\eta\notin G'_n} \diam(\eta)\vee
    \sup_{\eta\notin F'_n} \diam(\eta) &\le
    2\,d_\mathcal{L}(\Gamma^\partial_{\omega^n}, \Gamma^\partial)
\end{align*}
and such that $G_n,F_n,G_n',F_n',\pi_n,\pi_n'$ are measurable with respect to
$\Gamma, \Gamma_{\omega^n}$. We leave the construction of these measurable
mappings to the reader.

For fixed $n\ge 1$ we assign a color to the cluster $C\in \mathcal{C}(\omega^n)$
as follows. Recall that $\eta^C_{\omega^n}$ denoted the outer boundary of $C$.
If $\eta^C_{\omega^n}\notin G_n\cup G_n'$ we color $C$ (independently) all in
red ($R$) with probability $r$ and all in blue ($B$) otherwise. If
$\eta^C_{\omega^n}\in G_n$ and $\pi_n(\eta^C_{\omega^n})\in \Gamma^{OR}$ (resp.\
$\pi_n(\eta^C_{\omega^n})\in \Gamma^{OB}$) we color $C$ in red (resp.\ blue).
Similarly, if $\eta^C_{\omega^n}\in G'_n$ and $\pi'_n(\eta^C_{\omega^n})\in
\Gamma^{OR}$ (resp.\ $\pi'_n(\eta^C_{\omega^n})\in \Gamma^{OB}$) we color $C$ in
red (resp.\ blue). We call the obtained configuration $\sigma^n\in
\{R,B\}^{V_n}$. By Theorem \ref{thm:msw-free} we have $(\omega^n,\sigma^n)\sim
P^0_{\mathcal{D}_n,q}$ and it is direct from the definitions to check that
\eqref{eq:coupling-colors} holds.

\medspace

The results we prove on the convergence of fuzzy Potts `boundary to boundary'
interfaces and interface collections are as follows. We are formulating them
here as convergence statements assuming that we are working with a coupling
\eqref{eq:coupling-colors} but we can directly deduce convergence in
distribution assuming Conjecture \ref{conj:fk-to-cle} by the discussion above.

\begin{thm}
    \label{thm:interface-convergence}
    Let $a,b\in \partial D$ be distinct and let $a_n,b_n\in \partial \mathcal{D}_n$
    be such that $a_n\to a$ and $b_n\to b$. Then assuming
    \eqref{eq:coupling-colors}, we have
    \begin{align*}
        (\gamma_{\sigma^n,a_n,b_n}^+,\gamma_{\sigma^n,a_n,b_n}^-)\to
        (\gamma^{a,b},\gamma^{a,b})
    \end{align*}
    in probability with respect to the metric $d_{\mathcal{C}'}$ in each
    coordinate.
\end{thm}

\begin{thm}
    \label{thm:all-loops-converge}
    Assuming \eqref{eq:coupling-colors}, we have $(\Sigma_{\sigma^n}^+,
    \Sigma_{\sigma^n}^-)\to (\Sigma,\Sigma)$ in probability with respect to
    the metric $d_\mathcal{L}$ in each coordinate.
\end{thm}

By combining the two theorems above with \eqref{eq:coupling-colors}, one directly obtains the joint (distributional) convergence of the FK percolation and fuzzy Potts loop encodings.

The idea of the proofs is to approximate the continuum fuzzy Potts interface
with the use of Proposition \ref{prop:msw-interface-approx} and Corollary
\ref{cor:msw-loop-approx}. It will then suffice to argue that chains of touching
continuum cluster boundaries are approximated by chains of touching discrete
cluster boundaries. The key input (also in the convergence of all discrete fuzzy
Potts interfaces in the next section) is that if clusters get very close in the
discrete then they in fact, touch each other. The following lemma is folklore
and is a standard consequence of the fact that the six-arm exponent of FK
percolation is larger than $2$ as stated in Corollary
\ref{cor:six-arms-FK-with-bc}.

\begin{lemma}
    \label{lem:fk-clusters-touch}
    For $0<\delta'<\delta$, let $G_n(\delta',\delta)$ be the event that whenever for some $z \in \C$ 
    with $B_{2\delta}(z) \subset D$, the annulus $B_\delta(z)\setminus B_{\delta'}(z)$ is crossed by two $\omega^n$-open clusters $C$ and $C'$ that are disjoint inside the annulus $B_\delta(z)\setminus B_{\delta'}(z)$, then $C$ and $C'$ contain vertices $v$ and $v'$ respectively that are nearest neighbors in $\mathcal D_n$. 
    For all $\delta>0$ we have that $\inf_n\P(G_n(\delta',\delta))\to 1$ as
    $\delta'\to 0$.
\end{lemma}

\begin{proof}
    Let $A_\tau(z,\delta',\delta,n)$ be the event that we see an arm event of type
    $\tau=100100$ with respect to the percolation configuration $\omega^n$ in
    the annulus $B_\delta(z)\setminus B_{\delta'}(z)$. Then
    \begin{align*}
        G_n(\delta',\delta)^c\subset\bigcup_{z\in\C\colon B_{2\delta}(z)\subset
        D} A_\tau(z,\delta',\delta,n) \subset \bigcup_{z'\in \delta'(\Z+i\Z)\colon
        B_\delta(z')\subset D}
        A_\tau(z',2\delta',\delta/2,n)
    \end{align*}
    where for the second inclusion, we need to make the assumption
    $\delta'<\delta/4$. The first inclusion is immediate from the definition of
    $G_n(\delta',\delta)$ and the second one follows since if
    $B_{2\delta}(z)\subset D$ we can take $z'\in (\delta'(\Z+i\Z))\cap
    B_{\delta'}(z)$ and in that case $A_\tau(z,\delta',\delta,n)\subset
    A_\tau(z',2\delta',\delta/2,n)$ and $B_\delta(z')\subset D$. Therefore by
    Corollary \ref{cor:six-arms-FK-with-bc} we get
    \begin{align*}
        \P(G_n(\delta',\delta))\lesssim
        (\delta')^{-2}(\delta'/\delta)^{2+\beta_1} \to 0\quad\text{as
        $\delta'\to 0$}
    \end{align*}
    for any $\delta>0$ and the claim follows.
\end{proof}

The appearance of the condition $B_{2\delta}(z)\subset D$ is only to ensure that
the annuli are contained in the discrete domains for $n$ sufficiently large. The
next lemma is the analogous result for the fuzzy Potts model and will play a
role when we apply Lemma \ref{lem:technical-loop} and \ref{lem:technical-curve}.

\begin{lemma}
    \label{lem:no-dac-six-arms}
    For $0<\delta'<\delta$, let $G_n'(\delta',\delta)$ be the event that for
    each $z\in D$ with $B_{2\delta}(z)\subset D$, there exist at most five disjoint loop segments of loops in $\Sigma^+_{\sigma^n}$ crossing the annulus
    $B_\delta(z)\setminus B_{\delta'}(z)$. For all $\delta>0$ we
    have that $\inf_n\P(G'_n(\delta',\delta))\to 1$ as $\delta'\to 0$. The same statement holds with $\Sigma^-_{\sigma^n}$ in place of $\Sigma^+_{\sigma^n}$.
\end{lemma}

\begin{proof}
    The proof is virtually identical to the one of Lemma
    \ref{lem:fk-clusters-touch} except that we make use of Proposition
    \ref{prop:6-almost-arm-bound} instead of Corollary
    \ref{cor:six-arms-FK-with-bc}. We therefore omit the details here.
\end{proof}

We will also make use of a small deterministic lemma, the purpose of which is to
remove the need to consider boundaries of fillings of loops in Proposition
\ref{prop:msw-interface-approx} and Corollary \ref{cor:msw-loop-approx} so that
we can work with the loops directly.

\begin{lemma}
    \label{lem:remove-loop-exc}
    Consider $\eta\in C^s(\partial\D,\C)$ and suppose that the boundary of the
    unbounded connected component of $\C\setminus \eta(\partial\D)$ is a simple
    curve $\eta'$ which surrounds a point $z\in\C$. Consider $\epsilon>0$, then
    there exist distinct $s_0^-,s_0^+,\dots,s_{n-1}^-,s_{n-1}^+\in \partial\D$
    ordered in a counterclockwise way such that
    \begin{align*}
        \eta_{s_i^+}=\eta_{s_{i+1}^-}\in \eta'(\partial\D)\quad\text{for all
        $i< n$ (addition modulo $n$)}\;,
    \end{align*}
    $\widetilde{\eta}(\partial\D)\subset
    \eta'(\partial\D)+B_\epsilon(0)$ and such that $\widetilde{\eta}$ surrounds
    $z$ where $\widetilde{\eta}$ is the
	concatenation of $\eta\lvert_{(\!(s_0^-,s_0^+)\!)},\dots,
	\eta\lvert_{(\!(s_{n-1}^-,s_{n-1}^+)\!)}$ parametrized to be a function on
    $\partial\D$.
\end{lemma}

\begin{proof}
    By applying a conformal transformation to the domain enclosed by $\eta'$
    (which extends continuously to the boundary of the domain) we may assume
    that $\eta'$ parametrizes $\partial\D$. Therefore $\eta(\partial\D)\subset
    \overline{\D}$ and $\eta(\partial\D)\cap\partial\D =\partial\D$.

	Suppose without loss of generality that $\lvert\eta_1\rvert\le 1-\epsilon$. It suffices
    to find $t_-,t_+\in \partial\D$ such that $t_-,1,t_+$ are oriented
    counterclockwise and such that $\eta_{t_-}=\eta_{t_+}$ as well as $\eta(
    \{t_\pm\}\cup (\!(t_+,t_-)\!) )\cap\partial\D=\partial\D$. We can therefore
    cut out an excursion from $\eta$ without changing the outer boundary.
    Iterating this procedure a finite number of times yields the lemma.

    By the definition of $C^s(\partial\D,\C)$ as uniform limits of simple
    curves, it suffices to prove the following quantitative version for simple
    curves: Consider disjoint balls $B_i := B_{r_i}(x_i)$ for $i=0,\dots,N-1$ centered
    on points $x_0,\dots,x_{N-1}\in \partial \D$ (ordered counterclockwise), let
    $\eta$ be a simple counterclockwise oriented curve in $\D$ intersecting
    $B_{r_i}(x_i)$ for all $i$, and let $I_0,\dots,I_{m-1}\in \{0,\dots,N-1\}$
	be an order in which the balls are visited by $\eta$, i.e.\ there are $s_0,\ldots,s_{m-1} \in \partial \D$ (distinct and ordered counterclockwise) such that $\eta_{s_j} \in B_{I_j}$ for all $j$ and such that $\eta((\!(s_j,s_{j+1})\!))$ intersects $B_i$ only if $i \in \{I_j,I_{j+1}\}$. Then there exist $1\le k \le m-1$ such
	that $I_k = I_0 + 1$ (modulo $N$) and such that $\{ I_0, I_k, I_{k+1}, \dots, I_{m-1} \} = \{ 0, \dots, N-1 \}$. Showing this is a straightforward combinatorial exercise and omitted here.
	
	To obtain the claim about functions in $C^s(\partial\D, \C)$, let $\eta$ be as
	before and suppose that $\eta^n$ are simple curves uniformly approximating
	$\eta$. Let $B_i$ be arbitrary (as above) and write $s^n_0,\dots,s^n_{m_n - 1}$ and
	$I^n_0,\dots,I^n_{m_n - 1}$ for the points and indices associated to $\eta^n$.
	Without loss of generality, we may assume that $1 \in (\!(s^n_0,s^n_1)\!)$ for
	all $n\ge 1$ (dropping some initial elements of the sequence $(\eta^n)$ if
	necessary). We let $k_n$ be then associated to $\eta^n$ as before. Then
	\begin{align*}
		\eta^n_{s^n_0} \in B_{I^n_0} \; , \quad \eta^n_{s^n_{k_n}} \in B_{I^n_0 + 1} \quad
		\text{and} \quad \eta^n( (\!( s^n_{k_n}, s^n_0 )\!) ) \cap B_i \neq \emptyset 
		\ \forall i \; .
	\end{align*}
	The sequences $(I^n_0)$, $(s^n_0)$ and $(s^n_{k_n})$ have convergent
	subsequences so this property passes to the limit as $n \to \infty$ to ensure that there are
	distinct $s,s' \in \partial\D$ and $I_0 \le N-1$ such that
	\begin{align*}
		\eta_s \in \overline{B}_{I_0} \; , \quad \eta_{s'} \in \overline{B}_{I_0 + 1}
		\quad \text{and} \quad \eta^n( (\!( s',s )\!) ) \cap \overline{B}_i \ \forall
		i \; .
	\end{align*}
	The claim now follows by using this result with balls $B_{1/(n')^2} (e^{2 \pi i
		\cdot j / n'})$ for $j=0,\dots,n'-1$ and using another compactness argument in the
	$n' \to \infty$ limit.
\end{proof}

\subsection{Convergence of divide and color interfaces}

The proof of Theorem \ref{thm:interface-convergence} is now rather
straightforward by combining all the ingredients we have collected in the lemmas
above. In the proofs, the reason why certain inclusions only hold when $n$ is
sufficiently large (rather than for all values of $n$) is that the outer and
inner boundaries of FK clusters are themselves not open paths but are a distance given
by at most the mesh size of the lattice away from an open path in the
percolation configuration. Another effect which results in some inclusions only
holding for $n$ sufficiently large is the need to consider approximations where
the boundary curve of the discrete domain is sufficiently close to the continuum
one.

Let us point out that the assumption that boundary clusters converge
to boundary clusters in Conjecture \ref{conj:fk-to-cle} is needed in the proof
below to ensure that certain discrete paths end on, rather than only close to
the discrete boundary. In the case of domains with piecewise linear boundaries,
this property can also be deduced from the fact that  the alternating boundary three-arm event with
type sequence $010$ for FK percolation has exponent $>1$ (see Corollary \ref{cor:six-arms-FK-with-bc}).

\begin{proof}[Proof of Theorem \ref{thm:interface-convergence}]
	Fix $\epsilon>0$. By the SDE description of the marginal law of $\gamma^{a,b}$ (see Remark \ref{rk:interface-def}), the interior of $\gamma^{a,b}([0,1])\cap \partial D$ is almost surely empty. Therefore, by Lemma \ref{lem:technical-curve} and
    \ref{lem:no-dac-six-arms}, it suffices to show that the probability of the
    event
    \begin{align*}
        A_n :=\{\gamma^+_{\sigma^n,a_n,b_n}([0,1])
        \subset \gamma^{a,b}([0,1])+B_\epsilon(0) ,
        \gamma^-_{\sigma^n,a_n,b_n}([0,1])
        \subset \gamma^{a,b}([0,1])+B_\epsilon(0) \}
    \end{align*}
    tends to $1$ as $n\to \infty$. Let $I_B$ (resp.\ $I_R$) be the open
    counterclockwise (resp.\ clockwise) boundary segment from $a$ to $b$ along
    $\partial D$. By combining Proposition \ref{prop:msw-interface-approx}
    (together with the observations in Remark \ref{rk:small-boundary-clusters})
    and Lemma \ref{lem:remove-loop-exc} we get the following almost sure
    statement for any fixed $\epsilon'>0$.

    Suppose that $0\le s<t\le 1$ is such that $\gamma^{a,b}_s,\gamma^{a,b}_t\in
	I_B$ and $\gamma^{a,b}( (s,t))\cap I_B=\emptyset$; we call $\gamma^{a,b}\rvert_{(s,t)}$ a right boundary excursion of $\gamma^{a,b}$.
    
    Then there exists
    $\eta^1,\dots,\eta^m\in \Gamma^{OB}$ (oriented clockwise by convention) and
    $s_1^\pm,\dots,s_m^\pm\in \partial\D$ such that the concatenation of
	$\eta^1\rvert_{(\!(s_1^-,s_1^+)\!)},\dots,\eta^m\rvert_{(\!(s_m^-,s_m^+)\!)}$ defines
    a continuous curve $\widetilde{\gamma}$, viewed as a function on the
    interval $[0,1]$, with the following properties.
    \begin{itemize}
        \item The curve $\widetilde{\gamma}$ lies right of $\gamma^{a,b}$.
        \item We have $\widetilde{\gamma}_0 \in B_{\epsilon'}(\gamma^{a,b}_s)$,
            $\widetilde{\gamma}_1 \in B_{\epsilon'}(\gamma^{a,b}_t)$ and
            $\widetilde{\gamma}([0,1]) \subset
            \gamma^{a,b}([s,t])+B_{\epsilon'}(0)$.
        \item It holds that $\diam(\eta^1)<\epsilon'$,
            $\diam(\eta^m)<\epsilon'$.
        \item For all $i<m$ we have $\eta^i_{s^+_i}=\eta^{i+1}_{s^-_{i+1}}\notin
            \partial D$ and $\eta^1_{s^-_1},\eta^m_{s^+_m}\in \partial D$.
    \end{itemize}
    Note that Proposition \ref{prop:msw-interface-approx} was stated in the case
    $D=\D$ but the above statements follow since conformal transformations from
    the unit disk to Jordan domains extend to homeomorphisms on the closures of
	the domains. Let $0<s_1<t_1<\cdots<s_k<t_k<1$ be the times such that $\gamma^{a,b}\rvert_{(s_i,t_i)}$ are all the right boundary excursions for which
    $\gamma^{a,b}( (s_i,t_i))+B_{\epsilon/2}(0)$ does not contain the
    connected component of $D\setminus \gamma^{a,b}( (s_i,t_i) )$ which lies
    right of $\gamma^{a,b}$. From Lemma \ref{lem:fk-clusters-touch} and
    \eqref{eq:coupling-colors},
    it follows that for each
    $\epsilon''>0$, the probability of the following event tends to $1$ as $n\to
    \infty$: For each $1\le i\le k$ there exists a blue strong path
    $\nu^{n,i}_B$ in $\mathcal{D}_n$ from $u^{n,i}_B\in \partial\mathcal{D}_n
    \cap B_{\epsilon''/2}(\gamma^{a,b}_{s_i})$ to $v^{n,i}_B\in
    \partial\mathcal{D}_n\cap B_{\epsilon''/2}(\gamma^{a,b}_{t_i})$ such that its
    image is contained in $\gamma^{a,b}( (s_i,t_i) )+ B_{\epsilon''/2}(0)$. In
    fact, by construction, the path $\nu^{n,i}_B$ is a concatenation of a finite number
    of blue percolation cluster outer boundary curves.

    When $1\le i\le k+1$, let us write $\widetilde{\nu}^{n,i}_B$ for
    the simple curve parametrizing the counterclockwise boundary segment of
    $\mathcal{D}_n$ from $v^{n,i-1}_B$ to $u^{n,i}_B$ where $v^{n,0}_B=a_n$ and
    $u^{n,k+1}_B=b_n$.
    
    Then $\widetilde{\nu}^{n,i}_B$ is at distance
    $<\epsilon''/2$ with respect to $d_{\mathcal{C}'}$ to the counterclockwise
    boundary segment from $v^{i-1}_B$ to $u^i_B$ in $\partial D$ for $n$
    sufficiently large where
    \begin{align*}
        v^0_B=a\;,\quad u^i_B = \gamma^{a,b}_{s_i}\;,\  v^i_B =
        \gamma^{a,b}_{t_i}\quad\text{for $1\le i\le k$}\;,\quad v^{k+1}_B
        = b\;.
    \end{align*}
    The analogous statement holds when $B$ and $R$, `clockwise' and
    `counterclockwise', `blue' and `red', as well as `left' and `right' are
    interchanged in the statements above. Let us define $\nu^n_B$ to be the
    concatenation of the paths
    \begin{align*}
        \widetilde{\nu}^{n,1}_B, \nu^{n,1}_B, \dots
        \widetilde{\nu}^{n,k}_B, \nu^{n,k}_B, \widetilde{\nu}^{n,k+1}
    \end{align*}
    and define $\nu^n_R$ analogously. The claim now follows from the fact that
    $\gamma^\pm_{\sigma^n,a_n,b_n}$ is `sandwiched' in between $\nu^n_R$ and
    $\nu^n_B$, i.e.\ $\gamma^\pm_{\sigma^n,a_n,b_n}$ has to lie left of
	$\nu^n_B$ and right of $\nu^n_R$.
\end{proof}

\subsection{Limit of divide and color loops}

The proof of Theorem \ref{thm:all-loops-converge} will be rather similar to the
one in the previous section. The significant difference is that we need to
exclude the possibility of `small' FK clusters (which disappear in the scaling
limit) agglomerating into large fuzzy Potts clusters which need to be considered
when describing the scaling limit of the fuzzy Potts model. This will be
achieved by arguing that each discrete fuzzy Potts interface passes in between
two macroscopic FK clusters of differing colors and hence has to approximate the
continuum fuzzy Potts interface passing in between the scaling limit of the
macroscopic clusters.

\begin{proof}[Proof of Theorem \ref{thm:all-loops-converge}]
    We present the proof for the collection $(\Sigma_{\sigma^n}^+)$. The proof
    for $(\Sigma_{\sigma^n}^-)$ is analogous. 
    
	\emph{Step 1:  Let $A^\epsilon _n$ be the event that for each $\eta\in \Sigma$ with diameter
		$>\epsilon$, there exists $\eta'\in \Sigma_{\sigma^n}^+$ such that
		$d_\mathcal{C}(\eta,\eta')<\epsilon$. Then for every $\epsilon >0$, we have  $\P(A^\epsilon_n)\to 1$  as $n\to \infty$.}
	 
	By the construction in Section \ref{sec:cle-perco}, the following statement
    holds almost surely for any $\epsilon'>0$. Suppose that $\eta\in \Xi^n_{RB}$
    surrounds a point $z\in D$ (the case of $\eta\in \Xi^n_{BR}$ is analogous).
    Then by Corollary \ref{cor:msw-loop-approx} and Lemma
    \ref{lem:remove-loop-exc} there exists $\eta^1,\dots,\eta^m\in \Gamma^{OB}$
    oriented counterclockwise and $s_1^\pm,\dots,s^\pm_m\in \partial\D$ such
	that the concatenation of $\eta^1\rvert_{(\!(s_1^-,s_1^+)\!)},\dots,
	\eta^m\rvert_{(\!(s_m^-,s_m^+)\!)}$ defines a loop $\widetilde{\eta}_B$ surrounding
    $z$ such that $\widetilde{\eta}_B(\partial\D)\subset
    \eta(\partial\D)+B_{\epsilon'}(0)$, the loop $\eta$ surrounds $\widetilde{\eta}_B$, and such that we have
    \begin{align*}
        \eta^i_{s_i^+}=\eta^{i+1}_{s_{i+1}^-}\notin \partial D\quad \text{for
        all $1\le i\le m$}
    \end{align*}
    where addition in the indices is understood modulo $m$. To approximate
    $\eta$ from the outside, we need to consider two distinct cases. If $\eta$
    is an outermost continuum fuzzy Potts loop, i.e.\ $\eta\in \Xi^1_{RB}$
    (oriented counterclockwise) and $s,t\in \partial\D$ are such that
    $\eta_s,\eta_t\in \partial D$ and $\eta( (\!(s,t)\!) )\subset D$ then one
	can approximate $\eta\lvert_{(\!(s,t)\!)}$ from the outside just like in
    the proof of Theorem \ref{thm:interface-convergence} by making use of
    Proposition \ref{prop:msw-interface-approx}. Suppose now that $\eta\in
    \Xi^n_{RB}$ for $n>1$ (again, the case $\eta\in \Xi^n_{BR}$ is the same).
    Let $\eta'$ denote the curve parametrizing the connected component of
    $\cup_{\eta\in \Upsilon^{n-1}_R} \eta^o$ containing $\eta$. There are two
    cases which are possible in the construction in Section \ref{sec:cle-perco}.
    \begin{itemize}
        \item The curve $\eta'$ is the boundary of a bounded connected
            component of a loop $\eta''\in \Xi''^{(n-1)}_{BR}\cup
            \Xi''^{(n-1)}_{RR}\subset \Gamma^{IR}$. We write $S=\{\eta''\}$ in
            this case.
        \item The curve $\eta'$ is the boundary of a bounded connected component
            of a loop in $\Xi'^{*(n-1)}_{BR}\cup \Xi'^{*(n-1)}_{RR}$. Note that
            these are the false loops of $S:=\Xi'^{(n-1)}_{BR}\cup
            \Xi'^{(n-1)}_{RR}\subset \Gamma^{OR}$ and so a path which is a
            concatenation of segments of false loops also gives a path which is
            a concatenation of segments of true loops.
    \end{itemize}
    Lemma \ref{lem:remove-loop-exc} was stated in the case of boundaries of
    unbounded connected components but (by applying an inversion in a point) it
    extends to the case of bounded connected complementary components. Combining
    the lemma with Proposition \ref{prop:msw-interface-approx} yields that there
    exists $\eta^1,\dots,\eta^m\in \Gamma^{OR}\cup S\subset \Gamma^R$ and
    $s_1^\pm,\cdots,s_m^\pm$ such that the concatenation of
	$\eta^1\lvert_{(\!(s_1^-,s_1^+)\!)},\dots, \eta^m\rvert_{(\!(s_m^-,s_m^+)\!)}$ defines
    a loop $\widetilde{\eta}_R$ surrounding $\eta$ such that
    $\widetilde{\eta}_R(\partial\D)\subset \eta(\partial\D)+B_{\epsilon'}(0)$.
    Therefore the continuum interface is `sandwiched' in between
    $\widetilde{\eta}_R$ and $\widetilde{\eta}_B$ and one concludes that
    $\P(A^\epsilon_n)\to 1$ as $n\to \infty$ in the same way as in the proof of
    Theorem \ref{thm:interface-convergence}.

	\emph{Step 2: Let $G_n^{\epsilon,\alpha}$ be the event that each $\eta\in \Sigma^+_{\sigma^n}$ of 
	diameter $>\epsilon$ touches a red cluster with diameter $>\alpha$ and a blue cluster with 
	diameter $>\alpha$ with respect to the underlying FK percolation configuration $\omega^n$ 
	such that these clusters have graph distance $1$. Then for every $\epsilon >0$, we have 
	$\liminf_{n\to\infty} \P(G_n^{\epsilon,\alpha})\to 1$ as $\alpha\to 0$. }
    
    First, we show that fuzzy Potts interfaces do not stay near the boundary.
    For $\rho>0$ we define $D_{\rho}=\{z\in D\colon d(z,\partial D)>\rho\}$ and
    let $B^{\epsilon,\rho}_n$ denote the event that each $\eta\in \Sigma_{\sigma^n}^+$ with
    diameter $>\epsilon$ intersects $D_\rho$. We claim that
    $\liminf_{n\to\infty}\P(B^{\epsilon,\rho}_n)\to 1$ as $\rho\to 0$. One can prove
    this using the crossing estimates from Theorem \ref{thm:strongRSW} but the
    details are involved and we choose to derive it directly from the scaling
    limit assumption \eqref{eq:coupling-colors} here.
	 To this end, fix a conformal transformation $\varphi\colon D\to \D$ which then extends
    continuously to the boundary. Then since $\Gamma$ is locally finite and the
    set of points in $\partial D$ which lie on a loop in $\Gamma$ is dense in
    $\partial D$, the following statement holds almost surely: For each
    $\epsilon'>0$ there exist $0<\rho'<\rho''<1$ such that for each boundary
    segment of $\partial\D$ of length $\epsilon'$ there exists $\eta\in
    \Gamma^{OR}\cap \Gamma^\partial$ such that $\diam (\varphi\circ\eta)<\rho''$
    with the property that $\eta$ intersects $\varphi^{-1}((1-\rho')\partial\D)$.
    By \eqref{eq:coupling-colors}, these loops in $\Gamma^{OR}\cap
    \Gamma^\partial$ are limits of loops in $\Gamma^{OR}_{\omega^n,\sigma^n}\cap
    \Gamma_{\omega^n}^\partial$. The analogous statement holds for $\Gamma^{OB}\cap
    \Gamma^\partial$ and $\Gamma^{OB}_{\omega^n,\sigma^n}\cap
    \Gamma_{\omega^n}^\partial$. The claim now follows from the continuity of
    the extension of $\varphi$ and the fact that fuzzy Potts interfaces do not
    cross any loop in $\Gamma^{OR}_{\omega^n,\sigma^n}\cap
    \Gamma_{\omega^n}^\partial$ or in $\Gamma^{OB}_{\omega^n,\sigma^n}\cap
    \Gamma_{\omega^n}^\partial$. 
	
	Second, we show that paths cannot avoid macroscopic FK clusters.
    Let $\rho$ be as before and by decreasing it further, assume that
    $\rho<\epsilon$.
	We define the event $H^{\rho,\alpha}_n(z)$ as follows: There is an
    $\omega^n$-clusters chain, each cluster having diameter $>\alpha$, such that
    there is a loop following the annulus $B_{\rho/2}(z)\setminus B_{\rho/4}(z)$
    which only uses vertices in the aforementioned percolation clusters. From
    Theorem \ref{thm:large-cluster-chains}, we deduce that the event
    $H^{\rho,\alpha}_n(z)$ has probability tending to $1$ as $\alpha\to 0$
    uniformly in $n$ and in $B_\rho(z)\subset D$.
    Let us now take $z_1,\dots,z_m\in D_\rho$ such that $\cup_i
    B_{\rho/4}(z_i)\supset D_\rho$. Then the probability of the event
    $H^{\rho,\alpha}_n :=\cap_i H^{\rho,\alpha}_n(z_i)$ also tends to $1$ as
    $\alpha\to 0$ uniformly in $n$. 

	Finally, note that on the event $B^{\epsilon,\rho}_n \cap
    H^{\rho,\alpha}_n$, any $\eta\in \Sigma_{\sigma^n}^+$ of diameter
    $>\epsilon$ intersects $D_\rho$, therefore intersects $B_{\rho/4}(z_i)$
    for some $i\le m$ and hence crosses the annulus $B_{\rho/2}(z_i)\setminus
    B_{\rho/4}(z_i)$ since we assumed that $\rho<\epsilon$. By combining the 
    previous two arguments, we therefore deduce the claim that
    $\liminf_{n\to\infty}\P(G_n^{\epsilon,\alpha})\to 1$ as $\alpha\to 0$. 
    
    \emph{Step 3: Let $\eta_1,\dots,\eta_m$ denote the loops in $\Gamma_{\omega^n}$ of diameter $>\alpha$. Define $F_n^{\alpha,\alpha'}$ to be the event that there exist distinct loops $\eta_1',\dots,\eta_m'\in \Gamma$ of
    	diameter $>\alpha$ with the following properties: 
    	\begin{itemize}
		        \item It holds that $d_\mathcal{C}(\eta_i,\eta_i')<\alpha'$ for every $i\le m$.
		        \item We have $\eta_i\in \Gamma^{OR}_{\omega^n,\sigma^n}$ if and only
		            $\eta'_i\in \Gamma^{OR}$, and similarly for $\Gamma^{OB}$,
		            $\Gamma^{IR}$ and $\Gamma^{IB}$.
		        \item If the loops $\eta_i$ and $\eta_j$ are at distance $\le
		            \epsilon_n$ from each other then $\eta_i'$ and $\eta_j'$ touch
		            each other at a point.
		\end{itemize}
	    Then for every $\alpha,\alpha'>0$, $\P(F^{\alpha,\alpha'}_n) \to 1$ as as $n\to \infty$.}
    
    	This follows readily from \eqref{eq:coupling-colors}. 

    \emph{Step 4: Conclusion}
    
    Fix $\epsilon >0$. As a consequence of the previous three steps, the probability of the event $G_n^{\epsilon,\alpha} \cap F_n^{\alpha,\alpha'} \cap A_n^{\alpha'}$ can be made arbitrarily close to $1$ by choosing $\alpha, \alpha', n$ in this order. We conclude by arguing that on this event, $\Sigma_{\sigma_n}^+$ and $\Sigma$ are close with respect to the metric $d_{\mathcal L}$. 
    
    Let $\eta$ be a loop in $\Sigma_{\sigma_n}^+$ of diameter $>\epsilon$ and consider $\eta_i \in \Gamma_{\omega^n,\sigma^n}^{OR}\cup\Gamma_{\omega^n,\sigma^n}^{IR}$ and $\eta_j \in \Gamma_{\omega^n,\sigma^n}^{OB}\cup\Gamma_{\omega^n,\sigma^n}^{IB}$ of diameter $>\alpha$ such that $\eta_i$ and $\eta_j$ are at distance $\le \epsilon_n$ from each other and such that one is on the inside and the other one is on the outside of $\eta$. The corresponding loops $\eta'_i \in \Gamma^{OR}\cup\Gamma^{IR}$ $\eta'_j \in \Gamma^{OB}\cup\Gamma^{IB}$ touch each other and so there is a unique loop $\eta' \in \Sigma$ passing between them. We note that $\eta'$ must have diameter $>\alpha$ and we set $\pi(\eta)=\eta'$. This defines a map from the loops in $\Sigma_{\sigma_n}^+$ of diameter $>\epsilon$ into $\Sigma$ and it is now sufficient to check that $d_{\mathcal{C}}(\eta,\pi(\eta)) < \alpha'$ for every loop $\eta \in \Sigma_{\sigma_n}^+$, that $\pi$ is injective and that every loop in $\Sigma$ of diameter $>\epsilon + \alpha'$ is contained in the image of $\pi$. Indeed, this follows from the definition of the events $G_n^{\epsilon,\alpha}$, $F_n^{\alpha,\alpha'}$ and $A_n^{\alpha'}$ together with the key observation that whenever one considers two FK percolation
    clusters of opposite colors that are graph distance $1$ away from each
    other, there is a unique discrete fuzzy Potts loop passing between the
    two. 
\end{proof}

\begin{remark}
    When $r=1/q$, we can also prove a version of the theorem above in the case
    when we condition on each vertex in $\partial\mathcal{D}_n$ being red.
    Indeed, let us sample $\sigma^n\sim \mu^0_{\mathcal{D}_n,q,r}$ conditioned
    on $\sigma^n_v=R$ for all $v\in \partial\mathcal{D}_n$. Then, assuming
    Conjecture \ref{conj:fk-to-cle} for the value $q$, we have
    $(\Sigma^+_{\sigma^n},\Sigma^-_{\sigma^n})\to (\Sigma',\Sigma')$ in
    distribution where $\Sigma'$ is a nested $\CLE_\kappa$ in $D$.

    The key observation is that one can equivalently sample $\sigma^n$ by
    starting with $\omega^n \sim \mu^1_{\mathcal{D}_n,q}$ (i.e.\ we consider a
    measure with wired boundary conditions) and by coloring each boundary cluster
    in red and all others independently in red or blue with respective
    probabilities $r$ and $1-r$ to obtain the configuration $\sigma^n$. The
    proof of the result then goes through without changes and by making use of
    Theorem \ref{thm:msw-wired}.

    In the particular case of the FK-Ising model where $q=2$ and hence
    $\kappa=3$ this statement is unconditional by \cite{smirnov-ising,
    kemp-smirnov-fk-bdy,kemp-smirnov-fk-full} and it describes the scaling limit
    of the Ising model with `plus boundary conditions'. This result has already
    been established in \cite{benoist-hongler-cle3} using rather different
    techniques (in particular without going via the Edwards-Sokal coupling) and
    building on the convergence results
    \cite{chelkak-ising-sle,chelkak-smirnov-ising-universal} for Ising
    interfaces.
\end{remark}

\section{Continuum exponents via exploration paths}
\label{sec:continuum-exp}

In this section, we start with results on SLE interior and boundary arm exponents
for $\SLE_\kappa(0,\rho)$ and $\SLE_\kappa$ curves respectively as derived in
\cite{wu-ising-arm} and obtain the interior and boundary arm exponents for
$\SLE_\kappa(\kappa-6-\rho,\rho)$ curves. The novelty of the argument in this
section lies in the use of imaginary geometry (see Section \ref{sec:ig-lemmas})
to show that the additional force points do not affect the value of the
exponent.

We begin by recalling the relevant results on SLE exponents appearing in
\cite{wu-ising-arm}. In this paper, several exponents are computed using
stochastic analysis tools. Other papers where such SLE exponents are computed
using suitable SLE martingales are \cite{lsw-bm-exponents1,
lsw-bm-exponents2, lsw-bm-exponents3, beffara-dim, miller-wu-dim,
wu-poly-arm-exponent, wu-zhan-arm-exponent} and \cite{lsw-mono-arm, ssw-radii}.

We state the results in the unit disk rather than the upper halfplane since it
will be more convenient to transport the results in this setting.

For $r\in (0,2)$ we let $I_r\in \partial\D$ be such that $B_r(1)\cap \partial\D
= (\!(1/I_r,I_r)\!)$. Consider a random curve $\gamma$ from $-i$ to $i$ in $\D$
and consider parameters satisfying
\begin{gather*}
	x\in (\!(-i,i)\!)\;,R\in ( \abs{x + i} ,2)\;,\quad y\in (\!(i,-i)\!)\;,\quad
	0< r<\abs{y-i}\wedge (\abs{y+i}-R)\;,\\
    c>1\;,\quad c_0\in (0,1)\;.
\end{gather*}
Whenever we work with such parameters, we assume that the above relations are
satisfied. Write $\zeta_x= \inf\{t\ge 0\colon
\gamma_t\in (\!(i,x)\!)\}$ for the swallowing time of $x$. Set $\sigma_0=0$ and
for $\epsilon>0$ sufficiently small and $j\ge 1$ we make the following inductive
definition:
\begin{itemize}
	\item Let $\tau_j$  be the first time after $\sigma_{j-1}$ where $\gamma$
		hits the connected component of $\partial B_\epsilon(x)\setminus
		\gamma([0,\sigma_{j-1}])$ containing $xI_\epsilon$.
    \item Let $J$ be the connected component of $(\partial B_r(y)\cap
        \D)\setminus \gamma([0,\tau_1])$ containing $y(1-r)$; the curve $J$
        inherits the counterclockwise orientation of $\partial B_r(y)$. Let
        $\sigma_j$ be the first time after $\tau_j$ when $\gamma$ hits the most
        counterclockwise connected component of $J\setminus \gamma([0,\tau_j])$.
\end{itemize}
For $j\ge 1$ and $k\ge 1$ we define the events
\begin{align*}
	&B_{2j-1}(\gamma,x,\epsilon,y,r) = \{\tau_j<\zeta_x\}\;, \\
	& B_{2j}(\gamma,x,\epsilon,r,y) = \{ \sigma_j < \zeta_x \}\;,\\
	&G_B(\gamma,x,\epsilon,c,c_0,R) \hspace{5cm} \\ &\quad = \{\gamma([0,\tau_1])\subset B_{R}(-i)
    \,, \;
	\gamma([0,\tau_1])\cap ((\!(1/I_{c\epsilon},I_{c\epsilon})\!)
    +B_{c_0\epsilon}(0)) =\emptyset\}\;,\\
	&B'_k(\gamma,x,\epsilon,y,r,c,c_0,R) = B_k(\gamma,x,\epsilon,y,r) \cap
    G_B(\gamma,x,\epsilon,c,c_0,R)\;.
\end{align*}
We can now state
\cite[Proposition 3.3]{wu-ising-arm} on the exponents in the boundary case; the
result we state here is slightly weaker but will suffice for our purposes. We
note that on the event $B_k'(\gamma,x,\epsilon,y,r,c,c_0,R)$ we automatically
have $J=\partial B_r(y)\cap \D$ in agreement with the statement of the
proposition in \cite{wu-ising-arm}.

\begin{thm}[\cite{wu-ising-arm}]
    \label{thm:wu-boundary}
	Let $\kappa\in (0,4)$, $\rho \in (-2,\kappa/2-2)$ and for $j\ge 1$ define
	\begin{align*}
		\alpha^+_{2j-1} &= (2j+\rho)(2j+\rho+2-\kappa/2)/\kappa \;,\\
		\alpha^+_{2j} &= 2j(2j+\kappa/2-2)/\kappa \;.
	\end{align*}
	Consider $\gamma \sim \SLE_\kappa(0,\rho)$ from $-i$ to $i$ in $\D$. For
    $k\ge 1$ there exist constants $y$, $r$, $c$, $c_0$ and $R$ such that
	\begin{align*}
        \epsilon^{\alpha^+_k}\asymp \P( B'_k(\gamma,1,\epsilon,y,r,c,c_0,R) )
	\end{align*}
    for all $\epsilon >0$ sufficiently small.
\end{thm}

The story is similar in the case of interior exponents. Again, we consider a
random curve $\gamma$ in $\D$ from $-i$ to $i$ and parameters
\begin{gather*}
	c\in (0,1)\;,\quad \abs{z}<c\;,\quad R\in (1+c,2)\;,\quad y\in (\!(i,-i)\!)\;, \\
	 0<r<\abs{y-i}\wedge (\abs{y+i}-R)\;.
\end{gather*}
Again, we assume that these relations are satisfied whenever we work with these
parameters. Let $\sigma'_0=0$. For $\epsilon>0$ sufficiently small and $j\ge 1$
we make the following inductive definition.
\begin{itemize}
	\item Let $C_j$ denote the connected component of $\D \setminus
		(\gamma([0,\sigma'_{j-1}])\cup \partial B_r(y))$ with $i$ on its
        boundary. If $z\notin C_j$ we
		let $\tau'_j=\infty$.
        Otherwise (i.e., if $z\in C_j$), we write $C_j'$ for the connected
		component of $C_j\cap \partial B_\epsilon(z)$ which is connected to
		$i$ within $C_j\setminus B_\epsilon(z)$ and we let
        $\tau'_j$ be the first time after $\sigma'_{j-1}$ when $\gamma$ hits
        $C_j'$.
    \item Let $J$ be the connected component of $(\partial B_r(y)\cap
        \D)\setminus \gamma([0,\tau'_1])$ containing $y(1-r)$; the curve $J$
        inherits the counterclockwise orientation of $\partial B_r(y)$. Let
        $\sigma'_j$ be the first time after $\tau'_j$ when $\gamma$ hits the
        most counterclockwise connected component of $J\setminus
        \gamma([0,\tau'_j])$.
\end{itemize}
For $j\ge 1$ we now define the events
\begin{gather*}
    I_{2j}(\gamma,z,\epsilon,y,r)=\{\tau'_j<\zeta_z\}\;,\quad
    G_I(\gamma,z,\epsilon,R)=\{\gamma([0,\tau'_1])\subset B_R(-i)\}\;,\\
    I_{2j}'(\gamma,z,\epsilon,y,r,R) = I_{2j}(\gamma,z,\epsilon,y,r) \cap
    G_I(\gamma,z,\epsilon,R)
\end{gather*}
where $\zeta_z$ is the swallowing time of $z$. With all this notation set up, we
can now state the result on interior exponents
\cite[Proposition 4.1]{wu-ising-arm}; this result is stated for a fixed interior
point $z$, however, the proof also implies the stronger result as stated
below.\footnote{\,Indeed, the only place where the exact location of the
interior point plays a role is in \cite[Lemma 4.2]{wu-ising-arm} which makes use
of \cite[Lemma 4.1]{miller-wu-dim} and the latter is stated uniformly in the
location of the interior point.}

\begin{thm}[\cite{wu-ising-arm}]
    \label{thm:wu-interior}
	Let $\kappa\in (0,4)$ and for $j\ge 1$ define
	\begin{align*}
		\alpha_{2j} = (16j^2-(\kappa-4)^2)/(8\kappa) \;.
	\end{align*}
	Let $\gamma\sim \SLE_\kappa$ from $-i$ to $i$ in $\D$. For $j\ge 1$ there
    exist constants $y$, $r$, $c$ and $R$ such that
	\begin{align*}
        \P( I'_{2j}(\gamma,z,\epsilon,y,r,R) ) = \epsilon^{\alpha_{2j} + o(1)}
        \quad\text{as $\epsilon\to 0$}
	\end{align*}
	uniformly in $\abs{z}<c$.
\end{thm}

The goal is now to transport these results to the setting of
$\SLE_\kappa(\rho,\kappa-6-\rho)$. Because of the proof strategy, the results
are slightly weaker than in the theorems stated above, but this is not relevant
for applications to the discrete fuzzy Potts models.

\begin{prop}
    \label{prop:main-exp-boundary}
    Let $\kappa\in (2,4)$, $\rho\in (-2,\kappa-4)$ and suppose that $\gamma\sim
	\SLE_\kappa(\kappa-6-\rho,\rho)$ from $-i$ to $i$ in $\D$. For $k\ge 1$
    there are constants $y$, $r$, $r'$, $c'$, $c_0'$ and $R'$ such that
	\begin{gather*}
        \epsilon^{\alpha^+_k} \lesssim
        \P( B_k(\gamma,1,\epsilon,y,r))\;,\quad
        \P( B'_k(\gamma,1,\epsilon,y,r',c',c_0',R') )
        \lesssim \epsilon^{\alpha^+_k}
	\end{gather*}
    for all $\epsilon >0$ sufficiently small.
\end{prop}

\begin{prop}
    \label{prop:main-exp-interior}
    Let $\kappa\in (2,4)$, $\rho\in (-2,\kappa-4)$ and suppose that $\gamma\sim
	\SLE_\kappa(\kappa-6-\rho,\rho)$ from $-i$ to $i$ in $\D$. For $j\ge 1$
    there exists constants $y$, $r$, $r'$, $c$ and $R'$ such that
	\begin{gather*}
        \epsilon^{\alpha_{2j} + o(1)} \le \P(I_{2j}(\gamma,z,\epsilon,y,r))
        \;,\quad \P( I'_{2j}(\gamma,z,\epsilon,y,r',R') ) \le
        \epsilon^{\alpha_{2j} + o(1)}\quad\text{as $\epsilon\to 0$}
	\end{gather*}
	uniformly in $\abs{z}<c$.
\end{prop}

The proof of these propositions will be split into three lemmas, each of which
involves the use of imaginary results as listed in Section \ref{sec:ig-lemmas}.
The proofs of all the lemmas below follow roughly the same kind of strategy. We
start with one SLE curve and construct another one from it with the following
property: If the first one satisfies a certain arm event, so will the latter.
The results on distortion estimates for conformal maps used below appear in
Appendix \ref{sec:app-conformal}.

Throughout the remainder of this section $k$ and $2j$ are fixed and all
constants are allowed to depend on them (and we will not make this dependence
explicit everywhere). Moreover, whenever $\rho_\pm>-2$ we let
$\gamma_{\rho_-,\rho_+}\sim \SLE_\kappa(\rho_-,\rho_+)$ from $-i$ to $i$ in
$\D$. Each event we consider will only depend on one of these curves so we do
not specify a coupling.

In this first lemma, we will rely on Lemma \ref{lem:imaginary-map-in} and
\ref{lem:positive-whole} to perform the construction outlined in the previous
paragraph.

\begin{lemma}
    \label{lem:exp-map-in}
    Let $\kappa\in (0,4)$, $\rho_\pm > -2$ and $\bar{\rho}_\pm\ge 0$.
    Then for all $\lambda>1$ sufficiently close to $1$ there exists
    $C_\lambda<\infty$ such that
    \begin{align*}
        &\P(B'_k(\gamma_{\rho_-,\rho_+},1,\epsilon/\lambda,y,r/\lambda,\lambda
        c, \lambda c_0,
        R/\lambda) )\le C_\lambda
        \P(B'_k(\gamma_{\rho_-+\bar{\rho}_-,\rho_+},1,\epsilon,y,r,c,c_0,R) )
    \end{align*}
    for all $\epsilon>0$ sufficiently small. Moreover, for all $\lambda>1$
    sufficiently close to $1$ there exists $C'_\lambda<\infty$ such that
    \begin{align*}
		\inf_{\abs{z'}<\lambda c} \P(I'_{2j}(\gamma_{\rho_-,\rho_+}
        ,z',\epsilon/\lambda,y,r/\lambda,R/\lambda))
		&\le C'_\lambda \inf_{\abs{z}<c}
        \P(I'_{2j}(\gamma_{\rho_-+\bar{\rho}_-,\rho_++\bar{\rho}_+},
        z,\epsilon,y,r,R))
    \end{align*}
    for all $\epsilon>0$ sufficiently small.
\end{lemma}

\begin{figure}
	\centering
	\def\svgwidth{0.8\columnwidth}
\begingroup%
  \makeatletter%
  \providecommand\color[2][]{%
    \errmessage{(Inkscape) Color is used for the text in Inkscape, but the package 'color.sty' is not loaded}%
    \renewcommand\color[2][]{}%
  }%
  \providecommand\transparent[1]{%
    \errmessage{(Inkscape) Transparency is used (non-zero) for the text in Inkscape, but the package 'transparent.sty' is not loaded}%
    \renewcommand\transparent[1]{}%
  }%
  \providecommand\rotatebox[2]{#2}%
  \newcommand*\fsize{\dimexpr\f@size pt\relax}%
  \newcommand*\lineheight[1]{\fontsize{\fsize}{#1\fsize}\selectfont}%
  \ifx\svgwidth\undefined%
    \setlength{\unitlength}{296.75550795bp}%
    \ifx\svgscale\undefined%
      \relax%
    \else%
      \setlength{\unitlength}{\unitlength * \real{\svgscale}}%
    \fi%
  \else%
    \setlength{\unitlength}{\svgwidth}%
  \fi%
  \global\let\svgwidth\undefined%
  \global\let\svgscale\undefined%
  \makeatother%
  \begin{picture}(1,0.52998907)%
    \lineheight{1}%
    \setlength\tabcolsep{0pt}%
    \put(0,0){\includegraphics[width=\unitlength,page=1]{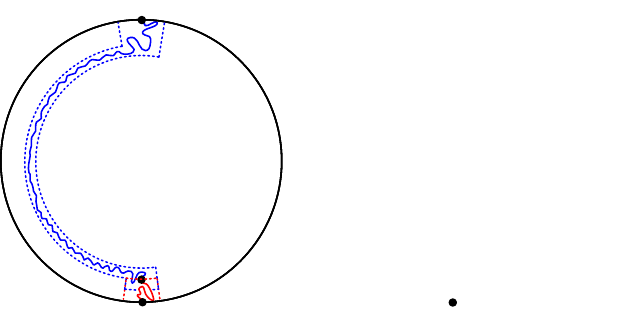}}%
    \put(0.98881232,0.25994945){\makebox(0,0)[lt]{\lineheight{1.25}\smash{\begin{tabular}[t]{l}$1$\end{tabular}}}}%
    \put(0.61066919,0.49840622){\makebox(0,0)[lt]{\lineheight{1.25}\smash{\begin{tabular}[t]{l}$y$\end{tabular}}}}%
    \put(0.71110563,0.00043771){\makebox(0,0)[lt]{\lineheight{1.25}\smash{\begin{tabular}[t]{l}$-i$\end{tabular}}}}%
    \put(0.76562172,0.00103789){\makebox(0,0)[lt]{\lineheight{1.25}\smash{\begin{tabular}[t]{l}$w_0$\end{tabular}}}}%
    \put(0.73728849,0.52455088){\makebox(0,0)[lt]{\lineheight{1.25}\smash{\begin{tabular}[t]{l}$i$\end{tabular}}}}%
    \put(0.76906596,0.52483331){\makebox(0,0)[lt]{\lineheight{1.25}\smash{\begin{tabular}[t]{l}$w_\infty$\end{tabular}}}}%
    \put(0,0){\includegraphics[width=\unitlength,page=2]{exponent-map-in.pdf}}%
  \end{picture}%
\endgroup%

	\caption{\emph{Left.} With positive probability $\gamma_-\sim
		\SLE_\kappa(\bar{\rho}_--2,2+\rho_-+\rho_+)$ hits the top of the red dashed
    box before it hits the sides and after that, it does not hit the the blue
    dashed lines. \emph{Right.} This drawing shows $\gamma_-$
    (in red) and $\gamma_+'$ (in blue). The conditional law of $\gamma_+'$ given
		$\gamma_-$ is given by a $\SLE_\kappa(\rho_-,\rho_+)$ from $w_0$ to
    $w_\infty$ in the connected component right of $\gamma_-$ which has $1$ on
    its boundary. The sets $\partial B_r(y)\cap \D$ and $\partial
    B_\epsilon(1)\cap \D$ are drawn in green and $\partial B_R(-i)\cap \D$ and
    $\partial ( (\!(1/I_{c\epsilon},I_{c\epsilon})\!)+B_{c_0\epsilon}(0) )
    \cap \D$ are drawn dashed in green.}
	\label{fig:exponent-map-in}
\end{figure}

\begin{proof}
	Let $\gamma_-\sim \SLE_\kappa(\bar{\rho}_--2,2+\rho_-+\rho_+)$ be coupled
	with $\gamma_+\sim \SLE_\kappa(\rho_-+\bar{\rho}_-,\rho_+)$ as in Lemma
    \ref{lem:imaginary-map-in}. Let $D$ be the connected component of
    $\D\setminus\gamma_-([0,1])$ with $1$ on its boundary, define $w_0$ to be
    the most counterclockwise point on $(\!(-i,1)\!)\cap \gamma_-([0,1])$ and
    $w_\infty$ to be the most clockwise point on $(\!(1,i)\!)\cap
    \gamma_-([0,1])$. Let $\gamma_+'$ be the part of $\gamma_+$ from $w_0$ to
    $w_\infty$. Then by the lemma, conditionally on $\gamma_-$, the curve
	$\gamma_+'$ is a $\SLE_\kappa(\rho_-,\rho_+)$ in $D$ from $w_0$ to
    $w_\infty$.

    We write $\phi\colon D\to \D$ for the unique conformal transformation with
    $\phi(0)=0$ and $\phi'(0)>0$. Then $\phi$ extends continuously to
    $(\!(w_0,w_\infty)\!)\cup\{w_0,w_\infty\}$ and we write $f$ for the unique
    Möbius transformation from $\D$ to itself mapping
    $(\phi(w_0),\phi(1),\phi(w_\infty))$ to $(-i,1,i)$. We let $\psi = f\circ
    \phi$.

    From Lemma \ref{lem:positive-whole} (applied twice) and the strong Markov
    property for SLE curves, we deduce that for all $\delta\in (0,1/16)$ the
    following event occurs with positive probability (see Figure
    \ref{fig:exponent-map-in}):
    \begin{itemize}
        \item The curve $\gamma_-$ hits $S_T := (1-2\delta)\cdot
            (\!(-i/I_\delta,-iI_\delta)\!)$ before $S_{LR}:= (1-2\delta,1)\cdot
            \{-i/I_\delta,-iI_\delta\}$ and we write $\tau$ for the hitting
            time.
		\item The curve $\gamma_-\lvert_{[\tau,1]}$ hits $S'_T$ before $S'_{LR}:=
            S'_B\setminus S'_T$ where $S'_T:=(\!(i/I_\delta,iI_\delta)\!)$,
            \begin{align*}
                S' &:= (1-3\delta,1-\delta)\cdot (\!(-i/I_\delta,-iI_\delta)\!)
                \cup (1-3\delta,1-2\delta)\cdot (\!(i/I_\delta,-iI_\delta)\!)\\
                &\qquad \cup (1-3\delta,1)\cdot (\!(i/I_\delta,iI_\delta)\!)
            \end{align*}
            and where $S'_B$ is the connected component of $\partial S'\setminus
            \gamma([0,\tau])$ containing $i$.
    \end{itemize}
    Let us call this event $N_\delta$; it ensures that $B_{1-3\delta}(0)\subset
    D$ and that $\gamma_+$ hits $w_0$ before exiting
    $(\!(-i/I_\delta,-iI_\delta)\!)\cdot (1-2\delta,1)$. We will now argue that
    for some (small) $\delta$ we have the inclusion
    \begin{align}
        \label{eq:inclusion-map-in}
        N_\delta\cap
        B'_k(\psi\circ \gamma_+,1,\epsilon/\lambda,y,r/\lambda,\lambda c,
        \lambda c_0, R/\lambda)  \subset
        B'_k(\gamma_+,1,\epsilon,y,r,c,c_0,R)
    \end{align}
    for all sufficiently small $\epsilon$. From this the result follows by
    taking probabilities and using the independence of $\gamma_-$ and
    $\psi\circ\gamma_+$.

    We will now work on the event $N_\delta$. For all $\delta'>0$ by Lemma
    \ref{lem:kernel-result} there exists $\delta$ such that we have
	$\lvert\phi(w)-w\rvert\le \delta'$ for all $w\in D_{(4\delta,8\delta)}$ (we are using
    the notation from Appendix \ref{sec:app-conformal}). We make the following
    key observations:
    \begin{itemize}
        \item The points $w_0,w_\infty,1$ lie in the closure of
            $D_{(4\delta,8\delta)}$, so since $\phi$ extends continuously to
		these points, we also have $\lvert\phi(w)-w\rvert\le \delta'$ for $w\in
            \{w_0,w_\infty,1\}$.
        \item Let $S$ denote the connected component of $\partial
            B_r(y)\setminus \gamma_-([0,1])$ containing $y(1-r)$ (this is
            defined for $3\delta<r$) and let $S'$ denote the connected component
            of $(\partial B_R(-i)\cap \D)\setminus \gamma_-([0,1])$ containing
            $i(R-1)$. Then we have $S,S'\subset D_{(4\delta,8\delta)}$.
    \end{itemize}
    It follows that for $\delta$ sufficiently small, $\psi(S)\cap
    B_{r/\lambda}(y)=\emptyset$, $\psi(S')\cap B_{R/\lambda}(-i)=\emptyset$.
    Moreover, $B_{1/2}(1)\cap\D\subset D$ and so by further decreasing $\delta$
    and using Lemma \ref{lem:boundary-koebe} applied to the function $\psi$ we
    get that for all $x\in (\!(1/I_{1/4},I_{1/4})\!)$ and all sufficiently small
    $\epsilon$,
    \begin{align*}
        B_{\epsilon/\lambda}(\psi(x))\cap\D\subset \psi(B_\epsilon(x)\cap\D)
        \subset B_{\lambda\epsilon}(\psi(x))\cap\D\;.
    \end{align*}
    In particular $B_{\epsilon/\lambda}(1)\cap\D\subset
    \psi(B_\epsilon(1)\cap\D)$ and
    \begin{align*}
        ((\!(1/I_{\lambda c\epsilon},I_{\lambda c\epsilon})\!)+
        B_{\lambda c_0\epsilon}(0))\cap\D
        \supset \psi( ((\!(1/I_{c\epsilon},I_{c\epsilon})\!)+B_{c_0\epsilon}(0))
        \cap\D  )
    \end{align*}
    for $\epsilon$ sufficiently small. Combining these inclusions readily
    implies \eqref{eq:inclusion-map-in} since one can see that if $\psi\circ
    \gamma_+$ satisfies the arm event and $N_\delta$ holds then $\gamma_+$ will
    also satisfy an arm event.

    The result on interior exponents follows similarly in the case where
    $\bar{\rho}_-=0$ or when $\bar{\rho}_+=0$ (now using Lemma
    \ref{lem:interior-koebe} instead of Lemma \ref{lem:boundary-koebe}). The
    general case is obtained by bounding the arm event probability of a
	$\SLE_\kappa(\rho_-,\rho_+)$ by that of a
	$\SLE_\kappa(\rho_-+\bar{\rho}_-,\rho_+)$ and then further by that of a
	$\SLE_\kappa(\rho_-+\bar{\rho}_-,\rho_++\bar{\rho}_+)$.
\end{proof}

The next lemma makes use of Lemma \ref{lem:sle-excursion} and implements a very
similar proof strategy to the one in the previous Lemma \ref{lem:exp-map-in} to
relate arm event probabilities for the two types of SLE appearing in Lemma
\ref{lem:sle-excursion}. Note that this result will only be needed in the
interior case.

\begin{lemma}
    \label{lem:exp-sle-excursion}
    Suppose that $\kappa\in (0,4)$ and $\rho\in (-2,\kappa/2-2)$. For all
    $\lambda>1$ sufficiently close to $1$, there exists $C_\lambda<\infty$ such that
    \begin{align*}
		\inf_{\abs{z'}<\lambda c}
        \P(I'_{2j}(\gamma_{0,\kappa-4-\rho},
		z',\epsilon/\lambda,y,r/\lambda,R/\lambda)) \le C_\lambda \inf_{\abs{z}<c}
        \P( I'_{2j}(\gamma_{0,\rho},z,\epsilon,y,r,R))
    \end{align*}
    for all $\epsilon>0$ sufficiently small.
\end{lemma}

\begin{proof}
	Let $\gamma\sim \SLE_\kappa(0,\rho)$, $\zeta_1$ and $\zeta_{1-}$ be as in
   Lemma \ref{lem:sle-excursion}. We write $D$ for the connected component of
   $\D\setminus (\gamma([0,\zeta_{1-}]\cup [\zeta_1,1]))$ having $1$ on its
	boundary. We write $\gamma'$ for $\gamma\lvert_{[\zeta_{1-},\zeta_1]}$
   (reparametrized so that it is a function on $[0,1]$). Thus by the lemma,
   conditionally on $\gamma([0,\zeta_{1-}])\cup\gamma([\zeta_1,1])$, the curve
	$\gamma'$ is a $\SLE_\kappa(0,\kappa-4-\rho)$ in $D$ from
   $w_0:=\gamma_{\zeta_{1-}}$ to $w_\infty:=\gamma_{\zeta_1}$. We define $\phi$,
   $f$ and $\psi$ as in the proof of Lemma \ref{lem:exp-map-in}.

   Consider $\delta\in (0,1/16)$. We first argue that with positive probability
   $\gamma([0,\zeta_{1-}])\subset [-\delta,\delta]+i[-1,-1+\delta]$ and
   $\gamma([\zeta_1,1])\subset [-\delta,\delta]+i[1-\delta,1]$; we call this
   event $N_\delta$. Indeed, Lemma \ref{lem:positive-whole} (applied three
   times) and the strong Markov property of SLE (applied twice) shows that the
   following event has positive probability (see also Figure
   \ref{fig:exp-continue}):
   \begin{itemize}
       \item Let $S_T$ and $S_{LR}$ be the top and the union of the left and right sides of
            $S=((-\delta,\delta)+i(-1,-1+2\delta))\cap\D$ respectively. The
            curve $\gamma$ hits $S_T$ before $S_{LR}$ and at the time $\tau$.
        \item Let $S'_T = (-\delta,\delta) + (1-\delta)i$ and
            $S'_{LR}=S'_B\setminus S'_T$ where
            $S'=(-\delta,\delta)+i(-1+\delta,1-\delta)$ and where $S'_B$ denotes
            the connected component of $\partial S'\setminus \gamma([0,\tau_1])$
		containing $i(1-\delta)$. The curve $\gamma\lvert_{[\tau,1]}$ hits $S'_T$
            before $S'_{LR}$ at a time $\tau'$.
        \item Let $S''_T$ denote the top of $S''=((-\delta,\delta)+
            i(1-2\delta,1))\cap\D$ and let $S''_{LR}=S''_B\setminus S''_T$ where
            $S''_B$ denotes the connected component of $\partial S''\setminus
		\gamma([0,\tau'])$ containing $i$. Then $\gamma\lvert_{[\tau',1]}$ hits
            $S''_T$ before $S''_{LR}$.
    \end{itemize}
	There exists $\delta>0$ such that for $\abs{z}<c$ one has $\lvert\psi(z)\rvert<\lambda
    c$ on the event $N_\delta$ as well as the inclusion
    \begin{align*}
        N_\delta \cap I'_{2j}(\psi\circ\gamma',
        \psi(z),\epsilon/\lambda,y,r/\lambda,R/\lambda) \subset
        I'_{2j}(\gamma,z,\epsilon,y,r,R)
    \end{align*}
    for all sufficiently small $\epsilon>0$. This follows by a reasoning similar
    to the one in the proof of Lemma \ref{lem:exp-map-in}. The result now
    follows by taking probabilities and using the independence of
    $\psi\circ\gamma'$ and $\gamma([0,\zeta_{1-}])\cup \gamma([\zeta_1,1])$.
\end{proof}

The last lemma leverages Lemma \ref{lem:coordinate-change} in combination with
Lemma \ref{lem:positive-hit} to compare the arm event probabilities (both in the
interior and the boundary case) of the types of SLEs appearing in Lemma
\ref{lem:coordinate-change}. The proof will make it transparent why we included
the set $J$ in the definition of the arm events.

\begin{figure}
	\centering
	\def\svgwidth{0.95\columnwidth}
\begingroup%
  \makeatletter%
  \providecommand\color[2][]{%
    \errmessage{(Inkscape) Color is used for the text in Inkscape, but the package 'color.sty' is not loaded}%
    \renewcommand\color[2][]{}%
  }%
  \providecommand\transparent[1]{%
    \errmessage{(Inkscape) Transparency is used (non-zero) for the text in Inkscape, but the package 'transparent.sty' is not loaded}%
    \renewcommand\transparent[1]{}%
  }%
  \providecommand\rotatebox[2]{#2}%
  \newcommand*\fsize{\dimexpr\f@size pt\relax}%
  \newcommand*\lineheight[1]{\fontsize{\fsize}{#1\fsize}\selectfont}%
  \ifx\svgwidth\undefined%
    \setlength{\unitlength}{432.39230391bp}%
    \ifx\svgscale\undefined%
      \relax%
    \else%
      \setlength{\unitlength}{\unitlength * \real{\svgscale}}%
    \fi%
  \else%
    \setlength{\unitlength}{\svgwidth}%
  \fi%
  \global\let\svgwidth\undefined%
  \global\let\svgscale\undefined%
  \makeatother%
  \begin{picture}(1,0.36276458)%
    \lineheight{1}%
    \setlength\tabcolsep{0pt}%
    \put(0,0){\includegraphics[width=\unitlength,page=1]{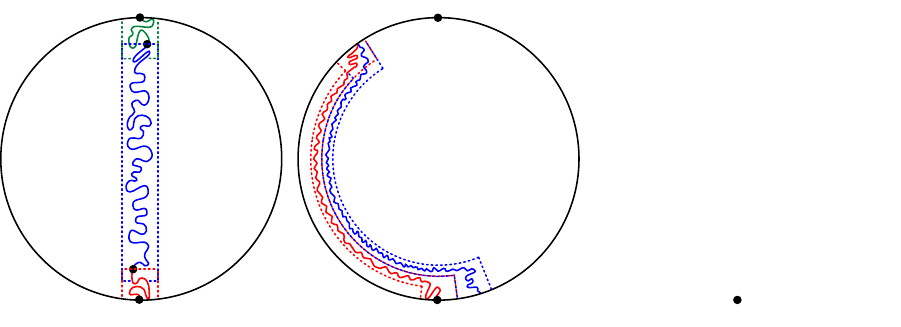}}%
    \put(0.86811517,0.0094831){\makebox(0,0)[lt]{\lineheight{1.25}\smash{\begin{tabular}[t]{l}$w_0$\end{tabular}}}}%
    \put(0.80300491,0.00026828){\makebox(0,0)[lt]{\lineheight{1.25}\smash{\begin{tabular}[t]{l}$-i$\end{tabular}}}}%
    \put(0.66005593,0.28161656){\makebox(0,0)[lt]{\lineheight{1.25}\smash{\begin{tabular}[t]{l}$y$\end{tabular}}}}%
    \put(0.99232177,0.17750177){\makebox(0,0)[lt]{\lineheight{1.25}\smash{\begin{tabular}[t]{l}$1$\end{tabular}}}}%
    \put(0.69107627,0.33515715){\makebox(0,0)[lt]{\lineheight{1.25}\smash{\begin{tabular}[t]{l}$w_\infty$\end{tabular}}}}%
    \put(0.81474257,0.35925388){\makebox(0,0)[lt]{\lineheight{1.25}\smash{\begin{tabular}[t]{l}$i$\end{tabular}}}}%
    \put(0,0){\includegraphics[width=\unitlength,page=2]{exp-continue.pdf}}%
  \end{picture}%
\endgroup%

	\caption{\emph{Left.} With positive probability $\gamma\sim
		\SLE_\kappa(0,\rho)$ hits the top of the red box before it hits the sides,
    then it hits the top dashed blue line before it hits the other blue dashed
    lines and finally then it does not hit any of the green dashed lines.
    \emph{Center.} With positive probability $\gamma\sim
		\SLE_\kappa(\kappa-6-\rho,\rho)$ hits the top boundary interval of the red
    `corridor' before it hits the sides and then it hits the bottom boundary
    interval of the blue `corridor' before it hits the blue dashed lines.
    \emph{Right.} The concatenation $\gamma^0$ of the red and the blue curve
    (the latter is called $\gamma$) is a
		$\SLE_\kappa(\kappa-6-\rho,\rho)$. Moreover, conditionally on the red curve,
		$\gamma$ can be coupled with $\gamma'\sim\SLE_\kappa(0,\rho)$ targeting
    $w_\infty$ which agrees until $\gamma$ disconnects $i$ and $w_\infty$; the
    remainder of $\gamma'$ is drawn dashed in blue.}
	\label{fig:exp-continue}
\end{figure}

\begin{lemma}
    \label{lem:exp-coord}
    Suppose that $\kappa\in (2,4)$ and $\rho\in (-2,\kappa-4)$. Then for all
    $\lambda>1$ sufficiently close to $1$, there exists $C_\lambda<\infty$ such
    that
    \begin{align*}
        \P(B'_k(\gamma_{0,\rho},1,\epsilon/\lambda,y,r/\lambda,\lambda c,
        \lambda c_0,R/\lambda) ) &\le C_\lambda
        \P(B_k(\gamma_{\kappa-6-\rho,\rho},1,\epsilon,y,r) )
    \end{align*}
    for all $\epsilon>0$ sufficiently small. Moreover, for all $\lambda>1$
    sufficiently close to $1$ there exists $C_\lambda'<\infty$ such that
    \begin{align*}
		\inf_{\abs{z'} <\lambda c} \P(I'_{2j}(\gamma_{0,\rho},
        z',\epsilon/\lambda,y,r/\lambda,R/\lambda) )
		&\le C'_\lambda \inf_{\abs{z}<c}
        \P(I_{2j}(\gamma_{\kappa-6-\rho,\rho},z,\epsilon,y,r,R) )
    \end{align*}
    for all $\epsilon>0$ sufficiently small.
\end{lemma}

\begin{proof}
	Let $\gamma^0\sim \SLE_\kappa(\kappa-6-\rho,\rho)$ from $-i$ to $i$ in $\D$.
    Consider $\delta\in (0,1/16)$. We first argue that the following event
    occurs with positive probability (see Figure \ref{fig:exp-continue}):
    \begin{itemize}
        \item Let $S_T=(\!(iI_\delta,iI_{3\delta})\!)$, $S_{LR}=S_B\setminus
            S_T$, $S=(1-2\delta,1)\cdot (\!(-i/I_\delta,iI_\delta)\!) \cup
            (1-2\delta,1-\delta)\cdot (\!(iI_\delta,-iI_\delta)\!) \cup
            (1-2\delta,1)\cdot (\!(iI_\delta,iI_{3\delta})\!)$ and
            $S_B=\partial S\setminus (\!(-i/I_\delta,-iI_\delta)\!)$. Then the
            curve $\gamma^0$ hits $S_T$ before hitting $S_{LR}$. We write $\tau$
            for the hitting time.
        \item Let $S'_T = (\!(-iI_\delta,-iI_{3\delta})\!)$,
            $S'_{LR}=S'_B\setminus S'_T$ where $S'= (1-3\delta,1)\cdot
            (\!(iI_\delta,iI_{3\delta})\!) \cup (1-3\delta,1-2\delta)\cdot
            (\!(iI_\delta,-iI_{3\delta})\!) \cup (1-3\delta,1)\cdot
            (\!(-iI_\delta,-iI_{3\delta})\!)$ and where $S'_B$ denotes the
		connected component of $\partial S'\setminus (\gamma^0\lvert_{[0,\tau]}
            \cup (\!( iI_\delta,iI_{3\delta} )\!) )$ containing $-iI_{2\delta}$.
		Then $\gamma^0\lvert_{[\tau,1]}$ hits $S'_T$ before $S'_{LR}$ and we
            write $\tau'$ for the hitting time.
    \end{itemize}
    Indeed, this follows by applying Lemma \ref{lem:positive-hit} (twice)
    together with the strong Markov property of SLE. From now on, we will work
    on the event which is described above. Let $D$ be the connected component of
    $\D\setminus \gamma^0([0,\tau'])$ with $1$ on its boundary, write
    $w_0=\gamma^0_{\tau'}$ and let $w_\infty$ be the most clockwise point on
    $(\!(i,-i)\!)\cap \gamma^0([0,\tau'])$. Define $\phi$, $f$ and $\psi$ as in
    the proof of Lemma \ref{lem:exp-map-in}.

	We now write $\gamma$ for $\gamma^0\lvert_{[\tau',1]}$ (reparametrized to be a
    function on $[0,1]$) and using Lemma \ref{lem:coordinate-change} couple it
    with a curve $\gamma'$ in $D$ from $w_0$ to $w_\infty$ which agrees with
    $\gamma$ until it hits $(\!(i,w_\infty)\!)$ (at the time $\nu$) and is a
	$\SLE_\kappa(0,\rho)$ in $D$ from $w_0$ to $w_\infty$ conditionally on
    $\gamma^0([0,\tau'])$. By the same reasoning as in the proof of Lemma
    \ref{lem:exp-map-in} we obtain that there exists $\delta$ such that for all
    $\epsilon$ sufficiently small, we have
    \begin{align*}
		N_\delta\cap B'_k(\psi\circ \gamma'\lvert_{[0,\nu]},1,
        \epsilon/\lambda,y,r/\lambda, \lambda c,\lambda c_0,R/\lambda) &\subset
        B_k(\gamma^0,1,\epsilon,y,r)\;.
    \end{align*}
	Also, there exists $\delta$ such that for $\lvert z\rvert<c$ one has $\lvert\psi(z)\rvert<\lambda
    c$ on the event $N_\delta$ and for $\epsilon>0$ sufficiently small,
    \begin{align*}
		N_\delta \cap I'_{2j}( \psi\circ\gamma'\lvert_{[0,\nu]},
        \psi(z),\epsilon/\lambda,y,r/\lambda,R/\lambda) &\subset
        I_{2j}(\gamma^0,z,\epsilon,y,r,R) \;.
    \end{align*}
    It remains to show that we can drop the restriction to the interval
    $[0,\sigma]$ since the result then follows by taking probabilities and using
    the independence of $\gamma^0([0,\sigma])$ and $\psi\circ\gamma'$.

    In the case of the boundary arm events this is clear since the swallowing
    time of the point $1$ by the curve $\psi\circ\gamma'$ appearing in the
    definition of the arm event occurs (not necessarily strictly) before $\nu$.
    For the case of the interior arm event, we suppose that $\sigma<\tau_j$, and
    hence $\tau_1<\sigma<\tau_j$ and therefore the point $\psi(z)$ lies left of
    $\psi\circ\gamma'([0,\nu])$ which contradicts the definition of the interior
    point arm event.
\end{proof}

Proposition \ref{prop:main-exp-boundary} and \ref{prop:main-exp-interior} are
now obtained by combining Lemma \ref{lem:exp-map-in},
\ref{lem:exp-sle-excursion} and \ref{lem:exp-coord} in a rather
straightforward way.

\begin{proof}[Proof of Proposition \ref{prop:main-exp-boundary}]
    The upper bound follows from the upper bound of Theorem \ref{thm:wu-boundary}
    together with the first display in Lemma \ref{lem:exp-map-in} with
    $\rho_+=\rho$, $\rho_-=\kappa-6-\rho$ and $\bar{\rho}_+ = -\rho_-$.
    The lower bound follows from the lower bound in Theorem
    \ref{thm:wu-boundary} and the first inequality of Lemma \ref{lem:exp-coord}.
\end{proof}

\begin{proof}[Proof of Proposition \ref{prop:main-exp-interior}]
    The upper bound follows from the upper bound in Theorem \ref{thm:wu-interior}
    together with the second display in Lemma \ref{lem:exp-map-in} with
    $\rho_+=\rho$, $\rho_-=\kappa-6-\rho$, $\bar{\rho}_+=-\rho_+$ and
    $\bar{\rho}_- = -\rho_-$. The lower bound is obtained as follows: We use the
    lower bound from Theorem \ref{thm:wu-interior} and then apply the lemmas
    above in the following settings:
    \begin{itemize}
        \item Use Lemma \ref{lem:exp-map-in} to bound the arm event probability
		of a $\SLE_\kappa$ by that of a $\SLE_\kappa(0,\kappa-4-\rho)$ using
            that $\kappa-4-\rho>0$.
        \item Use Lemma \ref{lem:exp-sle-excursion} to bound the arm event
		probability of a $\SLE_\kappa(0,\kappa-4-\rho)$ curve by that of a
		$\SLE_\kappa(0,\rho)$.
        \item Use Lemma \ref{lem:exp-coord} to bound the arm event probability
		of a $\SLE_\kappa(0,\rho)$ curve by that of a
		$\SLE_\kappa(\kappa-6-\rho,\rho)$.
    \end{itemize}
    Chaining these inequalities yields the claim.
\end{proof}

\section{Combining the discrete with the continuum results}
\label{sec:combine-results}

The proofs of our main theorems are now standard by combining the tools
collected in the previous sections.

\begin{proof}[Proof of Theorem \ref{thm:main-bulk}]
    Let us first consider the case where $\tau$ is alternating, i.e.\
	$\lvert\tau\rvert=I(\tau)=2j$. In this case, we see using Proposition
    \ref{prop:main-exp-interior}, Theorem \ref{thm:interface-convergence} and
    Theorem \ref{thm:arm-separation} (and mixing) by the same reasoning as in
    \cite{wu-ising-arm} that
    \begin{align*}
        \liminf_{N\to\infty}\mu_{\Z^2}(A_\tau(\epsilon N,N))\asymp \limsup_{N\to
			\infty} \mu_{\Z^2}(A_\tau(\epsilon N,N)) =
         \epsilon^{\alpha_{2j}(r)+o(1)} \quad \text{as $\epsilon\to 0$.}
    \end{align*}
	In fact, we note that if the limiting curve appearing in the statement of Theorem  \ref{thm:interface-convergence} satisfies the arm event appearing in Proposition \ref{prop:main-exp-interior}, then the two approximating discrete fuzzy Potts interfaces from Theorem \ref{thm:interface-convergence} will also satisfy the arm event simultaneously for sufficiently small lattice mesh size. In particular, we obtain
    \begin{align*}
	\liminf_{N\to\infty}\mu_{\Z^2}(A^s_\tau(\epsilon N,N))\asymp \limsup_{N\to
			\infty} \mu_{\Z^2}(A^s_\tau(\epsilon N,N))= 
	\epsilon^{\alpha_{2j}(r)+o(1)} \quad \text{as $\epsilon\to 0$.}
\end{align*}
     The theorem in the case where $\tau$ is alternating is
    then a consequence of quasi-multiplicativity (see Theorem
    \ref{thm:quasi-multiplicativity}) and follows again as in
    \cite{wu-ising-arm}. In particular, this shows that the exponent in the case
    of four alternating arms is $>2$ and so the general case of the theorem is a
    consequence of Proposition \ref{prop:alternating-non-alternating}.
\end{proof}

\begin{proof}[Proof of Theorem \ref{thm:main-boundary}]
    Again, let us first consider the case where $\tau$ is alternating and
    $I(\tau)=k$. In this case, using Proposition \ref{prop:main-exp-boundary},
    Theorem \ref{thm:interface-convergence} and Theorem
    \ref{thm:arm-separation-halfplane} (and mixing), we get by the same arguments as in
    \cite{wu-ising-arm} and as before that
    \begin{align*}
        &\liminf_{N\to\infty}\mu_{\Z\times\Z_+}^0(A^+_\tau(\epsilon N,N))\asymp \limsup_{N\to
        \infty} \mu_{\Z\times\Z_+}^0(A^+_\tau(\epsilon N,N))\\
        &\qquad \asymp
        \liminf_{N\to\infty}\mu_{\Z\times\Z_+}^0(A^{+s}_\tau(\epsilon N,N))\asymp \limsup_{N\to
			\infty} \mu_{\Z\times\Z_+}^0(A^{+s}_\tau(\epsilon N,N))\asymp \epsilon^{\alpha^+_k(r)}.
    \end{align*}
    Again, the theorem follows in the case where $\tau$ is alternating by
    quasi-multiplicativity (see Theorem
    \ref{thm:quasi-multiplicativity-halfplane}) as in \cite{wu-ising-arm} and
    the general case can be deduced from Remark
    \ref{rk:halfplane-non-alternating} by making use of Theorem
    \ref{thm:main-bulk} to see that the alternating four-arm exponent is $>2$.
\end{proof}

\appendix

\section{A technical lemma on the loop topology}
\label{sec:app-loops}

In this section we will prove a simple lemma which gives a practical criterion for
checking convergence with respect to $d_\mathcal{C}$. Roughly speaking, if
a curve is simple and is approximated in Hausdorff distance by a sequence of
curves which do not oscillate too much, then one can deduce convergence with
respect to the metric $d_\mathcal{C}$. The scenario that needs to be excluded is
that the approximating curves trace macroscopic segments of the limiting loop
several times. The proofs are illustrated in Figure \ref{fig:technical-loop}.

\begin{figure}
	\centering
	\def\svgwidth{0.8\columnwidth}
\begingroup%
  \makeatletter%
  \providecommand\color[2][]{%
    \errmessage{(Inkscape) Color is used for the text in Inkscape, but the package 'color.sty' is not loaded}%
    \renewcommand\color[2][]{}%
  }%
  \providecommand\transparent[1]{%
    \errmessage{(Inkscape) Transparency is used (non-zero) for the text in Inkscape, but the package 'transparent.sty' is not loaded}%
    \renewcommand\transparent[1]{}%
  }%
  \providecommand\rotatebox[2]{#2}%
  \newcommand*\fsize{\dimexpr\f@size pt\relax}%
  \newcommand*\lineheight[1]{\fontsize{\fsize}{#1\fsize}\selectfont}%
  \ifx\svgwidth\undefined%
    \setlength{\unitlength}{317.06407213bp}%
    \ifx\svgscale\undefined%
      \relax%
    \else%
      \setlength{\unitlength}{\unitlength * \real{\svgscale}}%
    \fi%
  \else%
    \setlength{\unitlength}{\svgwidth}%
  \fi%
  \global\let\svgwidth\undefined%
  \global\let\svgscale\undefined%
  \makeatother%
  \begin{picture}(1,0.42924073)%
    \lineheight{1}%
    \setlength\tabcolsep{0pt}%
    \put(0,0){\includegraphics[width=\unitlength,page=1]{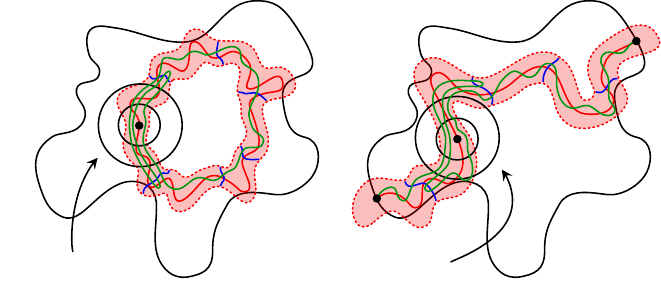}}%
    \put(0.44737207,0.09578785){\makebox(0,0)[lt]{\lineheight{1.25}\smash{\begin{tabular}[t]{l}\textcolor{red}{$U_r$}\end{tabular}}}}%
    \put(0.43894907,0.02360909){\makebox(0,0)[lt]{\lineheight{1.25}\smash{\begin{tabular}[t]{l}$B_\delta(\eta_{t'_i})\setminus B_{\delta'}(\eta_{t'_i}$)\end{tabular}}}}%
    \put(-0.00039109,0.00099664){\makebox(0,0)[lt]{\lineheight{1.25}\smash{\begin{tabular}[t]{l}$B_\delta(\eta_{t'_i})\setminus B_{\delta'}(\eta_{t'_i}$)\end{tabular}}}}%
  \end{picture}%
\endgroup%

	\caption{\emph{Left.} Illustration of the proof of Lemma
	\ref{lem:technical-loop}. \emph{Right.} Illustration of the proof of Lemma
	\ref{lem:technical-curve}. In the left (resp.\ right) figure, $\eta$ (resp.
	$\gamma$) is drawn in red, $U_r$ is shaded in red, $\eta^n$ (resp.
	$\gamma^n$) is drawn in green and the segments $I_r^0,\dots,I_r^{M-1}$ are
	shown in blue. Both figures display the scenario where the green curve has
	at least three crossings between two blue segments; this results in six
	crossings of the annulus $B_\delta(\eta_{t_i'})\setminus
	B_{\delta'}(\eta_{t_i'})$ which by assumption does not occur for
	$\delta'\in (0,\delta)$ sufficiently small.}
	\label{fig:technical-loop}
\end{figure}
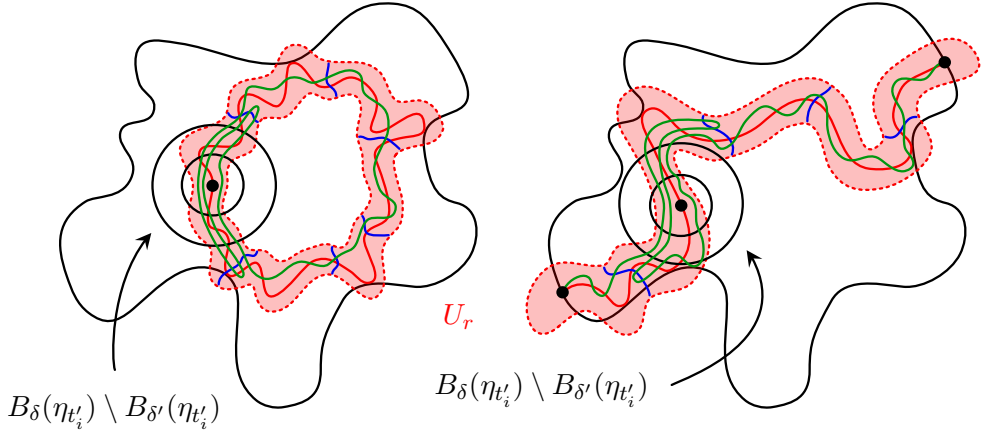

\begin{lemma}
	\label{lem:technical-loop}
	Let $D$ be a Jordan domain and consider a simple loop $\eta\colon
    \partial\D\to \overline{D}$ and $\eta^n\in C^s(\partial\D,\C)$ all
    surrounding a fixed point $z_0\in D$. Suppose that $Z=\{t\in \partial\D
    \colon \eta_t\in \partial D\}$ has empty interior and that
	\begin{align}
		\label{eq:conv-assump-technical}
		\sup_{s\in \partial\D}\inf_{t\in \partial\D} \lvert\eta_t -\eta^n_s\rvert\to 0
		\quad\text{as $n\to\infty$}\;.
	\end{align}
	Also assume that for each $\delta>0$ there is $\delta' \in (0,\delta)$ with
	the following property: For all $n\ge 1$, whenever $z\in D$ is such that
	$B_{2\delta}(z)\subset D$, the annulus $B_\delta(z)\setminus B_{\delta'}(z)$
	is not crossed six times (or more) by $\eta^n$. Then
	$d_\mathcal{C}(\eta^n,\eta)\to 0$ as $n\to \infty$.
\end{lemma}

\begin{proof}
	The proof relies on finding suitable reparametrizations of the curves
	$\eta^n$. Without loss of generality, all loops are oriented
	counterclockwise around $z_0$. Consider $\epsilon>0$ and let $M\ge 1$ be
	maximal such that there are points $t_0,\dots,t_{M-1}\in \partial\D$
	(ordered counterclockwise) with $\lvert\eta_{t_i}-\eta_{t_{i+1}}\rvert\ge \epsilon$
	for all $i<M$ (addition modulo $M$).

	Let $D_-$ and $D_+$ be the bounded and
	unbounded component of $\C\setminus \eta(\partial\D)$ respectively. Let
	$\psi_-\colon \D\to D_-$ and $\psi_\infty\colon \C\setminus\D\to D_+$
	be conformal transformations with $\psi_-(0)=z_0$. By Carathéodory's theorem
	(see e.g.\ \cite[Theorem 14.5.6]{conway-complex-ii}), $\psi_\pm$
	extend continuously to the boundary since $\eta$ is a continuous simple
	loop. Since $\psi_\pm\lvert_{\partial\D}$ are reparametrizations of $\eta$, for
	all $i<M$ there exist $t^\pm_i\in \partial\D$ with
	$\psi_+(t^+_i)=\psi_-(t^-_i)=\eta_{t_i}$. For $r\in (0,1)$ and $i<M$ we let
	\begin{align*}
		U_r&=\psi_-(\{z\in \C\colon 1-r< \abs{z}\le 1\})\cup \psi_+(\{z\in
		\C\colon 1\le \abs{z}< 1+r\})\,,\\
		I_r^i &= \psi_-(t^-_i\cdot (1-r,1])\cup \psi_+(t^+_i\cdot [1,1+r))\;,\\
		S_r^i &= I_r^i\cup \psi_-( (\!(t_i,t_{i+1})\!)\cdot (1-r,1]) \cup
		\psi_+( (\!(t_i,t_{i+1})\!)\cdot [1,1+r))\cup I_r^{i+1}\;.
	\end{align*}
	Note that $U_r$ has the topology of an annulus (which $\eta$ follows in
	a counterclockwise way) and that $I_r^0,\dots,I_r^{M-1}$ are `radial
	segments' within $U_r$ and splitting $U_r$ into $S_r^0,\dots,S_r^{M-1}$. For
	each $n\ge 1$ with $\eta^n(\partial\D)\subset U_r$, the curve $\eta^n$
	follows the annulus $U_r$ in counterclockwise way and we may take
	$s^n_0,\dots,s^n_{M-1}\in \partial\D$ which are distinct and
	counterclockwise ordered such that
	\begin{align*}
		\eta^n_{s^n_i}\in I^r_i\quad\text{for all $i<M$}\;.
	\end{align*}
	In this case, we pick an orientation-preserving homeomorphism $\phi_n\colon
	\partial\D\to \partial\D$ such that $\phi_n(t_i)=s^n_i$ for all $i<M$ and
	our goal will be to bound  $\|\eta-\eta^n\circ\phi_n\|_\infty$.

	Since $Z$ has empty interior, we see (arguing by contradiction) that there
	is $\epsilon'>0$ such that for all $u,v\in \partial \D$ with
	$\lvert\eta_u-\eta_v\rvert\ge \epsilon$ there exist $t\in (\!(u,v)\!)$ such that
	$\lvert\eta_u-\eta_t\rvert,\lvert\eta_v-\eta_t\rvert\ge \epsilon/4$ and $\dist(\eta_t,\partial
	D)>\epsilon'$. For any $i<M$, we apply this with $u=t_i$ and $v=t_{i+1}$ to
	obtain a point $t_i'\in (\!(t_i, t_{i+1})\!)$ with the stated property.

	Fix any $\delta<(\epsilon'/2)\wedge (\epsilon/8)$ and let
	$\delta'\in (0,\delta)$ be chosen as in the statement of the lemma.
	Moreover, take $r\in (0,1)$ sufficiently small such that
	\begin{align}
		\label{eq:technical-psi-func}
		\sup_{t\in \partial\D}\sup_{c\in (1-r,1]} \lvert\psi_-(ct)-\psi_-(t)\rvert\;,
		\sup_{t\in \partial\D}\sup_{c\in [1,1+r)} \lvert \psi_+(ct)-\psi_+(t)\rvert <
		\delta' \;.
	\end{align}
	By \eqref{eq:conv-assump-technical}, there exists $N\ge 1$ such that
	$\eta^n(\partial\D)\subset U_r$ (and $\eta^n$ follows the annulus $U_r$ in a
	counterclockwise way) for all $n\ge N$. Fix any such $n\ge N$.
	Since $\delta'<\delta<\epsilon/8$ we get $\delta <  \epsilon/4-\delta'$ and
	we also have $2\delta<\epsilon'$. Hence
	\begin{align*}
		 I^i_r\cap B_{\delta}(\eta_{t_i'})=I^{i+1}_r\cap
		 B_{\delta}(\eta_{t_i'})=\emptyset\quad\text{and}\quad
		B_{2\delta}(\eta_{t_i'})\subset D
	\end{align*}
	by \eqref{eq:technical-psi-func} and the
	definition of $t_i'$.
	Therefore the maximal disjoint number of subintervals of $\partial\D$ where
	$\eta^n$ (in its counterclockwise
	parametrization) crosses from $I^i_r$ to $I^{i+1}_r$ is $1$ since
	otherwise the annulus $B_{\delta}(\eta_{t_i'})\setminus
	B_{\delta'}(\eta_{t_i'})$ would be crossed at least six times by $\eta^n$
	(see Figure \ref{fig:technical-loop}). This implies that
	\begin{align*}
		\eta^n_{s}\in S^{i-1}_r\cup S^i_r\cup S^{i+1}_r\quad\text{for $s\in
		(\!(s^n_i,s^n_{i+1})\!)$}\;.
	\end{align*}
	Indeed, if this did not hold we would cross between
	two radial segments in counterclockwise direction more than once (again, see
	Figure \ref{fig:technical-loop}).

	If $t\in (\!(t_i,t_{i+1})\!)$ then by maximality of $t_0,\dots,t_{M-1}$ we
	deduce that $\lvert\eta_t-\eta_{t_i}\rvert\le 2\epsilon$. By maximality also
	$\diam(\eta([t_i,t_{i+1}]))\le 2\epsilon$ for all $i$ and hence for
	$s\in (\!(s^n_i,s^n_{i+1})\!)$,
	\begin{align*}
		\lvert \eta^n_{s^n_i}-\eta^n_s\rvert \le \diam(S^{i-1}_r\cup S^{i}_r\cup
		S^{i+1}_r) \le 3\cdot 2\epsilon+2\delta\le 7\epsilon\;.
	\end{align*}
	Clearly $\lvert\eta_{t_i}-\eta^n_{s^n_i}\rvert \le \delta'\le \epsilon/8\le \epsilon$.
	Combining everything yields
	$\|\eta-\eta^n\circ\phi_n\|_\infty\le 10\epsilon$.
\end{proof}

There is an analogous statement in the case of curves which we state here as
well. Since the proof is very similar to the one for Lemma
\ref{lem:technical-loop} we will only sketch the proof.

\begin{lemma}
	\label{lem:technical-curve}
    Let $D$ be a Jordan domain. Consider a simple curve $\gamma\colon [0,1]\to
    \overline{D}$ and $\gamma^n\in C^s([0,1],\C)$. We suppose that $Z=\{t\in
    [0,1]\colon \gamma_t\in \partial D\}$ has empty interior. We also assume
    that
	\begin{align*}
		\sup_{s\in [0,1]}\inf_{t\in [0,1]} \lvert \gamma_t - \gamma^n_s\rvert  \to 0\quad
        \text{and}\quad (\gamma^n_0,\gamma^n_1)\to(\gamma_0,\gamma_1)\quad
		\text{as $n\to \infty$}\;.
	\end{align*}
	Again we assume that for each $\delta>0$ there is $\delta' \in (0,\delta)$
	with the following property: For all $n\ge 1$, whenever $z\in D$ is such
	that $B_{2\delta}(z)\subset D$, the annulus $B_\delta(z)\setminus
	B_{\delta'}(z)$ is not crossed six times (or more) by $\gamma^n$. Then
	$d_{\mathcal{C}'}(\gamma^n,\gamma)\to 0$ as $n\to \infty$.
\end{lemma}

\begin{proof}
	Fix $\epsilon>0$ and let $M\ge 1$ be maximal such that there are points
	$t_0,\dots,t_{M-1}\in (0,1)$ (in increasing order) with
	$\lvert \gamma_{t_i}-\gamma_{t_{i+1}}\rvert \ge \epsilon$ for all $0\le i<M$. Let
	$\psi\colon \C\setminus \D\to \C\setminus \gamma([0,1])$ be a conformal
	transformation. By a version of Carathéodory's theorem \cite[Theorem
	14.5.5]{conway-complex-ii} $\psi$ extends continuously to $\partial \D$
    and we may take $x,y\in \partial\D$ such that $\psi(x)=\gamma_0$ and
    $\psi(y)=\gamma_1$. Let $I^+=(\!(x,y)\!)$ and $I^-=(\!(y,x)\!)$. 
    
    Since
	$\psi\lvert_{I^\pm}\colon I^\pm\to \gamma( (0,1))$ are homeomorphisms, there
    exist $t^\pm_0,\dots,t^\pm_{M-1}\in I^\pm$ such that
    $\psi(t^+_i)=\psi(t^-_i)=\eta_{t_i}$ for all $i<M$. For $r\in (0,1)$ and
    $i<M$ we now let
	\begin{align*}
		U_r &= \psi(\{z\in \C\colon \abs{z}\in [1,1+r)\})\;,\\
		I^i_r &= \psi(t^-_i[1,1+r))\cup \psi(t^+_i[1,1+r))\;.
	\end{align*}
	If $\gamma^n([0,1])\subset U_r$ then we can take $s^n_0,\dots,s^n_{M-1}\in
	(0,1)$ (in increasing order) such that $\gamma^n_{s^n_i}\in I^i_r$ for all
	$i<M$ and we let $\phi_n\colon [0,1]\to [0,1]$ be an increasing
	homeomorphisms with $\phi_n(t_i)=s^n_i$ for all $i$. The remainder of the
	argument is virtually identical to the one of Lemma \ref{lem:technical-loop}
	and is thus omitted.
\end{proof}

\section{Distortion estimates up to the boundary}
\label{sec:app-conformal}

In Section \ref{sec:continuum-exp} we need to carefully consider the
distortion of points, balls and semiballs by conformal transformations. The
complex analysis input is given by the following lemma.
We write $\phi_D\colon D\to \D$ for the
unique conformal transformation satisfying $\phi_D(0)=0$ and $\phi_D'(0)>0$
whenever $D\subset \D$ is a simply connected domain containing $0$. By Schwarz
reflection, $\phi_D$ continuously extends to the closure of $(\partial
D\cap\partial\D)^o$. We write $D_{(\delta,\delta_0)}$ for the union of
$B_{1-\delta}(0)$ and the collection of all points in $z\in D\setminus
B_{1-\delta}(0)$ which can be connected to $\partial B_{1-\delta}(0)$ by a curve
which remains in $D\setminus B_{1-\delta}(0)$ and has diameter $\le \delta_0$.

\begin{lemma}
    \label{lem:kernel-result}
	For all $\epsilon>0$ there exists $\delta_0\in (0,1)$ with the following
	property: For all $\delta\in (0,\delta_0)$ there exists $\delta'\in
	(0,\delta)$ such that $\lvert \phi_D(z)-z\rvert \le \epsilon$ whenever $z\in
	D_{(\delta,\delta_0)}$ and $B_{1-\delta'}(0)\subset D\subset \D$.
\end{lemma}

\begin{proof}
	We first define $\delta_0$. Let $B$ be a planar Brownian motion started from
	$0$ and write $\tau_K=\inf\{t\ge 0\colon B_t\in K\}$ for the hitting time of
	$K$ by the Brownian motion. Consider $\delta\in (0,\delta_0)$, $z\in
	D_{(\delta,\delta_0)}\setminus B_{1-\delta}(0)$ and let $\gamma\colon
	[0,1]\to D\setminus B_{1-\delta}(0)$ be a curve from a point $z_0\in
	\partial B_{1-\delta}(0)$ to $z$ with diameter $\le \delta_0$. Then by
	conformal invariance of $B$ (for the first equality) and inclusion
	estimates, we get
	\begin{align*}
		\P(\tau_{\phi_D(\gamma([0,1]))} < \tau_{\partial\D}) &= \P(
		\tau_{\gamma([0,1])} < \tau_{\partial D}) \le \P(\tau_{\gamma([0,1])}
		< \tau_{\partial \D})\\
		&\le \P(\tau_{\bar{B}_{\delta_0}(0)+[1-\delta_0,1]} <
		\tau_{\partial\D}) =: p(\delta_0) \to 0\quad\text{as $\delta_0\to 0$}\;.
	\end{align*}
	Note that $\phi_D\circ \gamma$ is a curve from $\phi_D(z_0)$ to $\phi_D(z)$.
	Let
	\begin{align*}
		q(\epsilon) = \inf_{\gamma'} \P(\tau_{\gamma'([0,1])}<\tau_{\partial\D})
	\end{align*}
	where the infimum is taken over all curves $\gamma'\colon [0,1]\to \D$ such
	that $\lvert \gamma'_0-\gamma'_1\rvert \ge \epsilon/3$. One can check that
	$q(\epsilon)>0$. Take $\delta_0$ sufficiently small such that
	$p(\delta_0)<q(\epsilon)$. Then by the definition of $q(\epsilon)$ we
	deduce that
	$\lvert \phi_D(z_0)-\phi_D(z)\rvert =\lvert \phi_D(\gamma_0)-\phi_D(\gamma_1)\rvert <\epsilon/3$. By
	further decreasing $\delta_0$ we can ensure that $\delta_0<\epsilon/3$ and
	$\delta_0<1$.

	By the Carathéodory kernel convergence theorem, given $\delta\in
	(0,\delta_0)$ there now exists $\delta'\in (0,\delta)$ such that
	$\lvert \phi_D(z)-z\rvert \le \epsilon/3$ for all $z\in B_{1-\delta}(0)$ whenever
	$B_{1-\delta'}(0)\subset D\subset \D$. The Carathéodory kernel convergence
	theorem is usually stated in sequential form
	\cite[Theorem 15.4.10]{conway-complex-ii} but
	the version in the previous sentence follows by arguing via contradiction.

	To conclude, let $z\in D_{(\delta,\delta_0)}$. If $z\in B_{1-\delta}(0)$
	then the claim is immediate. Otherwise, take $z_0\in \partial
	B_{1-\delta}(0)$ as in the first part of the proof and observe that
	\begin{align*}
		\lvert \phi_D(z)-z\rvert  \le \lvert \phi_D(z_0)-z_0 \rvert  + \lvert z_0-z\rvert  + \lvert \phi_D(z)-\phi_D(z_0)\rvert \le
		\epsilon/3+\delta_0+\epsilon/3\le \epsilon
	\end{align*}
	as required.
\end{proof}

We will next make use of the following distortion estimate (see \cite[Theorem
14.7.9]{conway-complex-ii}). Let $\phi\colon B_{r_0}(z_0)\to \C$ be holomorphic
and injective, then
\begin{align*}
	\frac{\lvert\phi'(z_0)\rvert\cdot\lvert z-z_0\rvert }{(1+\lvert z-z_0\rvert /r_0)^2} \le
	\lvert \phi(z)-\phi(z_0)\rvert \le \frac{\lvert\phi'(z_0)\rvert \cdot\lvert z-z_0\rvert }{(1-\lvert z-z_0\rvert /r_0)^2}
\end{align*}
It follows that for $r\in (0,r_0)$ we have
\begin{gather*}
	B_{\lvert\phi'(z_0)\rvert r/(1+r/r_0)^2}(\phi(z_0)) \subset \phi(B_r(z_0)) \subset
	B_{ \lvert\phi'(z_0)\rvert r/(1-r/r_0)^2 }(\phi(z_0)) \;.
\end{gather*}

The following two lemmas apply this result and explain how small balls
around boundary and interior points are distorted by conformal transformations.
We begin with the statement and proof in the boundary case.

\begin{lemma}
    \label{lem:boundary-koebe}
    Assume that $\psi_D\colon D\to \D$ is a conformal transformation and that
    $B_{r_0}(x)\cap \D\subset D$ for $x\in \partial \D$. Fix $\delta\in (0,1)$
	and suppose that $\lvert\psi_D(z)-z\rvert\le \delta^2 r_0/2$ for all $z\in
    B_{r_0}(x)\cap \D$. Then for all $r'\le r_0\delta$,
	\begin{align*}
		B_{r'(1-\delta)^5 }(\psi_D(x))\cap\D\subset
        \psi_D(B_{r'}(x)\cap \D) \subset
        B_{r'/(1-\delta)^5}(\psi_D(x))\cap \D\;.
	\end{align*}
\end{lemma}

\begin{proof}
    Let $r=r_0\delta<r_0$ and $\epsilon = r\delta/2=\delta^2 r_0/2$. We consider
    the Schwarz reflection of $\psi_D$ across the unit circle. Formally, we
    extend $\psi_D$ to a conformal transformation on a larger domain by
	\begin{align*}
        \psi_D(z)= 1/\overline{\psi_D(1/\overline{z})}\quad\text{for}\quad z\in
        D:= D\cup I\cup \{1/\overline{z}\colon z\in D\}
	\end{align*}
	where $I$ is the connected component of $x$ in $(\partial D\cap
	\partial\D)^o$. Our assumptions now imply that $B_{r-2\epsilon}(\psi_D(x)\cap
	\D)\subset \psi_D(B_{r}(x)\cap \D) \subset B_{r+2\epsilon}(\psi_D(x))\cap \D$.
	This statement extends via Schwarz reflection to $B_{r-2\epsilon}(\psi_D(x))
	\subset \psi_D(B_{r}(x)) \subset B_{r+2\epsilon}(\psi_D(x))$. Note that $x\in
	D'$ is an interior point of the extended domain. By the distortion estimate
    therefore
    \begin{align*}
		\frac{\lvert\psi_D'(x)\rvert r}{(1+r/r_0)^2}\le r+2\epsilon \quad\text{and}\quad
		r-2\epsilon \le \frac{\lvert\psi_D'(x)\rvert r}{(1-r/r_0)^2}\;.
    \end{align*}
	Thus $(1-\delta)^3\le \lvert\psi_D'(x)\rvert\le (1+\delta)^3$. By the distortion
    theorem, we obtain for $r'\le r_0$ that
	\begin{align*}
		B_{(1-\delta)^3 r'/(1+r'/r_0)^2}(\psi_D(x))\subset
        \psi_D(B_{r'}(x)) \subset
        B_{(1+\delta)^3 r'/(1-r'/r_0)^2}(\psi_D(x))\;.
	\end{align*}
    The result follows.
\end{proof}

Finally, we also state the result in the case of an interior point. We omit the
proof since it is identical to the one in the boundary case except that one does
not not need to define the Schwarz reflection in this case.

\begin{lemma}
    \label{lem:interior-koebe}
    Suppose that $D$ is a simply connected domain with $B_{r_0}(z)\subset D$.
    Let $\psi_D\colon D\to \D$ be a conformal transformation, fix $\delta\in
	(0,1)$ and suppose that $\lvert\psi_D(w)-w\rvert\le \delta^2 r_0/2$ for $w\in
    B_{r_0}(z)$. Then for all $r'\le r_0\delta$, we have $B_{r'(1-\delta)^5
    }(\psi_D(z))\subset \psi_D(B_{r'}(z)) \subset
    B_{r'/(1-\delta)^5}(\psi_D(z))$.
\end{lemma}

\bibliographystyle{hmralphaabbrv}
\bibliography{fuzzy-potts}

\end{document}